\documentclass[12pt]{article}
\usepackage{amsthm,amsfonts,amssymb,amsmath,enumerate,enumitem}
\usepackage{cite}
\usepackage{epsfig}
\usepackage{url}
\usepackage{xcolor,tikz}
\usetikzlibrary{matrix}
\usepackage{multicol}
\usepackage{authblk,fullpage}

\numberwithin{equation}{section}

\setlist{leftmargin=3\parindent,labelindent=3\parindent}
\setlist[enumerate]{%
  leftmargin=3\parindent,%
  align=left,%
  labelwidth=3\parindent,%
  labelsep=0pt%
}
\setlist[enumerate,1]{%
  label={\normalfont (\thesection.\arabic{equation})}, ref={\normalfont \thesection.\arabic{equation}},
  resume%
}

\newtheorem{thm}[equation]{Theorem}
\newtheorem{cor}[equation]{Corollary}
\newtheorem{lem}[equation]{Lemma}
\newtheorem{prop}[equation]{Proposition}

\newtheorem{conj}[equation]{Conjecture}

\newtheorem{prob}[equation]{Problem}

\theoremstyle{definition}

\newtheorem{const}[equation]{Construction}

\theoremstyle{remark}

\title{Extremal Bounds for Three-Neighbour Bootstrap Percolation in Dimensions Two and Three}
\author{Peter J. Dukes\thanks{Research supported by NSERC Discovery Grant RGPIN-312595-2017.}}
\author{Jonathan A. Noel\thanks{Research supported by NSERC Discovery Grant RGPIN-2021-02460, NSERC Early Career Supplement DGECR-2021-00024 and a Start-Up Grant from the University of Victoria.}} 
\author{Abel E. Romer}

\affil{\normalsize{Department of Mathematics and Statistics, University of Victoria, Victoria, B.C., Canada.}}
\affil{\texttt{\{dukes,noelj,aeromer\}@uvic.ca}}

\begin{document}

\maketitle

\begin{abstract}
For $r\geq1$, the \emph{$r$-neighbour bootstrap process} in a graph $G$ starts with a set of infected vertices and, in each time step, every vertex with at least $r$ infected neighbours becomes infected. The initial infection \emph{percolates} if every vertex of $G$ is eventually infected. We exactly determine the minimum cardinality of a set that percolates for the $3$-neighbour bootstrap process when $G$ is a $3$-dimensional grid with minimum side-length at least $11$. We also characterize the integers $a$ and $b$ for which there is a set of cardinality $\frac{ab+a+b}{3}$ that percolates for the $3$-neighbour bootstrap process in the $a\times b$ grid; this solves a problem raised by Benevides, Bermond, Lesfari and Nisse~[HAL Research Report 03161419v4, 2021]. 
\end{abstract}

\section{Introduction}

Bootstrap percolation is a cellular automaton introduced in the late 1970s as a model of the dynamics of ferromagnetism~\cite{ChalupaLeathReich79}. For $r\geq0$, the \emph{$r$-neighbour bootstrap process} (or, for brevity, the \emph{$r$-neighbour process}) is described as follows. Given a graph $G$ and an initial set $A_0$ of \emph{infected} vertices (non-infected vertices are \emph{healthy}), at each time step, a healthy vertex becomes infected if it has at least $r$ infected neighbours. That is, for $t\geq1$, the set of vertices infected after $t$ time steps is precisely
\[A_t:=A_{t-1}\cup\{v\in V(G): |N_G(v)\cap A_{t-1}|\geq r\}\]
where $N_G(v)$ denotes the neighbourhood of $v$ in $G$. The \emph{closure} of $A_0$ with respect to the $r$-neighbour process in $G$ is $[A_0]_{r,G}:=\bigcup_{t=0}^\infty A_t$; we simply write $[A_0]$ when $r$ and $G$ are clear from context. We say that $A_0$ \emph{percolates} under the $r$-neighbour process in $G$ if $[A_0]=V(G)$. 

A natural extremal problem is to determine the cardinality of the smallest set of vertices which percolates under the $r$-neighbour process in a graph $G$, which is denoted by $m(G;r)$. Due to its origins in statistical physics, a central focus in the study of bootstrap percolation has been on the case that $G$ is a multidimensional grid~\cite{Holroyd06,Pete97,PeteMSc,ShapiroStephens91,Benevides+21+,HambardzumyanHatamiQian20,CerfCirillo99,CerfManzo02,Bollobas+17,Balister+16,GravnerHolroydSivakoff21,GravnerHolroyd08,HartarskyMorris19,Uzzell19,BollobasSmithUzzell15,AizenmanLebowitz88,vanEnter87,BaloghPete98,BenevidesPrzykucki15,PrzykuckiShelton20,MorrisonNoel18,BaloghBollobasMorris09,BaloghBollobasMorris09-maj,BaloghBollobasMorris10,BaloghBollobas06,Balogh+12,Holroyd03}. For $n\in\mathbb{N}$, let $[n]:=\{1,2,\dots,n\}$. Let $\prod_{i=1}^d[a_i]$ denote the \emph{$d$-dimensional grid} with vertex set $[a_1]\times [a_2]\times\cdots\times[a_d]$, where two vertices $u=(u_1,\dots,u_d)$ and $v=(v_1,\dots,v_d)$ are adjacent if there is an index $i\in [d]$ such that $|u_i-v_i|=1$ and $u_j=v_j$ for all $j\neq i$. We use the notation $\prod_{i=1}^d[a_i]$ to refer both to the $a_1\times a_2\times\cdots \times a_d$ grid graph and to its vertex set. The \emph{direction} of an edge $uv$ of a $d$-dimensional grid is defined to be the unique index $1\leq i\leq d$ such that $u_i$ differs from $v_i$. For convenience, define
\[m(a_1,a_2,\dots,a_d;r):= m\left(\prod_{i=1}^d[a_i];r\right).\]
A beautiful exercise (featured in, e.g.,~\cite[Problem~34]{Bollobas06}) is to show that
\begin{equation}\label{eq:nxn}m(n,n;2)=n\end{equation}
for all $n\geq1$. The standard proof of the lower estimate $m(n,n;2)\geq n$ generalizes to the following bound in any dimension $d\geq1$:
\begin{equation}\label{eq:n^d-1}m([n]^d;d)\geq n^{d-1};\end{equation}
see~\cite{PeteMSc} or Proposition~\ref{prop:SA} in the next section. The fact that the matching upper bound to \eqref{eq:n^d-1} holds has been part of the folklore of bootstrap percolation since the work of Pete~\cite{PeteMSc}, if not earlier. While it is not hard to recursively describe a set of cardinality $n^{d-1}$ that is likely to percolate, verifying this fact is technical. To our knowledge, the first published proof appears in a puzzle book of Winkler~\cite[pp.~97--99]{Winkler07}, where it is attributed to Schulman.\footnote{Another proof was recently published by Przykucki and Shelton~\cite{PrzykuckiShelton20}.}

\begin{thm}[Lower bound~\cite{PeteMSc}. Upper bound~\cite{Winkler07}]
\label{thm:nnn}
For any $n,d\geq1$,
\[m([n]^d;d)= n^{d-1}.\] 
\end{thm}

Another natural generalization of \eqref{eq:nxn} was obtained by Balogh and Bollob\'as~\cite{BaloghBollobas06} and applied as part of the proof of a probabilistic result for bootstrap percolation in the hypercube:
\begin{equation}\label{eq:r=2}m(a_1,a_2,\dots,a_d;2) = \left\lceil \frac{\sum_{i=1}^d(a_i-1)}{2}\right\rceil+1\end{equation}
for all $d\geq1$ and $a_1,\dots,a_d\geq1$. 

For $r\notin \{2,d\}$, tight bounds on $m(a_1,\dots,a_d;r)$ are harder to come by. The following rough estimate for fixed $d,r$ and $a_1=\cdots=a_d=n$ was proved by Pete~\cite{PeteMSc} see~\cite[Section~6]{BaloghPete98}):
\[m([n]^d;r)=\begin{cases} \Theta(n^{r-1})&\text{if }1\leq r\leq d,\\ \Theta(n^d) &\text{otherwise.}\end{cases}\]
For most choices of parameters, the best known lower bound is given by~\cite[Theorem~4.2 and (1.5)]{MorrisonNoel18} (see~\cite{HambardzumyanHatamiQian20} for an alternative proof and several extensions). In~\cite{MorrisonNoel18}, it was proved that this lower bound is tight in the case that $r=3$, $d\geq 3$ and $a_1=a_2=\cdots=a_d=2$:
\begin{equation}
\label{eq:hypercube}
m([2]^d;3)= \left\lceil\frac{d(d+3)}{6}\right\rceil.
\end{equation}

In this paper, we focus on the problems of computing $m(a_1,a_2;3)$ and $m(a_1,a_2,a_3;3)$. Our first result is that the standard lower bound on $m(a_1,a_2,a_3;3)$ (see Proposition~\ref{prop:SA}) is tight, provided that the side-lengths of the grid are sufficiently large. 

\begin{thm}
\label{thm:largeEnough}
If $a_1,a_2,a_3\geq 11$, then
\[m(a_1,a_2,a_3;3)= \left\lceil\frac{a_1a_2 + a_1a_3+a_2a_3}{3}\right\rceil.\]
\end{thm}

We conjecture that the hypothesis $a_1,a_2,a_3\geq11$ in Theorem~\ref{thm:largeEnough} can be relaxed to $a_1,a_2,a_3\geq3$; see Conjecture~\ref{conj:a1a2a3>=3}. Briefly, the proof of Theorem~\ref{thm:largeEnough} involves establishing several infinite families of optimal constructions in restricted cases and assembling them recursively using percolating sets for the so-called ``modified bootstrap process'' as a template.

The problem of determining $m(n,n;3)$ for $n\equiv 2\bmod 6$ appears as a puzzle in Bollob\'as' recreational mathematics book~\cite[Problem~65]{Bollobas06}. Recently, Benevides, Bermond, Lesfari and Nisse~\cite{Benevides+21+} made progress in determining the exact value of $m(n,n;3)$ for a wider range of $n$. In particular, they showed that
\[m(n,n;3)=\left\lceil\frac{n^2+2n+4}{3}\right\rceil\]
if $n$ is even and that
\begin{equation}\label{eq:oddn}\left\lceil\frac{n^2+2n}{3}\right\rceil\leq m(n,n;3)\leq \left\lceil\frac{n^2+2n}{3}\right\rceil+1\end{equation}
if $n$ is odd. They also showed that, in the odd case, the upper bound is tight for $n\in\{9,13\}$ and the lower bound is attained when $n=5\bmod 6$ or $n$ is of the form $2^k-1$ for $k\geq1$. The construction for $n$ of the form $2^k-1$ is particularly natural. For $k=1$, we have $n=1$ and the unique percolating set is $A_0=\{(1,1)\}$. Now, for $k\geq2$, deleting the middle row and column from the $(2^k-1)\times (2^k-1)$ grid leaves behind four $(2^{k-1}-1)\times(2^{k-1}-1)$ subgrids. Let $A_0$ be an infection in the the $(2^k-1)\times (2^k-1)$ grid obtained from taking an optimal percolating set within each of these four subgrids, plus the vertex in the centre of the grid. It is not hard to show that this infection percolates under the $3$-neighbour process and that it it has precisely $\frac{n^2+2n}{3}$ elements. See Figure~\ref{fig:Purina} for an illustration. In all of the diagrams in this paper, vertices of the grid graph are represented by cells, and pairs of cells which share a boundary edge are adjacent. The point $(i,j)$ is represented by the cell on the $i$th row (from the top) and $j$th column (from the left). Infected cells are shown in grey; these are typically the initially infected cells in $A_0$, though later it is also convenient to use grey on sets which we know will become infected later in the process.

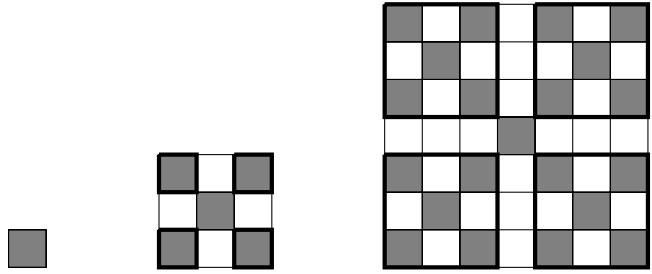
\begin{figure}[htbp]
\begin{center}
\begin{tikzpicture}
\begin{scope}[xshift=0.0cm,yshift=0cm]
\draw[step=0.5cm,color=black,fill=gray] (-0.001,-0.001) grid (0.501,0.501) (0.0,0.0) rectangle (0.5,0.5);
\end{scope}

\begin{scope}[xshift=2.0cm,yshift=0.0cm]
\draw[step=0.5cm,color=black,fill=gray] (-0.001,-0.001) grid (1.501,1.501) (0.0,0.0) rectangle (0.5,0.5)(1.0,0.0) rectangle (1.5,0.5)(0.5,0.5) rectangle (1.0,1.0)(0.0,1.0) rectangle (0.5,1.5)(1.0,1.0) rectangle (1.5,1.5);
\node at (0.75,0.25) {};\node at (0.25,0.75) {};\node at (1.25,0.75) {};\node at (0.75,1.25) {};

\draw[ultra thick](0,1.5)--(0,1)--(0.5,1)--(0.5,1.5)--(0,1.5);
\draw[ultra thick](0,0.5)--(0,0)--(0.5,0)--(0.5,0.5)--(0,0.5);
\draw[ultra thick](1,1.5)--(1,1)--(1.5,1)--(1.5,1.5)--(1,1.5);
\draw[ultra thick](1,0.5)--(1,0)--(1.5,0)--(1.5,0.5)--(1,0.5);
\end{scope}

\begin{scope}[xshift=5.0cm,yshift=0.0cm]
\draw[step=0.5cm,color=black,fill=gray] (-0.001,-0.001) grid (3.501,3.501) (0.0,0.0) rectangle (0.5,0.5)(1.0,0.0) rectangle (1.5,0.5)(2.0,0.0) rectangle (2.5,0.5)(3.0,0.0) rectangle (3.5,0.5)(0.5,0.5) rectangle (1.0,1.0)(2.5,0.5) rectangle (3.0,1.0)(0.0,1.0) rectangle (0.5,1.5)(1.0,1.0) rectangle (1.5,1.5)(2.0,1.0) rectangle (2.5,1.5)(3.0,1.0) rectangle (3.5,1.5)(1.5,1.5) rectangle (2.0,2.0)(0.0,2.0) rectangle (0.5,2.5)(1.0,2.0) rectangle (1.5,2.5)(2.0,2.0) rectangle (2.5,2.5)(3.0,2.0) rectangle (3.5,2.5)(0.5,2.5) rectangle (1.0,3.0)(2.5,2.5) rectangle (3.0,3.0)(0.0,3.0) rectangle (0.5,3.5)(1.0,3.0) rectangle (1.5,3.5)(2.0,3.0) rectangle (2.5,3.5)(3.0,3.0) rectangle (3.5,3.5);
\node at (0.75,0.25) {};\node at (1.75,0.25) {};\node at (2.75,0.25) {};\node at (0.25,0.75) {};\node at (1.25,0.75) {};\node at (1.75,0.75) {};\node at (2.25,0.75) {};\node at (3.25,0.75) {};\node at (0.75,1.25) {};\node at (1.75,1.25) {};\node at (2.75,1.25) {};\node at (0.25,1.75) {};\node at (0.75,1.75) {};\node at (1.25,1.75) {};\node at (2.25,1.75) {};\node at (2.75,1.75) {};\node at (3.25,1.75) {};\node at (0.75,2.25) {};\node at (1.75,2.25) {};\node at (2.75,2.25) {};\node at (0.25,2.75) {};\node at (1.25,2.75) {};\node at (1.75,2.75) {};\node at (2.25,2.75) {};\node at (3.25,2.75) {};\node at (0.75,3.25) {};\node at (1.75,3.25) {};\node at (2.75,3.25) {};

\draw[ultra thick](0,3.5)--(0,2)--(1.5,2)--(1.5,3.5)--(0,3.5);
\draw[ultra thick](0,1.5)--(0,0)--(1.5,0)--(1.5,1.5)--(0,1.5);
\draw[ultra thick](2,3.5)--(2,2)--(3.5,2)--(3.5,3.5)--(2,3.5);
\draw[ultra thick](2,1.5)--(2,0)--(3.5,0)--(3.5,1.5)--(2,1.5);
\end{scope}
\end{tikzpicture}
\end{center}
    \caption{Recursive constructions of optimal percolating sets for the $3$-neighbour process in the $n\times n$ grid, where $n=2^k-1$ and $k\in\{1,2,3\}$. For $k\in\{2,3\}$, the regions surrounded by bold lines indicate four copies of the previous construction. }
    \label{fig:Purina}
\end{figure}

Thus, the results of~\cite{Benevides+21+} nearly determined $m(n,n;3)$ in general. The only remaining problem, proposed in~\cite[Section~6]{Benevides+21+}, was to determine whether $m(n,n;3)$ is equal to the lower bound or the upper bound of \eqref{eq:oddn} when $n$ is odd and not congruent to $5\bmod 6$ or of the form $2^k-1$ for $k\geq1$. We complete the picture by proving that the upper bound is the truth in all such cases. This is a corollary of the following result.

\begin{thm}
\label{thm:thickness1}
Suppose that $a,b\geq1$ such that
\[m(a,b;3) = \frac{ab + a + b}{3}.\]
Then there exists $k\geq1$ such that $a=b=2^k-1$. 
\end{thm} 

\begin{cor}
\label{cor:sqGrids}
For all $n\geq1$,
\[m(n,n;3)=\begin{cases}\left\lceil\frac{n^2+2n+4}{3}\right\rceil &\text{if }n\text{ is even},\\
\frac{n^2+2n}{3} &\text{if }n=2^k-1\text{ for some }k\geq1,\\
\frac{n^2+2n+1}{3} &\text{if }n\equiv 5\bmod 6,\\
\frac{n^2+2n}{3}+1& \text{otherwise}.\end{cases}\]
\end{cor}

\begin{proof}
The first three cases follow from~\cite[Theorem~1]{Benevides+21+} and the fact that, if $n\equiv 2\bmod3$, then $n^2+2n\equiv 2\bmod 3$.

Suppose that $n$ is odd and that it is congruent to $0$ or $1$ modulo $3$. Then $n^2+2n$ is divisible by three. The bound $m(n,n;3)\leq \frac{n^2+2n}{3}+1$ is also shown in \cite[Theorem~1]{Benevides+21+}.  If $n$ is not of the form $n=2^k-1$, then $m(n,n;3)>\frac{n^2+2n}{3}$ follows from the lower bound in \eqref{eq:oddn} and Theorem~\ref{thm:thickness1}. This completes the proof.
\end{proof}

Let us pause for some additional historical context. As alluded to above, bootstrap percolation is motivated in part by connections to questions originating in statistical physics. From this perspective, the most natural problem is to locate and analyze the ``phase transition'' of the model; i.e. to determine the density at which a random infection becomes likely to percolate. More concretely, define the \emph{critical probability} of the $r$-neighbour process in $G$, denoted $p_c(G;r)$, to be the infimal value of $p\in (0,1)$ such that a random set of vertices of $G$ sampled from the binomial distribution with parameter $p$ percolates with probability at least $1/2$. As it turns out, \eqref{eq:nxn} can be combined with a few elementary observations to obtain a short proof of a classical result of Aizenman and Lebowitz~\cite{AizenmanLebowitz88} that $p_c([n]^2;2) = \Theta\left(1/\log(n)\right)$.
In a groundbreaking paper of Holroyd~\cite{Holroyd03}, this estimate was sharpened to 
\begin{equation}\label{eq:pcn2}p_c([n]^2;2) = (1+o(1))\frac{\pi^2}{18\log(n)}\end{equation}
where $\log$ denotes the natural (base $e$) logarithm. Surprisingly, this bound is more than a factor of two larger than the value of $\frac{0.245\pm 0.015}{\log(n)}$ predicted in the physics literature based on simulations on grids of size up to $28,000\times 28,000$~\cite{AdlerStaufferAharony89}. This discrepancy is explained, in part, by the fact that the second order asymptotic term of \eqref{eq:pcn2} is a negative term of order $\log(n)^{-3/2}$; this was recently proved in an impressive paper of Hartarsky and Morris~\cite{HartarskyMorris19}, improving on earlier estimates in~\cite{GravnerHolroyd08,GravnerHolroydMorris12}. This second order term competes admirably with the main $\log(n)^{-1}$ term until $n$ is extremely large. The breakthrough of Holroyd~\cite{Holroyd03} ignited a series of tight results on the critical probability in grids of higher dimensions with different values of $r$; see~\cite{Holroyd06,BaloghBollobasMorris09-maj,BaloghBollobas06,BaloghBollobasMorris10,BaloghBollobasMorris09,Balogh+12}. Most notably, Balogh, Bollob\'as, Duminil-Copin and Morris~\cite{Balogh+12} determined $p_c([n]^d;r)$ up to a $(1+o(1))$ factor for all fixed $2\leq r\leq d$ as $n\to \infty$, which sharpened earlier results of~\cite{CerfCirillo99,CerfManzo02}. When $r>d$, the behaviour of a random infection is very different as, in order to percolate, each $d$-dimensional cube must contain at least one infected vertex. The analysis in this case becomes similar to that of ordinary percolation; see~\cite[Section~1.4.1]{HartarskyPhD} for more background. 

In the next section, we prove the standard lower bound on $m(a_1,\dots,a_d;d)$ and provide several key definitions. In Section~\ref{sec:2D}, we present the proof of Theorem~\ref{thm:thickness1}. Along the way, we build up key facts about percolating sets for the $3$-neighbour process in $2$-dimensional grids which will be used in Section~\ref{sec:shapes} to  construction percolating sets in induced subgraphs of $2$-dimensional grids. In Section~\ref{sec:folding}, we show how these percolating sets in the $2$-dimensional setting can be  ``folded up'' to yield optimal percolating sets in the $a_1\times a_2\times a_3$ grid for several infinite families of side-lengths $(a_1,a_2,a_3)$. In Section~\ref{sec:recursion}, we feed these infinite families, and several sporadic examples from Appendix~\ref{app:sporadic}, into a recursive lemma to complete the proof of Theorem~\ref{thm:largeEnough}. We conclude the paper in Section~\ref{sec:concl} by making a few final observations and stating several open problems.

\section{Preliminaries}
\label{sec:prelims}

We begin by proving a simple lemma which bounds the size of the smallest percolating set for the $r$-neighbour process in a graph $G$ in terms of its order and size and deriving a lower bound on $m(a_1,\dots,a_d;d)$ as a corollary. Given a graph $G$ and a set $A\subseteq V(G)$, let $e(A)$ be the number of edges of $G$ with both endpoints in $A$. We write the neighbourhood $N_G(v)$ of a vertex $v$ of $G$ simply as $N(v)$ when the graph $G$ clear from context and let $d(v):=|N(v)|$ be the \emph{degree} of $v$. 

\begin{lem}
\label{lem:generalperim}
Let $G$ be a graph and $r\geq1$. If $A_0\subseteq V(G)$ percolates with respect to the $r$-neighbour process, then
\[|A_0|\geq |V(G)|+ \left\lceil \frac{e(A_0) -|E(G)|}{r}\right\rceil.\]
\end{lem}

\begin{proof}
Let $n:=|V(G)|$ and $k:=|A_0|$. Assuming that $A_0$ percolates in $G$, we can let $v_1,\dots,v_n$ be an ordering of the vertices of $G$ such that $\{v_1,\dots,v_k\}=A_0$ and, for each $k+1\leq i\leq n$, the vertex $v_i$ has at least $r$ neighbours in $\{v_1,\dots,v_{i-1}\}$. Then
\[|E(G)|= \sum_{i=1}^n|N(v_i)\cap \{v_1,\dots,v_{i-1}\}|\]
\[=\sum_{i=1}^k|N(v_i)\cap \{v_1,\dots,v_{i-1}\}| + \sum_{i=k+1}^n|N(v_i)\cap \{v_1,\dots,v_{i-1}\}|\geq e(A_0) + r(|V(G)|-|A_0|).\]
The result follows by solving for $|A_0|$ in the above inequality. 
\end{proof}

Next, we use Lemma~\ref{lem:generalperim} to obtain a lower bound on the cardinality of a percolating set in the $a_1\times a_2\times\cdots \times a_d$ grid in terms of the number of points on its boundary. In the case $d=2$ or $d=3$, this corresponds to a lower bound in terms of the ``perimeter'' or ``surface area'' of the grid, respectively. This bound is not new. Variants of it appear in~\cite{Pete97,PeteMSc,Bollobas06,BaloghPete98}, for example, and several generalizations are proved in~\cite{HambardzumyanHatamiQian20,MorrisonNoel18}. We provide a proof here for the sake of completeness. 

First, we need a few definitions. For $j\in [d]$ and $k\in [a_j]$, let 
\[L_{j,k}:=[a_1]\times \cdots\times[ a_{j-1}]\times\{k\}\times[a_{j+1}]\times\cdots\times [a_d].\]
We call $L_{j,k}$ the \emph{$k$th level in direction $j$} or, simply, a \emph{level} for short. The sets $L_{j,1}$ and $L_{j,a_j}$ are called the \emph{faces in direction $j$} or just \emph{faces} for short. A vertex $v$ contained in a face is called a \emph{boundary} vertex. Say that a vertex $v$ is a \emph{corner} if it is contained in a face in every direction $j\in[d]$. 

\begin{prop}
\label{prop:SA}
Let $d\geq1$ and $a_1,a_2,\dots,a_d\geq1$. If $A_0\subseteq \prod_{i=1}^d[a_i]$ percolates with respect to the $d$-neighbour process in $\prod_{i=1}^d[a_i]$, then
\[|A_0|\geq \left\lceil\frac{e(A_0) + \sum_{j=1}^d\prod_{i\neq j}^da_i}{d}\right\rceil.\]
In particular,
\[m(a_1,\dots,a_d;d) \geq\left\lceil\frac{\sum_{j=1}^d\prod_{i\neq j}^da_i}{d}\right\rceil.\]
\end{prop}

\begin{proof}
Let $G=\prod_{i=1}^d[a_i]$. By Lemma~\ref{lem:generalperim}, we have that
\[|A_0|\geq |V(G)|+ \left\lceil \frac{e(A_0) -|E(G)|}{d}\right\rceil=\left\lceil \frac{2d|V(G)| + 2e(A_0) -2|E(G)|}{2d}\right\rceil.\]
\[=\left\lceil \frac{2e(A_0) + \sum_{v\in V(G)}(2d-d(v))}{2d}\right\rceil.\]
For each vertex $v\in V(G)$, the quantity $2d-d(v)$ is simply equal to the number of faces that $v$ belongs to (where, if $a_j=1$ for some $j\in [d]$, then we count the face $L_{j,1}=L_{j,a_j}$ twice). Thus, by double-counting, we see that $\sum_{v\in V(G)}(2d-d(v))$ is equal to the sum of the cardinalities of all of the faces of the grid; clearly, the two faces in direction $j$ each have cardinality $\prod_{i\neq j}^da_i$. The result now follows by substituting $2\sum_{j=1}^d\prod_{i\neq j}^da_i$ in for $\sum_{v\in V(G)}(2d-d(v))$ in the above expression.
\end{proof}

For $\ell\in\{0,1,2\}$, say that a triple $(a_1,a_2,a_3)$ of positive integers is \emph{class $\ell$} if 
\[a_1a_2 + a_1a_3+a_2a_3\equiv \ell\bmod 3.\]
A useful fact to keep in mind is that $(a_1,a_2,a_3)$ is class $0$ if and only if either at least two of $a_1,a_2$ and $a_3$ are multiples of three, or $a_1\equiv a_2\equiv a_3\bmod 3$. We say that $(a_1,a_2,a_3)$ is \emph{optimal} if $m(a_1,a_2,a_3;3)= \left\lceil\frac{a_1a_2 + a_1a_3+a_2a_3}{3}\right\rceil$; that is, $(a_1,a_2,a_3)$ is optimal if the lower bound on $m(a_1,a_2,a_3;3)$ implied by Proposition~\ref{prop:SA} is tight. Moreover, we say that $(a_1,a_2,a_3)$ is \emph{perfect} if it is optimal and class 0. Clearly, permuting the coordinates of a triple $(a_1,a_2,a_3)$ preserves  class and optimality. In this language, Theorem~\ref{thm:largeEnough} says that $(a_1,a_2,a_3)$ is optimal whenever $a_1,a_2,a_3\geq11$ and Theorem~\ref{thm:thickness1} says that, if $(a_1,a_2,1)$ is perfect, then there exists an integer $k\geq1$ such that $a_1=a_2=2^k-1$.

\section{Three Neighbours in Two Dimensions}
\label{sec:2D}

Our focus in this section is on proving Theorem~\ref{thm:thickness1}. As we shall see, percolating sets for the $3$-neighbour process in $[a_1]\times [a_2]$ must have a very special structure; specifically, the complement of such a set must induce a forest in which each component contains at most one boundary vertex. This is a special case of a more general phenomenon, as the next lemma illustrates. 

\begin{lem}
\label{lem:immuneRegions}
For $r\geq1$, let $G$ be a graph of maximum degree at most $r+1$ and let $A_0\subseteq V(G)$. Then $A_0$ percolates under the $r$-neighbour process in $G$ if and only if 
\begin{enumerate}
\stepcounter{equation}
    \item\label{eq:morethanoneface} $A_0$ contains every vertex of $G$ of degree less than $r$ and 
\stepcounter{equation}
    \item\label{eq:atmostoneboundary} every component of $G\setminus A_0$ is a tree containing at most one vertex of degree $r$.
\end{enumerate}
\end{lem}

\begin{proof}
First, we prove the ``only if'' direction. Observe that any healthy vertex with at most $r-1$ neighbours can never become infected by the $r$-neighbour process; therefore, if $A_0$ percolates, then every such vertex must be in $A_0$. Now, suppose that there is a component $X$ of $G\setminus A_0$ containing either a cycle or two vertices of degree $r$. In either case, we can let $W$ be a walk in $X$ such that either (a) $W$ is a cycle or (b) $W$ is a path between two distinct vertices of degree $r$. Each internal vertex of $W$ has at least two neighbours in $W$ and therefore has at most $r-1$ neighbours outside of $W$. Likewise, the starting and ending vertices of $W$ also have at most $r-1$ neighbours outside of $W$, regardless of whether (a) or (b) holds. Thus, by considering the first vertex of $W$ to become infected and deriving a contradiction, we see that no vertex of $W$ is infected by the $r$-neighbour process starting with $A_0$, and so $A_0$ does not percolate.

Now, suppose that \eqref{eq:morethanoneface} and \eqref{eq:atmostoneboundary} hold. We show that $A_0$ percolates by induction on $|V(G)\setminus A_0|$, where the base case $|V(G)\setminus A_0|=0$ is trivial. Let $X$ be any component of $G\setminus A_0$. If $|X|=1$, then, by \eqref{eq:morethanoneface}, the unique $v\in X$ has degree at least $r$ and so it becomes infected after one step; we are therefore done by applying induction to $A_0\cup\{v\}$. If $|X|\geq2$, then let $u$ and $v$ be distinct leaves of $G[X]$. By \eqref{eq:atmostoneboundary}, we can assume, without loss of generality, that $v$ has degree $r+1$. Since $v$ is a leaf of $G[X]$, it has at least $r$ neighbours in $A_0$ and so it is infected after one step; we are done by applying induction to $A_0\cup\{v\}$ once again.
\end{proof}

The following corollary is immediate from Lemma~\ref{lem:immuneRegions} since $2$-dimensional grids have maximum degree at most $4$. In this section, we are mainly interested in the ``only if'' direction of this corollary, but the ``if'' direction will be used several times in the next section. 

\begin{cor}
\label{cor:iff2D}
For $a_1,a_2\geq1$ a set $A_0\subseteq [a_1]\times [a_2]$ percolates under the $3$-neighbour process in $[a_1]\times [a_2]$ if and only if $A_0$ contains every corner and every component of $([a_1]\times [a_2])\setminus A_0$ is a tree containing at most one boundary vertex.
\end{cor}

The next lemma provides some strong structural restrictions on percolating sets of cardinality $\frac{a_1a_2+a_1+a_2}{3}$ in $[a_1]\times [a_2]$. Say that an element $(v_1,v_2)\in \mathbb{Z}^2$ is \emph{even} or \emph{odd} depending on the parity of $v_1+v_2$. 

\begin{lem}
\label{lem:2DPerfect}
Let $a_1,a_2\geq2$ and $A_0\subseteq [a_1]\times [a_2]$ such that $|A_0|=\frac{a_1a_2+a_1+a_2}{3}$. If $A_0$ percolates under the $3$-neighbour process in $[a_1]\times [a_2]$, then
\begin{enumerate}
\stepcounter{equation}
    \item \label{eq:indep} $A_0$ is an independent set,
\stepcounter{equation}
    \item \label{eq:boundaryAlternates} every even boundary vertex is contained in $A_0$,
\stepcounter{equation}    
\item \label{eq:odd} $a_1$ and $a_2$ are both odd,
\stepcounter{equation}
    \item \label{eq:exactlyOneBoundary} every component of $([a_1]\times [a_2])\setminus A_0$ is a tree containing exactly one boundary vertex and
\stepcounter{equation}
    \item \label{eq:allEven} every $v\in A_0$ is even.
\end{enumerate}
\end{lem}

\begin{proof}
We will establish each of the properties \eqref{eq:indep}--\eqref{eq:allEven} sequentially, where, in the proof of each of these statements, we may assume that all of the earlier statements hold. First, if $A_0$ is not an independent set, then $e(A_0)\geq1$ and so applying Proposition~\ref{prop:SA}  to $[a_1]\times[a_2]\times[1]$ which is isomorphic to $[a_1]\times[a_2]$, implies that $|A_0|\geq \frac{a_1a_2+a_1+a_2+1}{3}$, which contradicts our assumption on $|A_0|$. So, \eqref{eq:indep} holds. 

By Corollary~\ref{cor:iff2D}, we know that no component of $([a_1]\times[a_2])\setminus A_0$ can contain two boundary vertices. In particular, if $u$ and $v$ are consecutive boundary vertices, then $A_0$ must contain at least one of $u$ or $v$. However, by \eqref{eq:indep}, $A_0$ contains at most one of $u$ or $v$. Putting this together, we get that $A_0$ contains either every even boundary vertex or every odd boundary vertex. By Corollary~\ref{cor:iff2D}, we know that $A_0$ contains every corner vertex; in particular, it contains $(1,1)$, which is even. Thus, \eqref{eq:boundaryAlternates} holds. Now, since \eqref{eq:indep} and \eqref{eq:boundaryAlternates} hold and $A_0$ contains every corner vertex, we must have that $(a_1,1)$ and $(1,a_2)$ are even, and so \eqref{eq:odd} holds. 

Denote the graph $([a_1]\times [a_2])\setminus A_0$ by $H$. Our next goal is to show that the number of components of $H$ is equal to $a_1+a_2-2$ and use this to deduce \eqref{eq:exactlyOneBoundary}. By Corollary~\ref{cor:iff2D}, $H$ is a forest. Therefore, the number of components of $H$ is precisely 
\begin{equation}
\label{eq:V-E}
\begin{gathered}
|V(H)| - |E(H)| = a_1a_2 - |A_0| - |E(H)|= a_1a_2 - \left(\frac{a_1a_2+a_1+a_2}{3}\right) - |E(H)|\\
=\frac{2a_1a_2 - a_1 - a_2}{3} - |E(H)|.
\end{gathered}
\end{equation}
Let us now determine $|E(H)|$. By definition of $H$, $|E(H)|$ is equal to the number of edges in the $a_1\times a_2$ grid, which is $2a_1a_2-a_1-a_2$, minus the number of edges which have at least one endpoint in $A_0$. By \eqref{eq:indep}, we know that $A_0$ is an independent set, and so each edge has at most one endpoint in $A_0$. Therefore,
\[|E(H)|=2a_1a_2-a_1-a_2 - \sum_{v\in A_0}d(v).\]
We know by \eqref{eq:boundaryAlternates} and \eqref{eq:indep} that $A_0$ contains all of the even boundary vertices and all of the other vertices of $A_0$ have degree four. For illustration purposes, let us write $A_0=C\sqcup B\sqcup I$ where $C$ is the set of corner vertices, $B$ is the set of even boundary vertices that are not corner vertices and $I$ is the set of vertices of $A_0$ that are not on the boundary. Then
\[\sum_{v\in A_0}d(v) = 2|C|+3|B|+4|I|.\]
Clearly, $|C|=4$ and its not hard to see that $|B|=a_1+a_2-6$. Therefore,
\[|I|=|A_0| - \left(a_1+a_2-2\right) = \frac{a_1a_2-2a_1-2a_2}{3}+2.\]
Putting all of this together, we get that $|E(H)|$ is equal to
\[2a_1a_2-a_1-a_2-8-3(a_1+a_2-6)-4\left(\frac{a_1a_2-2a_1-2a_2}{3}+2\right)=\frac{2a_1a_2 - 4a_1 - 4a_2 + 6}{3}.\]
So, plugging this into \eqref{eq:V-E}, we get that the number of components of $H$ is
\[\frac{2a_1a_2 - a_1 - a_2}{3} -\frac{2a_1a_2 - 4a_1 - 4a_2 + 6}{3} = a_1+a_2-2.\]
So, the number of components of $H$ is equal to $a_1+a_2-2$. By \eqref{eq:indep}, \eqref{eq:boundaryAlternates} and \eqref{eq:odd}, $a_1+a_2-2$ is equal to the number of odd boundary vertices and, at the same time, every odd boundary vertex is contained in a component of $H$. Since each component of $H$ contains at most one boundary vertex by Corollary~\ref{cor:iff2D}, we see that every component of $H$ must contain exactly one odd boundary vertex; thus, \eqref{eq:exactlyOneBoundary} holds.

Finally, we prove \eqref{eq:allEven}. Let $X_0$ be the set of all odd vertices in $A_0$, suppose to the contrary that $X_0$ is non-empty, and let $X:=[X_0]$. By \eqref{eq:indep} and \eqref{eq:boundaryAlternates}, $X_0$ contains no boundary vertices, and so neither does $X$. Therefore, the \emph{neighbourhood} of $X$, i.e. the set
\[N(X):=\left(\bigcup_{x\in X}N(x)\right)\setminus X,\]
contains a cycle. Our goal is to show that $N(X)$ contains no vertex of $A_0$ which, by Corollary~\ref{cor:iff2D}, will contradict the fact that $A_0$ percolates. 

Let $\ell:=|X_0|$ and $m:=|X|$ and label the vertices of $X$ by $v_1,\dots,v_m$ so that $X_0=\{v_1,\dots,v_\ell\}$ and $v_i$ has at least $3$ neighbours in $\{v_1,\dots,v_{i-1}\}$ whenever $\ell+1\leq i\leq m$. Suppose to the contrary that there is a vertex $w\in A_0\cap N(X)$ which is adjacent to $v_i$ for some $1\leq i\leq m$. If $1\leq i\leq \ell$, then this immediately violates \eqref{eq:indep}, and so we must have $\ell+1\leq i\leq m$. Consider the set
\[Y:=A_0\cup\{v_{\ell+1},\dots,v_i\}.\]
By our choice of vertex ordering, the vertex $v_j$ for $\ell+1\leq j\leq i-1$ has at least $3$ neighbours in $A_0\cup\{v_{\ell+1},\dots,v_{j-1}\}$. Also, $v_i$ has at least $4$ neighbours in $A_0\cup\{v_{\ell+1},\dots,v_{i-1}\}$, three of which are in $X$ and the other is the vertex $w$, which is not in $X$ by assumption. Thus,
\[e(Y)\geq 3(i-\ell-1) + 4 = 3(i-\ell)+1.\]
Since $A_0$ percolates, so does $Y$. Note that, since $|A_0|=\frac{a_1a_2+a_1+a_2}{3}$, we have that $a_1a_2+a_1+a_2$ is divisible by three. So, by Proposition~\ref{prop:SA}, we get
\[|Y|\geq \left\lceil\frac{a_1a_2+a_1+a_2+3(i-\ell)+1}{3}\right\rceil= \frac{a_1a_2+a_1+a_2}{3} + i-\ell+1 = |A_0|+i-\ell+1.\]
However, $|Y|\leq |A_0|+i-\ell$, simply by definition. This contradiction completes the proof of \eqref{eq:allEven} and of the lemma. 
\end{proof}

We now present the proof of Theorem~\ref{thm:thickness1}.

\begin{proof}[Proof of Theorem~\ref{thm:thickness1}]
We proceed by induction on $a_1+a_2$. For the base case, suppose $a_1=1$. Then every vertex of $[a_1]\times[a_2]$ has degree at most two and so by Corollary~\ref{cor:iff2D}, the only set that percolates is $A_0=[a_1]\times [a_2]$. 
From $a_2 = |A_0|= \frac{a_1a_2+a_1+a_2}{3} = \frac{2a_2+1}{3}$, we immediately get that $a_2=1$.
The proof in the case $a_2=1$ is symmetric.

So, from here forward, we assume $a_1,a_2\geq 2$ and suppose that there is a set $A_0\subseteq [a_1]\times [a_2]$ such that $|A_0|=\frac{a_1a_2+a_1+a_2}{3}$ and $A_0$ percolates. By Lemma~\ref{lem:2DPerfect}, we have that $a_1$ and $a_2$ are both odd. Let $b_i=\frac{1}{2}(a_i-1)$ for $i\in\{1,2\}$. Our goal is to show that there is a percolating set $B_0$ in $[b_1]\times[b_2]$ of cardinality $\frac{b_1b_2+b_1+b_2}{3}$. Given this, we will get that $b_1=b_2=2^k-1$ for some $k\geq1$ by induction, and so $a_1=a_2=2^{k+1}-1$ and the proof will be complete.

Let us now describe the construction of $B_0$. For each $v=(v_1,v_2)\in [b_1]\times [b_2]$, we observe that the set 
\[T_v := \{2v_1-1,2v_1\}\times\{2v_2-1,2v_2\},\]
which we call the \emph{tile} corresponding to $v$, must contain either one or two elements of $A_0$. This is because $T_v$ induces a $4$-cycle in $[a_1]\times [a_2]$ and so $|A_0\cap T_v|\geq 1$ by Corollary~\ref{cor:iff2D} and $|A_0\cap T_v|\leq 2$ by \eqref{eq:indep}. We define
\[B_0:=\{v\in [b_1]\times[b_2]: |A_0\cap T_v|=2\}.\]
The sets $T_v$ for $v\in [b_1]\times [b_2]$ partition $[a_1-1]\times [a_2-1]$. By \eqref{eq:indep} and \eqref{eq:boundaryAlternates}, we have that $A_0$ contains precisely $\frac{a_1+a_2}{2}$ vertices outside of $[a_1-1]\times [a_2-1]$. Therefore,
\[|A_0| = 2|B_0| + (b_1b_2-|B_0|) + \frac{a_1+a_2}{2}.\]
Solving for $|B_0|$ yields
\[|B_0|=\frac{a_1a_2+a_1+a_2}{3} - \frac{a_1+a_2}{2} - b_1b_2\]
\[= \frac{(2b_1+1)(2b_2+1)+(2b_1+1)+(2b_2+1)}{3} -\frac{2b_1+1+2b_2+1}{2}- b_1b_2 = \frac{b_1b_2+b_1+b_2}{3}.\]

By \eqref{eq:allEven}, we know that, for $v=(v_1,v_2)$, 
\[T_v\cap A_0\subseteq \{(2v_1-1,2v_2-1),(2v_1,2v_2)\}.\]
For $v\in[b_1]\times[b_2]$, say that the tile $T_v$ is \emph{type $0$} if $T_v\cap A_0=\{(2v_1,2v_2)\}$, \emph{type $1$} if $T_v\cap A_0=\{(2v_1-1,2v_2-1)\}$ and \emph{type $2$} otherwise; see Figure~\ref{fig:types}. We claim that there are no type $0$ tiles. To this end, suppose that $T_v$ is type $0$ and, using \eqref{eq:exactlyOneBoundary}, let $w_0\cdots w_k$ be the unique path in $[a_1]\times [a_2]\setminus A_0$ from $w_0=(2v_1-1,2v_2-1)$ to a boundary vertex $w_k$. For $0\leq i\leq k-1$, let $T_i$ be the unique tile containing $w_i$; in particular, $T_0=T_v$. Note that there may be some indices $i$ such that $T_i$ and $T_{i+1}$ are the same tile. Also, $w_k$ may not be in $[a_1-1]\times[a_2-1]$ and so it is not necessarily contained in any tile. 

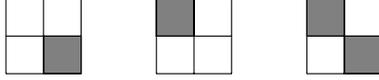
\begin{figure}[htbp]
\begin{center}
\begin{tikzpicture}
\begin{scope}[xshift=0.0cm,yshift=0.0cm]
\draw[step=0.5cm,color=black,fill=gray] (-0.001,-0.001) grid (1.001,1.001) (0.5,0.5) rectangle (1.0,0.0);
\end{scope}
\begin{scope}[xshift=2.0cm,yshift=0.0cm]
\draw[step=0.5cm,color=black,fill=gray] (-0.001,-0.001) grid (1.001,1.001) (0.5,0.5) rectangle (0.0,1.0);
\end{scope}
\begin{scope}[xshift=4.0cm,yshift=0.0cm]
\draw[step=0.5cm,color=black,fill=gray] (-0.001,-0.001) grid (1.001,1.001) (0.5,0.5) rectangle (0.0,1.0)(0.5,0.5) rectangle (1.0,0.0);
\end{scope}
\end{tikzpicture}
\end{center}
    \caption{Tiles of types 0, 1 and 2, respectively. Note that the cell in the top left corner represents a vertex of the grid whose coordinates are both odd.}
    \label{fig:types}
\end{figure}

We claim that $T_{k-1}$ cannot be type $0$. If $w_k$ is on the top or left face of the grid, then, by \eqref{eq:boundaryAlternates}, $w_{k}$ is an odd vertex of the form $(1,2\ell)$ or $(2\ell,1)$ for some $\ell$, respectively, and $w_{k-1}$ is either $(2,2\ell)$ or $(2\ell,2)$. This implies that $w_k$ is in the tile $T_{k-1}$. However, if $T_{k-1}$ is type $0$, then this gives us two consecutive boundary vertices that are not in $A_0$, contradicting \eqref{eq:boundaryAlternates}. On the other hand, if $w_k$ is on the bottom or the right, then it is an odd vertex of the form $(a_1,2\ell)$ or $(2\ell,a_2)$ for some $\ell$, respectively. However, applying \eqref{eq:boundaryAlternates} again, we see that $w_{k-1}$ would have to be  $(a_1-1,2\ell)$ or $(2\ell,a_2-1)$. Thus, both coordinates of $w_{k-1}$ are even, and so $w_{k-1}$ is the bottom right vertex of $T_{k-1}$. However, since $w_{k-1}\notin A_0$ by definition, this contradicts the assumption that $T_{k-1}$ has type 0. 

Therefore, there must exist $0\leq j\leq k-2$ such that $T_j$ is type $0$ and $T_{j+1}$ is not. In particular, the bottom right vertex of $T_j$ is in $A_0$ and the top left vertex of $T_{j+1}$ is in $A_0$. Since $w_j\in T_j\setminus A_0$ and $w_{j+1}\in T_{j+1}\setminus A_0$ are adjacent, we see that $T_{j+1}$ cannot be to the right or below $T_j$; see Figure~\ref{fig:type0}~(a). We assume that $T_{j+1}$ is to the left of $T_j$ and note that the other case is symmetric. If $T_{j+1}$ is type 1, then the two vertices on the right side of $T_{j+1}$ together with the two on the left side of $T_j$ induce a $4$-cycle in the complement of $A_0$ which contradicts Corollary~\ref{cor:iff2D}; see Figure~\ref{fig:type0}~(b). So, $T_{j+1}$ is type $2$. Now, we must have that $w_j$ is the top left vertex of $T_j$ and $w_{j+1}$ is the top right vertex of $T_{j+1}$. Note that $j+1\neq k$ because, if $w_{j+1}$ is a boundary vertex, then it must be on the top face, but then $w_j$ would also be on the top face, thereby contradicting \eqref{eq:boundaryAlternates}. The vertices to the left and below $w_{j+1}$ are both in $A_0$ and so the only possibility for $w_{j+2}$ is that it must be above $w_{j+1}$. This implies that the tile $T_{j+2}$ is of type $1$. However, we now get a $4$-cycle in the complement of $A_0$ containing the top left vertex of $T_j$, the bottom right vertex of $T_{j+2}$ and their two common neighbours in the grid (which are both odd and therefore not in $A_0$ by \eqref{eq:allEven}); see Figure~\ref{fig:type0} (c) for an illustration. This contradiction completes the proof of that all tiles are type $1$ or $2$. 

\begin{figure}[htbp]
\begin{center}
\begin{tikzpicture}
\begin{scope}[xshift=0.0cm,yshift=0.0cm]
\draw[step=0.5cm,color=black,fill=gray] (-0.001,-0.001) grid (2.001,1.001) (1.0,0.5) rectangle (1.5,1.0)(0.5,0.5) rectangle (1.0,0.0);
\draw[ultra thick](1.0,0.0)--(1.0,1.0);
\node at (0.5,-0.5) {$T_j$};
\node at (1.5,-0.5) {$T_{j+1}$};
\node at (1,-1.5) {(a)};
\end{scope}

\begin{scope}[xshift=3.0cm,yshift=0.0cm]
\draw[step=0.5cm,color=black,fill=gray] (-0.001,-0.001) grid (2.001,1.001) (0.0,0.5) rectangle (0.5,1.0)(1.5,0.5) rectangle (2.0,0.0);
\draw[ultra thick](1.0,0.0)--(1.0,1.0);
\node at (0.5,-0.5) {$T_{j+1}$};
\node at (1.5,-0.5) {$T_{j}$};
\node at (1,-1.5) {(b)};
\end{scope}

\begin{scope}[xshift=6.0cm,yshift=0.0cm]
\draw[step=0.5cm,color=black,fill=gray] (-0.001,-0.001) grid (2.001,2.001) (0.0,2.0) rectangle (0.5,1.5)(1.0,2.0) rectangle (1.5,1.5)(1.5,1.5) rectangle (2.0,1.0)(0.0,1.0) rectangle (0.5,0.5)(0.5,0.5) rectangle (1.0,0.0)(1.5,0.5) rectangle (2.0,0.0);
\draw[ultra thick](1.0,0.0)--(1.0,2.0);
\draw[ultra thick](0.0,1.0)--(2.0,1.0);
\node at (0.5,-0.5) {$T_{j+1}$};
\node at (1.5,-0.5) {$T_{j}$};
\node at (0.5,2.5) {$T_{j+2}$};
\node at (1,-1.5) {(c)};
\end{scope}
\end{tikzpicture}
\end{center}
    \caption{An illustration of three possible situations that arise in the proof that there are no tiles of type $0$: (a) $T_{j+1}$ is to the right of $T_j$, (b) $T_{j+1}$ is to the left of $T_j$ and is type 1, (c) $T_{j+1}$ is type 2. Bold lines are used to separate the vertices of distinct tiles.}
    \label{fig:type0}
\end{figure}
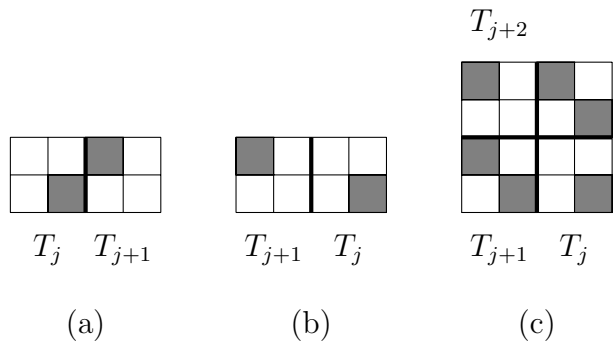

Now, finally, we show that $B_0$ percolates in $[b_1]\times [b_2]$. By Corollary~\ref{cor:iff2D}, this is equivalent to showing that every component of $[b_1]\times [b_2]\setminus B_0$ is a tree containing at most one boundary vertex. Suppose that $v_1\cdots v_m$ is a walk in $[b_1]\times [b_2]\setminus B_0$ that forms either a cycle or a path from one boundary vertex to another. By definition of $B_0$ and the fact that there are no type $0$ tiles, this means that all of the tiles $T_{v_1},\dots,T_{v_m}$ are of type $1$. Therefore, for any $1\leq i\leq m-1$, the subgraph of $[a_1]\times [a_2]$ induced by $T_{v_i}\cup T_{v_{i+1}}\setminus A_0$ is connected. So, there is a walk in $[a_1]\times [a_2]\setminus A_0$ from any vertex of $T_{v_1}$ to any vertex of $T_{v_m}$ that passes through all of the intermediate tiles. Using this, we either find a cycle or a or a path between two boundary vertices in $[a_1]\times [a_2]\setminus A_0$ which, by Corollary~\ref{cor:iff2D}, contradicts the fact that $A_0$ percolates and completes the proof of the theorem. 
\end{proof}

\section{2D Shapes with Small Percolating Sets}
\label{sec:shapes}

In contrast to the previous section, which focused on non-existence results for efficient percolating sets in two dimensions, our focus in this section will be on showing that certain induced subgraphs of $2$-dimensional grids do admit percolating sets of the smallest possible cardinality allowed by Lemma~\ref{lem:generalperim}. These constructions will be applied in the next section to construct numerous optimal percolating sets in three dimensions. We start with some rather simple constructions which form the ``building blocks'' of more involved constructions that come later. 

\begin{lem}
\label{lem:skinny}
For $a_2\geq 2$, let $s\in \{1,2\}$ such that $a_2\equiv s\bmod 2$. Then there exists a set $A_0\subseteq [2]\times[a_2]$ such that $|A_0|=a_2$ and $[A_0]=([2]\times[a_2])\setminus\{(1,1),(s,a_2)\}$. 
\end{lem}

\begin{proof}
Let $A_0$ be the set of all odd vertices of $[2]\times [a_2]$. That is, $A_0$ contains exactly one element from each column, alternating between the ``bottom'' and ``top'' vertex. See Figure~\ref{fig:skinny} for an illustration of some small cases. It is evident that every vertex in $[2]\times [a_2]$ is either in $A_0$ or has three neighbours in $A_0$, except for the vertices $(1,1)$ and $(s,a_2)$, which never become infected because they have degree 2. 
\end{proof}

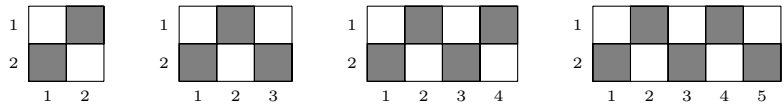
\begin{figure}[htbp]
\begin{center}
\begin{tikzpicture}
\begin{scope}[xshift=0.0cm,yshift=0.0cm]
\node at (-0.2,0.75) {\tiny $1$};
\node at (-0.2,0.25) {\tiny $2$};
\node at (0.25,-0.2) {\tiny $1$};
\node at (0.75,-0.2) {\tiny $2$};
\draw[step=0.5cm,color=black,fill=gray] (-0.001,-0.001) grid (1.001,1.001) (0.0,0.0) rectangle (0.5,0.5)(0.5,0.5) rectangle (1.0,1.0);
\end{scope}

\begin{scope}[xshift=2cm,yshift=0.0cm]
\node at (-0.2,0.75) {\tiny $1$};
\node at (-0.2,0.25) {\tiny $2$};
\node at (0.25,-0.2) {\tiny $1$};
\node at (0.75,-0.2) {\tiny $2$};
\node at (1.25,-0.2) {\tiny $3$};
\draw[step=0.5cm,color=black,fill=gray] (-0.001,-0.001) grid (1.501,1.001) (0.0,0.0) rectangle (0.5,0.5)(1.0,0.0) rectangle (1.5,0.5)(0.5,0.5) rectangle (1.0,1.0);
\end{scope}

\begin{scope}[xshift=4.5cm,yshift=0.0cm]
\node at (-0.2,0.75) {\tiny $1$};
\node at (-0.2,0.25) {\tiny $2$};
\node at (0.25,-0.2) {\tiny $1$};
\node at (0.75,-0.2) {\tiny $2$};
\node at (1.25,-0.2) {\tiny $3$};
\node at (1.75,-0.2) {\tiny $4$};
\draw[step=0.5cm,color=black,fill=gray] (-0.001,-0.001) grid (2.001,1.001) (0.0,0.0) rectangle (0.5,0.5)(1.0,0.0) rectangle (1.5,0.5)(0.5,0.5) rectangle (1.0,1.0)(1.5,0.5) rectangle (2.0,1.0);
\end{scope}

\begin{scope}[xshift=7.5cm,yshift=0.0cm]
\node at (-0.2,0.75) {\tiny $1$};
\node at (-0.2,0.25) {\tiny $2$};
\node at (0.25,-0.2) {\tiny $1$};
\node at (0.75,-0.2) {\tiny $2$};
\node at (1.25,-0.2) {\tiny $3$};
\node at (1.75,-0.2) {\tiny $4$};
\node at (2.25,-0.2) {\tiny $5$};
\draw[step=0.5cm,color=black,fill=gray] (-0.001,-0.001) grid (2.501,1.001) (0.0,0.0) rectangle (0.5,0.5)(1.0,0.0) rectangle (1.5,0.5)(2.0,0.0) rectangle (2.5,0.5)(0.5,0.5) rectangle (1.0,1.0)(1.5,0.5) rectangle (2.0,1.0);
\end{scope}
\end{tikzpicture}
\end{center}
    \caption{The sets $A_0$ in the proof of Lemma~\ref{lem:skinny} for $a_2=2,3,4$ and $5$.}
    \label{fig:skinny}
\end{figure}

\begin{lem}
\label{lem:fatter}
Let $a_1,a_2\geq 2$ such that $a_1\equiv 2\bmod 3$ and $a_2\equiv 0\bmod 2$ and let 
\[s:=\begin{cases}1&\text{if }a_1\text{ is odd},\\
a_2&\text{otherwise}.\end{cases}\] 
Then there exists $A_0\subseteq [a_1]\times [a_2]$ such that $|A_0|=\frac{a_1a_2+a_1+a_2-2}{3}$ and $[A_0]=([a_1]\times[a_2])\setminus\{(1,1),(a_1,s)\}$.
\end{lem}

\begin{proof}
Let $k\geq0$ such that $a_1=2+3k$. We proceed by induction on $k$. If $k=0$, then we are immediately done by Lemma~\ref{lem:skinny}; so, we assume $k\geq1$. By induction, there exists a set $A_0'\subseteq [a_1-3]\times [a_2]$ of cardinality $\frac{(a_1-3)a_2+(a_1-3)+a_2-2}{3}$ such that the closure of $A_0'$ contains all of $[a_1-3]\times [a_2]$ except for the point $(1,1)$ and $(a_1-3,t)$, where $t$ is the unique element of $\{1,a_2\}\setminus\{s\}$. We let $A_0$ be the subset of $[a_1]\times[a_2]$ consisting of the elements of $A_0'$ along with $(a_1-2,t)$ and a subset $A_0''$ of $\{a_1-1,a_1\}\times [a_2]$ of cardinality $a_2$ such that $[A_0'']$ consists of all of $\{a_1-1,a_1\}\times [a_2]$ except for the point $(a_1,s)$ and $(a_1-1,t)$. Then
\[|A_0|=|A_0'|+1+|A_0''| = \frac{(a_1-3)a_2+(a_1-3)+a_2-2}{3} + 1+a_2=\frac{a_1a_2+a_1+a_2-2}{3}.\]
The fact that $A_0$ has the desired closure is not hard to see; an illustration is provided in Figure~\ref{fig:fatter}. First, run the $3$-neighbour process in each of $A_0'$ and $A_0''$ separately. After doing so, the graph obtained by deleting the infected vertices consists of the corners $(1,1)$ and $(a_1,s)$ and trees (in fact, one path and two singletons) which each contain a unique boundary vertex. Thus, we are done by Lemma~\ref{lem:immuneRegions}.  
\end{proof}

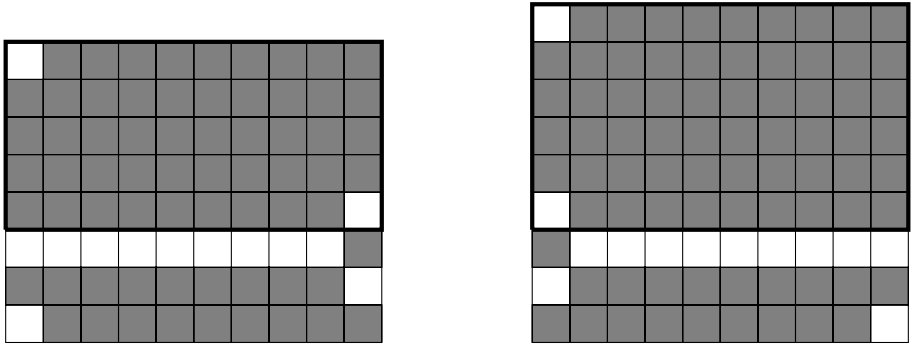
\begin{figure}[htbp]
\begin{center}
\begin{tikzpicture}
\begin{scope}[xshift=0.0cm,yshift=0.0cm]
\draw[step=0.5cm,color=black,fill=gray] (-0.001,-0.001) grid (5.001,4.001) (0.5,0.0) rectangle (1.0,0.5)(1.0,0.0) rectangle (1.5,0.5)(1.5,0.0) rectangle (2.0,0.5)(2.0,0.0) rectangle (2.5,0.5)(2.5,0.0) rectangle (3.0,0.5)(3.0,0.0) rectangle (3.5,0.5)(3.5,0.0) rectangle (4.0,0.5)(4.0,0.0) rectangle (4.5,0.5)(4.5,0.0) rectangle (5.0,0.5)(0.0,0.5) rectangle (0.5,1.0)(0.5,0.5) rectangle (1.0,1.0)(1.0,0.5) rectangle (1.5,1.0)(1.5,0.5) rectangle (2.0,1.0)(2.0,0.5) rectangle (2.5,1.0)(2.5,0.5) rectangle (3.0,1.0)(3.0,0.5) rectangle (3.5,1.0)(3.5,0.5) rectangle (4.0,1.0)(4.0,0.5) rectangle (4.5,1.0)(4.5,1.0) rectangle (5.0,1.5)(0.0,1.5) rectangle (0.5,2.0)(0.5,1.5) rectangle (1.0,2.0)(1.0,1.5) rectangle (1.5,2.0)(1.5,1.5) rectangle (2.0,2.0)(2.0,1.5) rectangle (2.5,2.0)(2.5,1.5) rectangle (3.0,2.0)(3.0,1.5) rectangle (3.5,2.0)(3.5,1.5) rectangle (4.0,2.0)(4.0,1.5) rectangle (4.5,2.0)(0.0,2.0) rectangle (0.5,2.5)(0.5,2.0) rectangle (1.0,2.5)(1.0,2.0) rectangle (1.5,2.5)(1.5,2.0) rectangle (2.0,2.5)(2.0,2.0) rectangle (2.5,2.5)(2.5,2.0) rectangle (3.0,2.5)(3.0,2.0) rectangle (3.5,2.5)(3.5,2.0) rectangle (4.0,2.5)(4.0,2.0) rectangle (4.5,2.5)(4.5,2.0) rectangle (5.0,2.5)(0.0,2.5) rectangle (0.5,3.0)(0.5,2.5) rectangle (1.0,3.0)(1.0,2.5) rectangle (1.5,3.0)(1.5,2.5) rectangle (2.0,3.0)(2.0,2.5) rectangle (2.5,3.0)(2.5,2.5) rectangle (3.0,3.0)(3.0,2.5) rectangle (3.5,3.0)(3.5,2.5) rectangle (4.0,3.0)(4.0,2.5) rectangle (4.5,3.0)(4.5,2.5) rectangle (5.0,3.0)(0.0,3.0) rectangle (0.5,3.5)(0.5,3.0) rectangle (1.0,3.5)(1.0,3.0) rectangle (1.5,3.5)(1.5,3.0) rectangle (2.0,3.5)(2.0,3.0) rectangle (2.5,3.5)(2.5,3.0) rectangle (3.0,3.5)(3.0,3.0) rectangle (3.5,3.5)(3.5,3.0) rectangle (4.0,3.5)(4.0,3.0) rectangle (4.5,3.5)(4.5,3.0) rectangle (5.0,3.5)(0.5,3.5) rectangle (1.0,4.0)(1.0,3.5) rectangle (1.5,4.0)(1.5,3.5) rectangle (2.0,4.0)(2.0,3.5) rectangle (2.5,4.0)(2.5,3.5) rectangle (3.0,4.0)(3.0,3.5) rectangle (3.5,4.0)(3.5,3.5) rectangle (4.0,4.0)(4.0,3.5) rectangle (4.5,4.0)(4.5,3.5) rectangle (5.0,4.0);
\draw[ultra thick](0,1.5)--(5,1.5)--(5,4)--(0,4)--(0,1.5);
\end{scope}

\begin{scope}[xshift=7.0cm,yshift=0.0cm]
\draw[step=0.5cm,color=black,fill=gray] (-0.001,-0.001) grid (5.001,4.501) (0.0,0.0) rectangle (0.5,0.5)(0.5,0.0) rectangle (1.0,0.5)(1.0,0.0) rectangle (1.5,0.5)(1.5,0.0) rectangle (2.0,0.5)(2.0,0.0) rectangle (2.5,0.5)(2.5,0.0) rectangle (3.0,0.5)(3.0,0.0) rectangle (3.5,0.5)(3.5,0.0) rectangle (4.0,0.5)(4.0,0.0) rectangle (4.5,0.5)(0.5,0.5) rectangle (1.0,1.0)(1.0,0.5) rectangle (1.5,1.0)(1.5,0.5) rectangle (2.0,1.0)(2.0,0.5) rectangle (2.5,1.0)(2.5,0.5) rectangle (3.0,1.0)(3.0,0.5) rectangle (3.5,1.0)(3.5,0.5) rectangle (4.0,1.0)(4.0,0.5) rectangle (4.5,1.0)(4.5,0.5) rectangle (5.0,1.0)(0.0,1.0) rectangle (0.5,1.5)(0.5,1.5) rectangle (1.0,2.0)(1.0,1.5) rectangle (1.5,2.0)(1.5,1.5) rectangle (2.0,2.0)(2.0,1.5) rectangle (2.5,2.0)(2.5,1.5) rectangle (3.0,2.0)(3.0,1.5) rectangle (3.5,2.0)(3.5,1.5) rectangle (4.0,2.0)(4.0,1.5) rectangle (4.5,2.0)(4.5,1.5) rectangle (5.0,2.0)(0.0,2.0) rectangle (0.5,2.5)(0.5,2.0) rectangle (1.0,2.5)(1.0,2.0) rectangle (1.5,2.5)(1.5,2.0) rectangle (2.0,2.5)(2.0,2.0) rectangle (2.5,2.5)(2.5,2.0) rectangle (3.0,2.5)(3.0,2.0) rectangle (3.5,2.5)(3.5,2.0) rectangle (4.0,2.5)(4.0,2.0) rectangle (4.5,2.5)(4.5,2.0) rectangle (5.0,2.5)(0.0,2.5) rectangle (0.5,3.0)(0.5,2.5) rectangle (1.0,3.0)(1.0,2.5) rectangle (1.5,3.0)(1.5,2.5) rectangle (2.0,3.0)(2.0,2.5) rectangle (2.5,3.0)(2.5,2.5) rectangle (3.0,3.0)(3.0,2.5) rectangle (3.5,3.0)(3.5,2.5) rectangle (4.0,3.0)(4.0,2.5) rectangle (4.5,3.0)(4.5,2.5) rectangle (5.0,3.0)(0.0,3.0) rectangle (0.5,3.5)(0.5,3.0) rectangle (1.0,3.5)(1.0,3.0) rectangle (1.5,3.5)(1.5,3.0) rectangle (2.0,3.5)(2.0,3.0) rectangle (2.5,3.5)(2.5,3.0) rectangle (3.0,3.5)(3.0,3.0) rectangle (3.5,3.5)(3.5,3.0) rectangle (4.0,3.5)(4.0,3.0) rectangle (4.5,3.5)(4.5,3.0) rectangle (5.0,3.5)(0.0,3.5) rectangle (0.5,4.0)(0.5,3.5) rectangle (1.0,4.0)(1.0,3.5) rectangle (1.5,4.0)(1.5,3.5) rectangle (2.0,4.0)(2.0,3.5) rectangle (2.5,4.0)(2.5,3.5) rectangle (3.0,4.0)(3.0,3.5) rectangle (3.5,4.0)(3.5,3.5) rectangle (4.0,4.0)(4.0,3.5) rectangle (4.5,4.0)(4.5,3.5) rectangle (5.0,4.0)(0.5,4.0) rectangle (1.0,4.5)(1.0,4.0) rectangle (1.5,4.5)(1.5,4.0) rectangle (2.0,4.5)(2.0,4.0) rectangle (2.5,4.5)(2.5,4.0) rectangle (3.0,4.5)(3.0,4.0) rectangle (3.5,4.5)(3.5,4.0) rectangle (4.0,4.5)(4.0,4.0) rectangle (4.5,4.5)(4.5,4.0) rectangle (5.0,4.5);
\draw[ultra thick](0,1.5)--(5,1.5)--(5,4.5)--(0,4.5)--(0,1.5);
\end{scope}

\end{tikzpicture}
\end{center}
    \caption{The infection after running the $3$-neighbour process on $A_0'$ and $A_0''$ separately when $a_1$ is even or odd, respectively. The subgrid surrounded by bold lines is $[a_1-3]\times[a_2]$.}
    \label{fig:fatter}
\end{figure}

\begin{lem}
\label{lem:otherFatter}
Let $a_1\geq 2$ and $a_2\geq 5$ such that $a_1\equiv 2\bmod 6$ and $a_2\equiv 1\bmod 2$. Then there is a set $A_0\subseteq [a_1]\times[a_2]$ such that $|A_0|=\frac{a_1a_2+a_1+a_2-2}{3}$ and $[A_0]=([a_1]\times[a_2])\setminus\{(1,1),(1,a_2)\}$.
\end{lem}

\begin{proof}
Since $a_2-3\geq2$ is even, by Lemma~\ref{lem:fatter}, there is a set $A_0'\subseteq [a_1]\times[a_2-3]$ of cardinality $\frac{a_1(a_2-3)+a_1+(a_2-3)-2}{3}$ such that $[A_0']$ contains all of $[a_1]\times[a_2-3]$ except for $(1,1)$ and $(a_1,a_2-3)$. Since $a_1$ is even, by Lemma~\ref{lem:skinny}, there is a set $A_0''\subseteq [a_1]\times\{a_2-1,a_2\}$ of cardinality $a_1$ whose closure leaves behind only $(a_1,a_2-1)$ and $(1,a_2)$. Taking $A_0'$ and $A_0''$ together with $(a_1,a_2-2)$ yields the desired set $A_0$, analogous to the proof of Lemma~\ref{lem:fatter}.
\end{proof}

Finally, we will start to consider infections whose closures leave behind more than just corner vertices. 

\begin{lem}
\label{lem:longerCorner}
Let $a_1,a_2\geq6$ and $a_1\equiv a_2\equiv 0\bmod 6$, then there is a set $A_0\subseteq[a_1]\times [a_2]$ such that $|A_0|=\frac{a_1a_2+a_1+a_2-6}{3}$ and $[A_0]$ contains all of $[a_1]\times [a_2]$ except for the following:
\[(1,1),(a_1-4,a_2),(a_1-3,a_2),(a_1-2,a_2),(a_1-1,a_2),(a_1,a_2).\]
\end{lem}

\begin{proof}
In the case that $a_1=6$, we use Lemma~\ref{lem:fatter} to get a set $A_0'\subseteq[6]\times[a_2-1]$ of cardinality $\frac{6(a_2-1)+6+(a_2-1)-2}{3} = \frac{6a_2+a_2-3}{3}$ whose closure contains all of $[6]\times[a_2-1]$ except for $(1,1)$ and $(1,a_2-1)$. Now, let $A_0:=A_0'\cup\{(1,a_2)\}$. Then $|A_0|=\frac{6a_2+a_2-3}{3} + 1 = \frac{a_1a_2+a_1+a_2-6}{3}$ since $a_1=6$ and it is not hard to see that $A_0$ has the correct closure; see Figure~\ref{fig:longerCorner}~(a). 

So, we assume that $a_1 \ge 12$. Let $a_1':=a_1-7$ and note that $a_1'\equiv 5\bmod 6$. By Lemma~\ref{lem:fatter}, there is a set $A_0'\subseteq[a_1']\times[a_2]$ of cardinality $\frac{a_1'a_2+a_1'+a_2-2}{3}=\frac{(a_1-7)a_2+(a_1-7)+a_2-2}{3}$ such that the closure of $A_0'$ contains all but the corners $(1,1)$ and $(a_1',1)$. By the result of the previous paragraph, there is a subset $A_0''$ of $\{a_1'+2,\dots,a_1\}\times [a_2]$ of cardinality $\frac{6a_2+a_2}{3}$ whose closure contains all but $(a_1'+2,1), (a_1-4,a_2),(a_1-3,a_2),(a_1-2,a_2),(a_1-1,a_2),(a_1,a_2)$. We let $A_0:=A_0'\cup A_0''\cup\{(a_1'+1,1)\}$. Then
\[|A_0|=\frac{(a_1-7)a_2+(a_1-7)+a_2-2}{3} + 1 + \frac{6a_2+a_2}{3} = \frac{a_1a_2+a_1+a_2-6}{3}.\]
To see that $A_0$ has the correct closure is a simple application of Lemma~\ref{lem:immuneRegions}; see Figure~\ref{fig:longerCorner}~(b). 
\end{proof}

\begin{figure}[htbp]
\begin{center}
\begin{tikzpicture}
\begin{scope}[xshift=0.0cm,yshift=0.0cm]
\draw[step=0.5cm,color=black,fill=gray] (-0.001,-0.001) grid (6.001,3.001) (0.0,0.0) rectangle (0.5,0.5)(0.5,0.0) rectangle (1.0,0.5)(1.0,0.0) rectangle (1.5,0.5)(1.5,0.0) rectangle (2.0,0.5)(2.0,0.0) rectangle (2.5,0.5)(2.5,0.0) rectangle (3.0,0.5)(3.0,0.0) rectangle (3.5,0.5)(3.5,0.0) rectangle (4.0,0.5)(4.0,0.0) rectangle (4.5,0.5)(4.5,0.0) rectangle (5.0,0.5)(5.0,0.0) rectangle (5.5,0.5)(0.0,0.5) rectangle (0.5,1.0)(0.5,0.5) rectangle (1.0,1.0)(1.0,0.5) rectangle (1.5,1.0)(1.5,0.5) rectangle (2.0,1.0)(2.0,0.5) rectangle (2.5,1.0)(2.5,0.5) rectangle (3.0,1.0)(3.0,0.5) rectangle (3.5,1.0)(3.5,0.5) rectangle (4.0,1.0)(4.0,0.5) rectangle (4.5,1.0)(4.5,0.5) rectangle (5.0,1.0)(5.0,0.5) rectangle (5.5,1.0)(0.0,1.0) rectangle (0.5,1.5)(0.5,1.0) rectangle (1.0,1.5)(1.0,1.0) rectangle (1.5,1.5)(1.5,1.0) rectangle (2.0,1.5)(2.0,1.0) rectangle (2.5,1.5)(2.5,1.0) rectangle (3.0,1.5)(3.0,1.0) rectangle (3.5,1.5)(3.5,1.0) rectangle (4.0,1.5)(4.0,1.0) rectangle (4.5,1.5)(4.5,1.0) rectangle (5.0,1.5)(5.0,1.0) rectangle (5.5,1.5)(0.0,1.5) rectangle (0.5,2.0)(0.5,1.5) rectangle (1.0,2.0)(1.0,1.5) rectangle (1.5,2.0)(1.5,1.5) rectangle (2.0,2.0)(2.0,1.5) rectangle (2.5,2.0)(2.5,1.5) rectangle (3.0,2.0)(3.0,1.5) rectangle (3.5,2.0)(3.5,1.5) rectangle (4.0,2.0)(4.0,1.5) rectangle (4.5,2.0)(4.5,1.5) rectangle (5.0,2.0)(5.0,1.5) rectangle (5.5,2.0)(0.0,2.0) rectangle (0.5,2.5)(0.5,2.0) rectangle (1.0,2.5)(1.0,2.0) rectangle (1.5,2.5)(1.5,2.0) rectangle (2.0,2.5)(2.0,2.0) rectangle (2.5,2.5)(2.5,2.0) rectangle (3.0,2.5)(3.0,2.0) rectangle (3.5,2.5)(3.5,2.0) rectangle (4.0,2.5)(4.0,2.0) rectangle (4.5,2.5)(4.5,2.0) rectangle (5.0,2.5)(5.0,2.0) rectangle (5.5,2.5)(0.5,2.5) rectangle (1.0,3.0)(1.0,2.5) rectangle (1.5,3.0)(1.5,2.5) rectangle (2.0,3.0)(2.0,2.5) rectangle (2.5,3.0)(2.5,2.5) rectangle (3.0,3.0)(3.0,2.5) rectangle (3.5,3.0)(3.5,2.5) rectangle (4.0,3.0)(4.0,2.5) rectangle (4.5,3.0)(4.5,2.5) rectangle (5.0,3.0)(5.5,2.5) rectangle (6.0,3.0);
\node at (3,-0.5) {(a)};
\end{scope}

\begin{scope}[xshift=7.0cm,yshift=0.0cm]
\draw[step=0.5cm,color=black,fill=gray] (-0.001,-0.001) grid (6.001,6.001) (0.0,0.0) rectangle (0.5,0.5)(0.5,0.0) rectangle (1.0,0.5)(1.0,0.0) rectangle (1.5,0.5)(1.5,0.0) rectangle (2.0,0.5)(2.0,0.0) rectangle (2.5,0.5)(2.5,0.0) rectangle (3.0,0.5)(3.0,0.0) rectangle (3.5,0.5)(3.5,0.0) rectangle (4.0,0.5)(4.0,0.0) rectangle (4.5,0.5)(4.5,0.0) rectangle (5.0,0.5)(5.0,0.0) rectangle (5.5,0.5)(0.0,0.5) rectangle (0.5,1.0)(0.5,0.5) rectangle (1.0,1.0)(1.0,0.5) rectangle (1.5,1.0)(1.5,0.5) rectangle (2.0,1.0)(2.0,0.5) rectangle (2.5,1.0)(2.5,0.5) rectangle (3.0,1.0)(3.0,0.5) rectangle (3.5,1.0)(3.5,0.5) rectangle (4.0,1.0)(4.0,0.5) rectangle (4.5,1.0)(4.5,0.5) rectangle (5.0,1.0)(5.0,0.5) rectangle (5.5,1.0)(0.0,1.0) rectangle (0.5,1.5)(0.5,1.0) rectangle (1.0,1.5)(1.0,1.0) rectangle (1.5,1.5)(1.5,1.0) rectangle (2.0,1.5)(2.0,1.0) rectangle (2.5,1.5)(2.5,1.0) rectangle (3.0,1.5)(3.0,1.0) rectangle (3.5,1.5)(3.5,1.0) rectangle (4.0,1.5)(4.0,1.0) rectangle (4.5,1.5)(4.5,1.0) rectangle (5.0,1.5)(5.0,1.0) rectangle (5.5,1.5)(0.0,1.5) rectangle (0.5,2.0)(0.5,1.5) rectangle (1.0,2.0)(1.0,1.5) rectangle (1.5,2.0)(1.5,1.5) rectangle (2.0,2.0)(2.0,1.5) rectangle (2.5,2.0)(2.5,1.5) rectangle (3.0,2.0)(3.0,1.5) rectangle (3.5,2.0)(3.5,1.5) rectangle (4.0,2.0)(4.0,1.5) rectangle (4.5,2.0)(4.5,1.5) rectangle (5.0,2.0)(5.0,1.5) rectangle (5.5,2.0)(0.0,2.0) rectangle (0.5,2.5)(0.5,2.0) rectangle (1.0,2.5)(1.0,2.0) rectangle (1.5,2.5)(1.5,2.0) rectangle (2.0,2.5)(2.0,2.0) rectangle (2.5,2.5)(2.5,2.0) rectangle (3.0,2.5)(3.0,2.0) rectangle (3.5,2.5)(3.5,2.0) rectangle (4.0,2.5)(4.0,2.0) rectangle (4.5,2.5)(4.5,2.0) rectangle (5.0,2.5)(5.0,2.0) rectangle (5.5,2.5)(0.5,2.5) rectangle (1.0,3.0)(1.0,2.5) rectangle (1.5,3.0)(1.5,2.5) rectangle (2.0,3.0)(2.0,2.5) rectangle (2.5,3.0)(2.5,2.5) rectangle (3.0,3.0)(3.0,2.5) rectangle (3.5,3.0)(3.5,2.5) rectangle (4.0,3.0)(4.0,2.5) rectangle (4.5,3.0)(4.5,2.5) rectangle (5.0,3.0)(5.0,2.5) rectangle (5.5,3.0)(5.5,2.5) rectangle (6.0,3.0)(0.0,3.0) rectangle (0.5,3.5)(0.5,3.5) rectangle (1.0,4.0)(1.0,3.5) rectangle (1.5,4.0)(1.5,3.5) rectangle (2.0,4.0)(2.0,3.5) rectangle (2.5,4.0)(2.5,3.5) rectangle (3.0,4.0)(3.0,3.5) rectangle (3.5,4.0)(3.5,3.5) rectangle (4.0,4.0)(4.0,3.5) rectangle (4.5,4.0)(4.5,3.5) rectangle (5.0,4.0)(5.0,3.5) rectangle (5.5,4.0)(5.5,3.5) rectangle (6.0,4.0)(0.0,4.0) rectangle (0.5,4.5)(0.5,4.0) rectangle (1.0,4.5)(1.0,4.0) rectangle (1.5,4.5)(1.5,4.0) rectangle (2.0,4.5)(2.0,4.0) rectangle (2.5,4.5)(2.5,4.0) rectangle (3.0,4.5)(3.0,4.0) rectangle (3.5,4.5)(3.5,4.0) rectangle (4.0,4.5)(4.0,4.0) rectangle (4.5,4.5)(4.5,4.0) rectangle (5.0,4.5)(5.0,4.0) rectangle (5.5,4.5)(5.5,4.0) rectangle (6.0,4.5)(0.0,4.5) rectangle (0.5,5.0)(0.5,4.5) rectangle (1.0,5.0)(1.0,4.5) rectangle (1.5,5.0)(1.5,4.5) rectangle (2.0,5.0)(2.0,4.5) rectangle (2.5,5.0)(2.5,4.5) rectangle (3.0,5.0)(3.0,4.5) rectangle (3.5,5.0)(3.5,4.5) rectangle (4.0,5.0)(4.0,4.5) rectangle (4.5,5.0)(4.5,4.5) rectangle (5.0,5.0)(5.0,4.5) rectangle (5.5,5.0)(5.5,4.5) rectangle (6.0,5.0)(0.0,5.0) rectangle (0.5,5.5)(0.5,5.0) rectangle (1.0,5.5)(1.0,5.0) rectangle (1.5,5.5)(1.5,5.0) rectangle (2.0,5.5)(2.0,5.0) rectangle (2.5,5.5)(2.5,5.0) rectangle (3.0,5.5)(3.0,5.0) rectangle (3.5,5.5)(3.5,5.0) rectangle (4.0,5.5)(4.0,5.0) rectangle (4.5,5.5)(4.5,5.0) rectangle (5.0,5.5)(5.0,5.0) rectangle (5.5,5.5)(5.5,5.0) rectangle (6.0,5.5)(0.5,5.5) rectangle (1.0,6.0)(1.0,5.5) rectangle (1.5,6.0)(1.5,5.5) rectangle (2.0,6.0)(2.0,5.5) rectangle (2.5,6.0)(2.5,5.5) rectangle (3.0,6.0)(3.0,5.5) rectangle (3.5,6.0)(3.5,5.5) rectangle (4.0,6.0)(4.0,5.5) rectangle (4.5,6.0)(4.5,5.5) rectangle (5.0,6.0)(5.0,5.5) rectangle (5.5,6.0)(5.5,5.5) rectangle (6.0,6.0);
\node at (3,-0.5) {(b)};
\end{scope}

\end{tikzpicture}
\end{center}
    \caption{Illustrations of the constructions that arise in the proof of Lemma~\ref{lem:longerCorner} in the cases (a) $a_1=6$ and $a_2=12$ and (b) $a_1=a_2=12$.}
    \label{fig:longerCorner}
\end{figure}
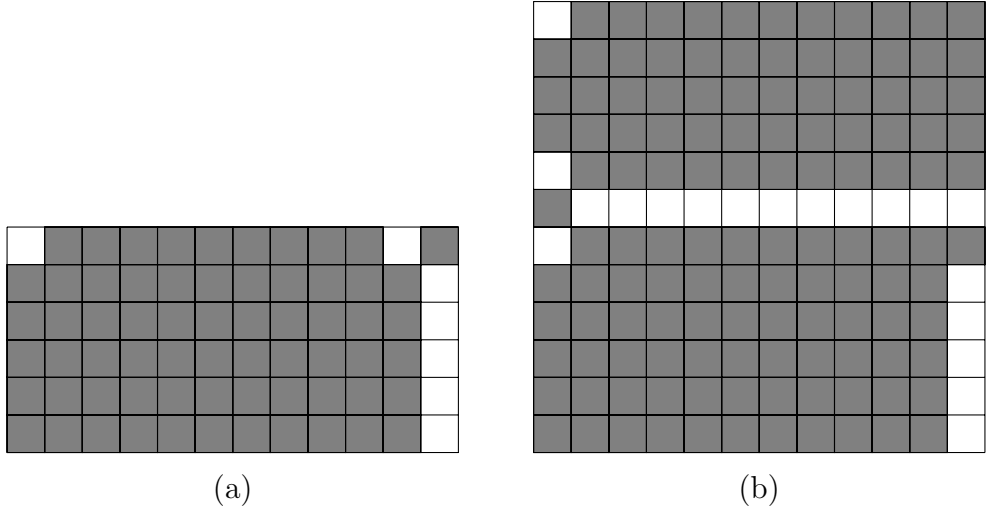

\begin{lem}
\label{lem:longerCorner2}
Let $a_1\geq6$ and $a_2\geq9$ such that $a_1\equiv 0\bmod 6$ and $a_2\equiv 3\bmod 6$. Then there is a set $A_0\subseteq[a_1]\times [a_2]$ such that $|A_0|=\frac{a_1a_2+a_1+a_2-6}{3}$ and $[A_0]$ contains all of $[a_1]\times [a_2]$ except for the following:
\[(a_1,1),(a_1-4,a_2),(a_1-3,a_2),(a_1-2,a_2),(a_1-1,a_2),(a_1,a_2).\]
\end{lem}

\begin{proof}
Applying Lemma~\ref{lem:skinny}, we get a set of cardinality $a_1$ in $[a_1]\times [2]$ whose closure contains all points of this grid except for $(a_1,1)$ and $(1,2)$. Also, by Lemma~\ref{lem:longerCorner}, we get a set in $[a_1]\times\{4,\dots,a_2\}$ of cardinality $\frac{a_1(a_2-3)+a_1+(a_2-3)-6}{3}$ whose closure contains all elements of this grid except for $(1,4),(a_1-4,a_2),(a_1-3,a_2),(a_1-2,a_2),(a_1-1,a_2)$ and $(a_1,a_2)$. Putting these two constructions together with an additional infection on $(1,3)$ yields a set with the desired closure of cardinality
\[a_1+\frac{a_1(a_2-3)+a_1+(a_2-3)-6}{3}+1=\frac{a_1a_2+a_1+a_2-6}{3}.\qedhere\]
\end{proof}

\begin{lem}
\label{lem:jagged}
If $a_1,a_2\geq7$ and $a_1\equiv a_2\equiv 1\bmod 6$, then there is a set $A_0\subseteq[a_1]\times[a_2]$ such that $|A_0|=\frac{a_1a_2+a_1+a_2-9}{3}$ and $[A_0]$ contains all of $[a_1]\times[a_2]$ except for the following:
\[(a_1-5,1),(a_1-4,1),(a_1-3,1),(a_1-2,1),(a_1-1,1),(a_1,1),(a_1-2,2),(a_1-1,2),(a_1,2).\]
\end{lem}

\begin{proof}
Since $a_1-5\equiv 2\bmod 6$ and $a_2$ is odd, Lemma~\ref{lem:otherFatter} tells us that there is a subset of $[a_1-5]\times [a_2]$ of cardinality $\frac{(a_1-5)a_2+(a_1-5)+a_2-2}{3}$ whose closure contains all of this grid except for the corners $(a_1-5,1)$ and $(a_1-5,a_2)$. Also, since $a_2-2\equiv 5\bmod 6$ by Lemma~\ref{lem:fatter}, there is a subset of $\{a_1-3,\dots,a_1\},\times \{3,\dots,a_2\}$ of cardinality $\frac{4(a_2-2) + 4 + (a_2-2)-2}{3}$ whose closure contains all of this grid except for $(a_1-3,3)$ and $(a_1-3,a_2)$. We take these two constructions and add two additional infected points: $(a_1-3,2)$ and $(a_1-4,a_2)$. It is not hard to see that this infection has the desired closure; see Figure~\ref{fig:jagged}. The cardinality of this set is
\[\frac{(a_1-5)a_2+(a_1-5)+a_2-2}{3} + \frac{4(a_2-2) + 4 + (a_2-2)-2}{3} +2 = \frac{a_1a_2+a_1+a_2-9}{3}.\qedhere\]
\end{proof}

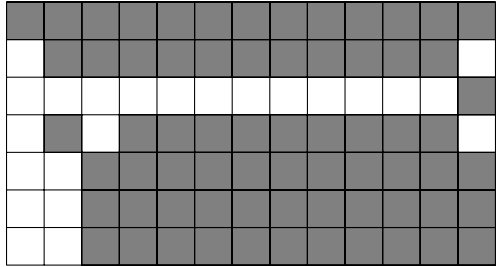
\begin{figure}[htbp]
\begin{center}
\begin{tikzpicture}
\begin{scope}[xshift=0.0cm,yshift=0.0cm]
\draw[step=0.5cm,color=black,fill=gray] (-0.001,-0.001) grid (6.501,3.501) (1.0,0.0) rectangle (1.5,0.5)(1.5,0.0) rectangle (2.0,0.5)(2.0,0.0) rectangle (2.5,0.5)(2.5,0.0) rectangle (3.0,0.5)(3.0,0.0) rectangle (3.5,0.5)(3.5,0.0) rectangle (4.0,0.5)(4.0,0.0) rectangle (4.5,0.5)(4.5,0.0) rectangle (5.0,0.5)(5.0,0.0) rectangle (5.5,0.5)(5.5,0.0) rectangle (6.0,0.5)(6.0,0.0) rectangle (6.5,0.5)(1.0,0.5) rectangle (1.5,1.0)(1.5,0.5) rectangle (2.0,1.0)(2.0,0.5) rectangle (2.5,1.0)(2.5,0.5) rectangle (3.0,1.0)(3.0,0.5) rectangle (3.5,1.0)(3.5,0.5) rectangle (4.0,1.0)(4.0,0.5) rectangle (4.5,1.0)(4.5,0.5) rectangle (5.0,1.0)(5.0,0.5) rectangle (5.5,1.0)(5.5,0.5) rectangle (6.0,1.0)(6.0,0.5) rectangle (6.5,1.0)(1.0,1.0) rectangle (1.5,1.5)(1.5,1.0) rectangle (2.0,1.5)(2.0,1.0) rectangle (2.5,1.5)(2.5,1.0) rectangle (3.0,1.5)(3.0,1.0) rectangle (3.5,1.5)(3.5,1.0) rectangle (4.0,1.5)(4.0,1.0) rectangle (4.5,1.5)(4.5,1.0) rectangle (5.0,1.5)(5.0,1.0) rectangle (5.5,1.5)(5.5,1.0) rectangle (6.0,1.5)(6.0,1.0) rectangle (6.5,1.5)(0.5,1.5) rectangle (1.0,2.0)(1.5,1.5) rectangle (2.0,2.0)(2.0,1.5) rectangle (2.5,2.0)(2.5,1.5) rectangle (3.0,2.0)(3.0,1.5) rectangle (3.5,2.0)(3.5,1.5) rectangle (4.0,2.0)(4.0,1.5) rectangle (4.5,2.0)(4.5,1.5) rectangle (5.0,2.0)(5.0,1.5) rectangle (5.5,2.0)(5.5,1.5) rectangle (6.0,2.0)(6.0,2.0) rectangle (6.5,2.5)(0.5,2.5) rectangle (1.0,3.0)(1.0,2.5) rectangle (1.5,3.0)(1.5,2.5) rectangle (2.0,3.0)(2.0,2.5) rectangle (2.5,3.0)(2.5,2.5) rectangle (3.0,3.0)(3.0,2.5) rectangle (3.5,3.0)(3.5,2.5) rectangle (4.0,3.0)(4.0,2.5) rectangle (4.5,3.0)(4.5,2.5) rectangle (5.0,3.0)(5.0,2.5) rectangle (5.5,3.0)(5.5,2.5) rectangle (6.0,3.0)(0.0,3.0) rectangle (0.5,3.5)(0.5,3.0) rectangle (1.0,3.5)(1.0,3.0) rectangle (1.5,3.5)(1.5,3.0) rectangle (2.0,3.5)(2.0,3.0) rectangle (2.5,3.5)(2.5,3.0) rectangle (3.0,3.5)(3.0,3.0) rectangle (3.5,3.5)(3.5,3.0) rectangle (4.0,3.5)(4.0,3.0) rectangle (4.5,3.5)(4.5,3.0) rectangle (5.0,3.5)(5.0,3.0) rectangle (5.5,3.5)(5.5,3.0) rectangle (6.0,3.5)(6.0,3.0) rectangle (6.5,3.5);
\end{scope}

\end{tikzpicture}
\end{center}
    \caption{Illustration of the construction in the proof of Lemma~\ref{lem:jagged} in the case $a_1=7$ and $a_2=13$.}
    \label{fig:jagged}
\end{figure}

\begin{lem}
\label{lem:cutcorner}
Let $a_1,a_2\geq5$ so that $a_1\equiv 1\bmod 2$ and $a_2\equiv 5\bmod 6$. Then there is a set $A_0\subseteq[a_1]\times[a_2]$ such that $|A_0|=\frac{a_1a_2+a_1+a_2-11}{3}$ and $[A_0]$ contains all of $[a_1]\times[a_2]$ except for the 11 elements of the following set:
\[\{(x,y): x+y\geq a_1+a_2-3\}\cup\{(a_1-2,a_2-2)\}.\]
\end{lem}

\begin{proof}
By Lemma~\ref{lem:fatter}, we can take a subset of $[a_1-3]\times [a_2]$ of cardinality $\frac{(a_1-3)a_2+(a_1-3)+a_2-2}{3}$ whose closure contains all of this grid except for the corners $(a_1-3,1)$ and $(a_1-3,a_2)$. Also, by Lemma~\ref{lem:skinny}, there is a subset of the grid $\{a_1-1,a_1\}\times[a_2-3]$ of cardinality $a_2-3$ whose closure contains all of this grid except for $(a_1-1,1)$ and $(a_1,a_2-3)$. We take the union of these two sets together with the point $(a_1-2,1)$. It is an easy application of Lemma~\ref{lem:immuneRegions} to show that this set has the correct closure; see Figure~\ref{fig:cutcorner}. The cardinality of this set is
\[\frac{(a_1-3)a_2+(a_1-3)+a_2-2}{3}+a_2-3+1=\frac{a_1a_2+a_1+a_2-11}{3}.\qedhere\]
\end{proof}

\begin{figure}[htbp]
\begin{center}
\begin{tikzpicture}
\begin{scope}[xshift=0.0cm,yshift=0.0cm]
\draw[step=0.5cm,color=black,fill=gray] (-0.001,-0.001) grid (5.501,3.501) (0.0,0.0) rectangle (0.5,0.5)(0.5,0.0) rectangle (1.0,0.5)(1.0,0.0) rectangle (1.5,0.5)(1.5,0.0) rectangle (2.0,0.5)(2.0,0.0) rectangle (2.5,0.5)(2.5,0.0) rectangle (3.0,0.5)(3.0,0.0) rectangle (3.5,0.5)(0.5,0.5) rectangle (1.0,1.0)(1.0,0.5) rectangle (1.5,1.0)(1.5,0.5) rectangle (2.0,1.0)(2.0,0.5) rectangle (2.5,1.0)(2.5,0.5) rectangle (3.0,1.0)(3.0,0.5) rectangle (3.5,1.0)(3.5,0.5) rectangle (4.0,1.0)(0.0,1.0) rectangle (0.5,1.5)(0.5,1.5) rectangle (1.0,2.0)(1.0,1.5) rectangle (1.5,2.0)(1.5,1.5) rectangle (2.0,2.0)(2.0,1.5) rectangle (2.5,2.0)(2.5,1.5) rectangle (3.0,2.0)(3.0,1.5) rectangle (3.5,2.0)(3.5,1.5) rectangle (4.0,2.0)(4.0,1.5) rectangle (4.5,2.0)(4.5,1.5) rectangle (5.0,2.0)(0.0,2.0) rectangle (0.5,2.5)(0.5,2.0) rectangle (1.0,2.5)(1.0,2.0) rectangle (1.5,2.5)(1.5,2.0) rectangle (2.0,2.5)(2.0,2.0) rectangle (2.5,2.5)(2.5,2.0) rectangle (3.0,2.5)(3.0,2.0) rectangle (3.5,2.5)(3.5,2.0) rectangle (4.0,2.5)(4.0,2.0) rectangle (4.5,2.5)(4.5,2.0) rectangle (5.0,2.5)(5.0,2.0) rectangle (5.5,2.5)(0.0,2.5) rectangle (0.5,3.0)(0.5,2.5) rectangle (1.0,3.0)(1.0,2.5) rectangle (1.5,3.0)(1.5,2.5) rectangle (2.0,3.0)(2.0,2.5) rectangle (2.5,3.0)(2.5,2.5) rectangle (3.0,3.0)(3.0,2.5) rectangle (3.5,3.0)(3.5,2.5) rectangle (4.0,3.0)(4.0,2.5) rectangle (4.5,3.0)(4.5,2.5) rectangle (5.0,3.0)(5.0,2.5) rectangle (5.5,3.0)(0.0,3.0) rectangle (0.5,3.5)(0.5,3.0) rectangle (1.0,3.5)(1.0,3.0) rectangle (1.5,3.5)(1.5,3.0) rectangle (2.0,3.5)(2.0,3.0) rectangle (2.5,3.5)(2.5,3.0) rectangle (3.0,3.5)(3.0,3.0) rectangle (3.5,3.5)(3.5,3.0) rectangle (4.0,3.5)(4.0,3.0) rectangle (4.5,3.5)(4.5,3.0) rectangle (5.0,3.5)(5.0,3.0) rectangle (5.5,3.5);
\end{scope}

\end{tikzpicture}
\end{center}
    \caption{Illustration of the construction in the proof of Lemma~\ref{lem:cutcorner} in the case $a_1=7$ and $a_2=11$.}
    \label{fig:cutcorner}
\end{figure}
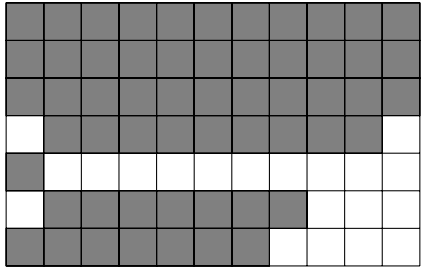

\section{Folding 2D Shapes Into the Third Dimension}
\label{sec:folding}

Our aim in this section is to present a strategy for ``folding'' the constructions from the previous section into three dimensions to yield a percolating set in $\prod_{i=1}^3[a_i]$. Some of the arguments involve showing that the folded $2$-dimensional structure ``separates'' the two faces in direction $j$ for all $j\in [3]$; the following lemma shows that this is sufficient. 

\begin{lem}
\label{lem:sidesSplit}
For $d\geq1$ and $a_1,\dots,a_d\geq1$, let $A$ be a subset of $\prod_{i=1}^d[a_i]$. If, for every $j\in [d]$, every path in $\prod_{i=1}^d[a_i]$ between the two faces in direction $j$ contains a vertex of $A$, then $A$ percolates under the $d$-neighbour process in $\prod_{i=1}^d[a_i]$. 
\end{lem}

\begin{proof}
We proceed by induction on $\prod_{i=1}^d a_i - |A|$, where the base case $\prod_{i=1}^d a_i - |A|=0$ is trivial. Let $X$ be any component of the subgraph of $\prod_{i=1}^d[a_i]$ induced by $\prod_{i=1}^d[a_i]\setminus A$. By the hypothesis of the lemma, for each $j\in [d]$, we have that either $X\cap L_{j,1}=\emptyset$ or $X\cap L_{j,a_j}=\emptyset$. Without loss of generality, $X\cap L_{j,1}=\emptyset$ for all $j\in [d]$. Now, let $v$ be the lexicographically minimal element of $X$. Then, for every $j\in [d]$, we see that $v$ is adjacent to a vertex of $A$ via an edge in direction $j$. Thus, if the vertices of $A$ are infected, then $v$ becomes infected in the first step of the $d$-neighbour process. The result follows by applying the induction hypothesis to $A\cup\{v\}$. 
\end{proof}

Let us now explain the manner in which we implant $2$-dimensional structures into $3$-dimensional grids. Given graphs $H$ and $G$, a \emph{local embedding} of $H$ into $G$ is a function $f:V(H)\to V(G)$ such that $f(u)f(v)\in E(G)$ whenever $uv\in E(H)$ (i.e. $f$ is a \emph{homomorphism}) and, for each $v\in V(H)$, the restriction of $f$ to $N(v)$ is injective. We remark that most, but not all, of the local embeddings used in this section will actually be injective on their entire domain.

\begin{lem}
\label{lem:embedding}
Let $H$ and $G$ be graphs, let $r\geq1$ and let $A_0\subseteq V(H)$. If $f$ is a local embedding of $H$ into $G$, then $f([A_0]_{r,H})\subseteq [f(A_0)]_{r,G}$. 
\end{lem}

\begin{proof}
Let $m$ be the cardinality of $[A_0]$ and let $v_1,\dots,v_m$ be the vertices of $[A_0]$ ordered so that $\{v_1,\dots,v_{|A_0|}\}=A_0$ and, if $|A_0|+1\leq i\leq m$, then $v_i$ has at least $r$ neighbours in $\{v_1,\dots,v_{i-1}\}$. Clearly, $[f(A_0)]_{r,G}$ contains $f(v_i)$ for all $1\leq i\leq |A_0|$. Since $f$ is a local embedding, for each $|A_0|+1\leq i\leq m$, the vertex $f(v_i)$ has at least $r$ neighbours in $\{f(v_1),\dots, f(v_{i-1})\}$. Thus, $[f(A_0)]_{r,G}$ contains $f(v_i)$ for all $1\leq i\leq m$ and the lemma follows. 
\end{proof}

Our aim in the rest of this section is to use Lemmas~\ref{lem:immuneRegions},~\ref{lem:sidesSplit} and~\ref{lem:embedding} to demonstrate that many tuples $(a_1,a_2,a_3)$ for $a_1\in \{2,3\}$ are perfect. Most of the arguments follow the same general pattern. First, we exhibit a set which percolates under the $3$-neighbour process in an induced subgraph $H$ of a $2$-dimensional grid; we call an induced subgraph of a $2$-dimensional grid a \emph{grid graph} for short. Second, we describe a local embedding of $H$ into $[a_1]\times [a_2]\times[a_3]$. Third, we argue that either the image of $H$ under the local embedding separates the faces of $[a_1]\times [a_2]\times [a_3]$ in direction $j$ for all $j\in [3]$, or that this property is satisfied after adding a few more infected vertices and/or performing a few steps of the $3$-neighbour process, and so percolation occurs by Lemma~\ref{lem:sidesSplit}. In some cases, the third step is replaced by an argument showing that the set percolates without appealing to Lemma~\ref{lem:sidesSplit}. In fact, some of the proofs do not follow these three steps at all.

As in the previous section, many of the proofs are best explained with a series of accompanying diagrams. In these diagrams, the way that we depict an $a\times b\times c$ grid in this section is as $c$ copies of an $a\times b$ grid listed left to right.\footnote{Particularly large constructions are sometimes written on two lines. The first $a\times b$ grid on the second line is viewed as being immediately to the right of the last one on the first line.} Each square in an $a\times b$ grid is viewed as being directly above the corresponding square in the grid to the right of it and below the corresponding square in the grid to the left of it. In some of these diagrams, some of the non-grey squares contain numbers corresponding to the number of steps the corresponding vertex takes to become infected. In other words, each square containing a $1$ is adjacent to three grey squares and, for $k\geq2$, each square containing the number $k$ is adjacent to three squares which are either grey or contain a number less than $k$, at least one of which contains the number $k-1$. We also indicate local embeddings by diagrams in which each cell is labelled by the image of the vertex that it is mapped to. 

As discussed in the previous paragraph, many of our arguments involves running the $3$-neighbour process to starting with a set that percolates in a grid graph $H$, mapping $H$ into $[a_1]\times [a_2]\times [a_3]$ via a local embedding, and then running the $3$-neighbour process in $[a_1]\times [a_2]\times [a_3]$ with the image of $H$ (perhaps with a few additional infections added) as the starting infection. For this reason, most of the diagrams depicting infections in 3-dimensional grids actually show the full image of a grid graph $H$ under a local embedding in grey; e.g., Figure~\ref{fig:236}~(c) shows the image of $V(H_6)$ in Figure~\ref{fig:236}~(a) under the local embedding described in Figure~\ref{fig:236}~(b). 

\begin{prop}
\label{prop:2,3,0mod6}
If $a_3\geq 6$ and $a_3\equiv 0\bmod 6$, then $(2,3,a_3)$ is perfect. 
\end{prop}

\begin{proof}
For the purposes of presenting the diagrams, we will work in the grid $[2]\times[6k]\times [3]$ for $k\geq1$. Let us first consider the case $k=1$. Our goal is to construct a percolating set of cardinality $\frac{2\cdot 3+2\cdot 6+3\cdot 6}{3} = 12$. Consider first the grid graph $H_6$ in Figure~\ref{fig:236}~(a). The infection depicted in the figure percolates in $H_6$ by Corollary~\ref{cor:iff2D}. Figure~\ref{fig:236}~(b) describes a local embedding of $H_6$ into $[2]\times[6]\times [3]$. By Lemma~\ref{lem:embedding}, if the 10 vertices of the image of the initial infection in $H_6$ under the embedding are infected in $[2]\times[6]\times[3]$, then the image of every vertex of $H_6$ under the embedding eventually becomes infected. The image of $H_6$ under this embedding is depicted in Figure~\ref{fig:236}~(c). In Figure~\ref{fig:236}~(d), it is shown that the set obtained by adding $(1,6,1)$ and $(2,1,3)$ to the image of $H_6$ percolates in $[2]\times [6]\times [3]$. Thus, there is a percolating set of cardinality $10+2=12$. 

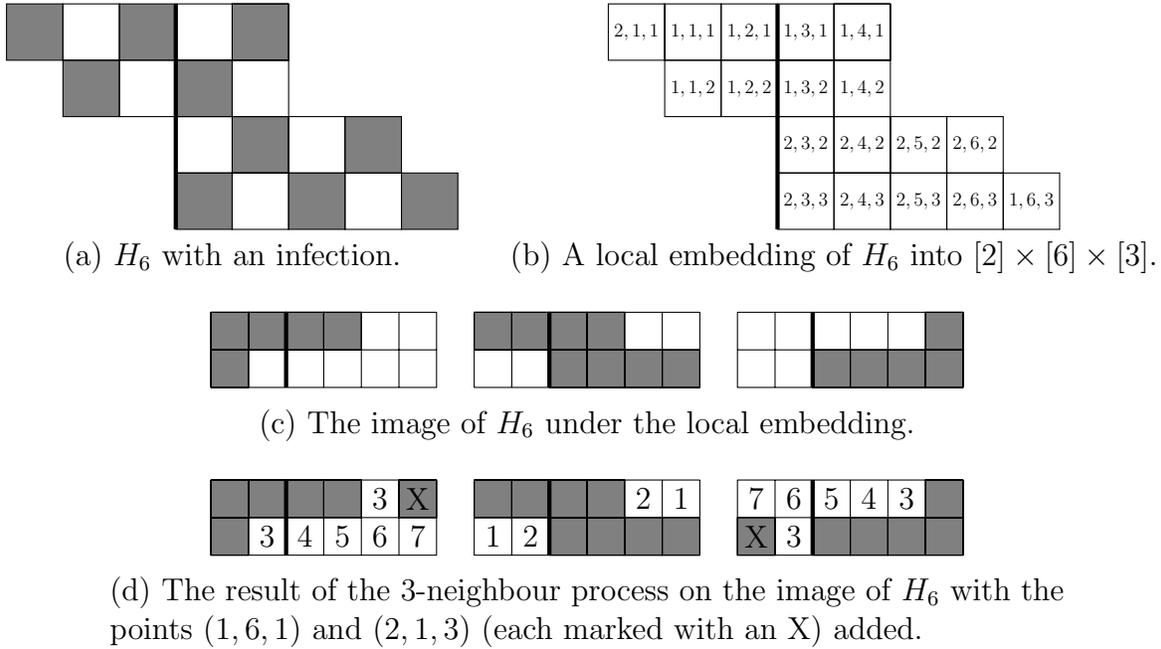
\begin{figure}[htbp]
    \begin{center}
\begin{tikzpicture}
\begin{scope}[xshift=0.0cm,yshift=0.0cm,scale=1.5]
\draw[step=0.5cm,color=black,fill=gray] (0.499,0.499) grid (1.501,2.001) (0.5,0.0) rectangle (1.0,0.5)(1.5,0.0) rectangle (2.0,0.5)(1.0,0.5) rectangle (1.5,1.0)(0.5,1.0) rectangle (1.0,1.5)(0.0,1.5) rectangle (0.5,2.0)(1.0,1.5) rectangle (1.5,2.0)(-0.5,1.0) rectangle (0.0,1.5) (-1.0,1.5) rectangle (-0.5,2.0) (2.0,0.5) rectangle (2.5,1.0) (2.5,0.0) rectangle (3.0,0.5);
\node at (1.25,0.25) {};\node at (0.75,0.75) {};\node at (1.25,1.25) {};\node at (0.75,1.75) {};
\draw  (1.0,0.0) rectangle (1.5,0.5) (2.0,0.0) rectangle (2.5,0.5) (1.5,0.5) rectangle (2.0,1.0) (0.0,1.0) rectangle (0.5,1.5) (-0.5,1.5) rectangle (0.0,2.0);
\draw[ultra thick](0.5,0)--(0.5,2);
\node at (1,-0.25) {(a) $H_6$ with an infection.};
\end{scope}
\begin{scope}[xshift=8.0cm,yshift=0.0cm,scale=1.5]
\draw[step=0.5cm,color=black] (0.499,0.499) grid (1.501,2.001) (0.5,0.0) rectangle (1.0,0.5)(1.5,0.0) rectangle (2.0,0.5)(1.0,0.5) rectangle (1.5,1.0)(0.5,1.0) rectangle (1.0,1.5)(0.0,1.5) rectangle (0.5,2.0)(1.0,1.5) rectangle (1.5,2.0)(-0.5,1.0) rectangle (0.0,1.5) (2.0,0.5) rectangle (2.5,1.0) (2.5,0.0) rectangle (3.0,0.5);
\node[scale=0.6] at (-0.25,1.25) {$1,1,2$};\node[scale=0.6] at (0.25,1.25) {$1,2,2$};\node[scale=0.6] at (0.75,1.25) {$1,3,2$};\node[scale=0.6] at (1.25,1.25) {$1,4,2$};\node[scale=0.6] at (-0.75,1.75) {$2,1,1$};\node[scale=0.6] at (-0.25,1.75) {$1,1,1$};\node[scale=0.6] at (0.25,1.75) {$1,2,1$};\node[scale=0.6] at (0.75,1.75) {$1,3,1$};\node[scale=0.6] at (1.25,1.75) {$1,4,1$};\node[scale=0.6] at (0.75,0.75) {$2,3,2$};\node[scale=0.6] at (1.25,0.75) {$2,4,2$};\node[scale=0.6] at (1.75,0.75) {$2,5,2$};\node[scale=0.6] at (2.25,0.75) {$2,6,2$};\node[scale=0.6] at (0.75,0.25) {$2,3,3$};\node[scale=0.6] at (1.25,0.25) {$2,4,3$};\node[scale=0.6] at (1.75,0.25) {$2,5,3$};\node[scale=0.6] at (2.25,0.25) {$2,6,3$};\node[scale=0.6] at (2.75,0.25) {$1,6,3$};
\draw  (1.0,0.0) rectangle (1.5,0.5) (2.0,0.0) rectangle (2.5,0.5) (1.5,0.5) rectangle (2.0,1.0) (0.0,1.0) rectangle (0.5,1.5) (-1,1.5) rectangle (-0.5,2.0) (-0.5,1.5) rectangle (0.0,2.0);
\draw[ultra thick](0.5,0)--(0.5,2);
\node at (1,-0.25) {(b) A local embedding of $H_6$ into $[2]\times[6]\times [3]$.};
\end{scope}
\end{tikzpicture}
\end{center}

\begin{center}
\begin{tikzpicture}
\begin{scope}[xshift=0.0cm,yshift=0.0cm]
\draw[step=0.5cm,color=black,fill=gray] (-0.001,-0.001) grid (3.001,1.001) (0.0,0.0) rectangle (0.5,0.5)(0.0,0.5) rectangle (0.5,1.0)(0.5,0.5) rectangle (1.0,1.0)(1.0,0.5) rectangle (1.5,1.0)(1.5,0.5) rectangle (2.0,1.0);
\draw[ultra thick](1,0)--(1,1);
\end{scope}

\begin{scope}[xshift=3.5cm,yshift=0.0cm]
\draw[step=0.5cm,color=black,fill=gray] (-0.001,-0.001) grid (3.001,1.001) (1.0,0.0) rectangle (1.5,0.5)(1.5,0.0) rectangle (2.0,0.5)(2.0,0.0) rectangle (2.5,0.5)(2.5,0.0) rectangle (3.0,0.5)(0.0,0.5) rectangle (0.5,1.0)(0.5,0.5) rectangle (1.0,1.0)(1.0,0.5) rectangle (1.5,1.0)(1.5,0.5) rectangle (2.0,1.0);
\draw[ultra thick](1,0)--(1,1);
\node at (1.5,-0.5) {(c) The image of $H_6$ under the local embedding.};
\end{scope}

\begin{scope}[xshift=7.0cm,yshift=0.0cm]
\draw[step=0.5cm,color=black,fill=gray] (-0.001,-0.001) grid (3.001,1.001) (1.0,0.0) rectangle (1.5,0.5)(1.5,0.0) rectangle (2.0,0.5)(2.0,0.0) rectangle (2.5,0.5)(2.5,0.0) rectangle (3.0,0.5)(2.5,0.5) rectangle (3.0,1.0);
\draw[ultra thick](1,0)--(1,1);
\end{scope}
\end{tikzpicture}
\end{center}

\begin{center}
\begin{tikzpicture}
\begin{scope}[xshift=0.0cm,yshift=0.0cm]
\draw[step=0.5cm,color=black,fill=gray] (-0.001,-0.001) grid (3.001,1.001) (0.0,0.0) rectangle (0.5,0.5)(0.0,0.5) rectangle (0.5,1.0)(0.5,0.5) rectangle (1.0,1.0)(1.0,0.5) rectangle (1.5,1.0)(1.5,0.5) rectangle (2.0,1.0)(2.5,0.5) rectangle (3.0,1.0);
\node at (2.75,0.75) {X};\node at (0.75,0.25) {3};\node at (1.25,0.25) {4};\node at (1.75,0.25) {5};\node at (2.25,0.25) {6};\node at (2.75,0.25) {7};\node at (2.25,0.75) {3};
\draw[ultra thick](1,0)--(1,1);
\end{scope}

\begin{scope}[xshift=3.5cm,yshift=0.0cm]
\draw[step=0.5cm,color=black,fill=gray] (-0.001,-0.001) grid (3.001,1.001) (1.0,0.0) rectangle (1.5,0.5)(1.5,0.0) rectangle (2.0,0.5)(2.0,0.0) rectangle (2.5,0.5)(2.5,0.0) rectangle (3.0,0.5)(0.0,0.5) rectangle (0.5,1.0)(0.5,0.5) rectangle (1.0,1.0)(1.0,0.5) rectangle (1.5,1.0)(1.5,0.5) rectangle (2.0,1.0);
\node at (0.25,0.25) {1};\node at (0.75,0.25) {2};\node at (2.25,0.75) {2};\node at (2.75,0.75) {1};
\draw[ultra thick](1,0)--(1,1);
\node[align=left] at (1.5,-0.75) {(d) The result of the $3$-neighbour process on the image of $H_6$ with the\\ points $(1,6,1)$ and $(2,1,3)$ (each marked with an X) added.};
\end{scope}

\begin{scope}[xshift=7.0cm,yshift=0.0cm]
\draw[step=0.5cm,color=black,fill=gray] (-0.001,-0.001) grid (3.001,1.001) (0.0,0.0) rectangle (0.5,0.5)(1.0,0.0) rectangle (1.5,0.5)(1.5,0.0) rectangle (2.0,0.5)(2.0,0.0) rectangle (2.5,0.5)(2.5,0.0) rectangle (3.0,0.5)(2.5,0.5) rectangle (3.0,1.0);
\node at (0.25,0.25) {X};\node at (0.75,0.25) {3};\node at (0.25,0.75) {7};\node at (0.75,0.75) {6};\node at (1.25,0.75) {5};\node at (1.75,0.75) {4};\node at (2.25,0.75) {3};
\draw[ultra thick](1,0)--(1,1);
\end{scope}
\end{tikzpicture}
\end{center}
    \caption{Diagrams used to verify that $m(2,3,6;3)=12$ in the proof of Proposition~\ref{prop:2,3,0mod6}.}
    \label{fig:236}
\end{figure}

For $k\geq2$, let $H_{6k}$ be the grid graph obtained from $H_6$ by inserting a $4\times 6(k-1)$ grid in the location of the bold vertical line in Figure~\ref{fig:236}~(a). Using Lemma~\ref{lem:fatter}, we place an infection of cardinality $\frac{4(6(k-1)-1) + 4 + (6(k-1)-1)-2}{3} = 10k-11$ in the first $6(k-1)-1$ columns of this inserted grid whose closure leaves behind only the top left and top right corner. We also infect the vertex on the first row of the last column of the inserted $4\times 6(k-1)$ grid. The infection on the rest of $H_{6k}$ is the same as in $H_6$. See Figure~\ref{fig:230mod6} for an illustration of the infection in $H_{12}$ after taking the closure of the infection on the first $6(k-1)-1$ columns of the inserted grid. It is easily observed that, after taking the closure of the infection on the first $6(k-1)-1$ columns of the inserted grid, the complement of the infection consists of trees, each of which contains a unique vertex of degree three in $H_{6k}$. So, by Lemma~\ref{lem:immuneRegions}, this infection percolates in $H_{6k}$. By construction, its cardinality is
\[10k-11 + 1 + 10 = 10k.\]
The local embedding of $H_{6k}$ into $[2]\times[6k]\times [3]$ is analogous to the local embedding of $H_6$ into $[2]\times[6]\times[3]$. Specifically, the vertices on the top row of the inserted $4\times 6(k-1)$ grid are mapped to $(1,3,1),\dots,(1,6(k-1)+2,1)$ from left to right. The vertices on the second, third and fourth rows are mapped to $(1,3,2),\dots,(1,6(k-1)+2,2)$, $(2,3,2),\dots,(2,6(k-1)+2,2)$ and $(2,3,3),\dots,(2,6(k-1)+2,3)$ from left to right, respectively. The vertices to the right of the inserted $4\times 6(k-1)$ grid are mapped to the same vertex as they are under the embedding of $H_6$, except that the second coordinate is increased by $6(k-1)$. The embedding in the case $k=2$ is provided in Figure~\ref{fig:230mod6}~(b) for reference. As in the case $k=1$, by adding two additional infected points, namely $(1,6k,1)$ and $(2,1,3)$, to the image of $H_{6k}$ under the local embedding, we obtain a set which percolates under the $3$-neighbour process in $[2]\times [6k]\times [3]$. To formally prove this, one can observe that, after three steps of the process, there are no paths between the two faces in direction $j$ for any $j\in[3]$, and so percolation occurs by Lemma~\ref{lem:sidesSplit}. An illustration in the case $k=2$ is given in Figure~\ref{fig:230mod6}~(c). Thus, $m(2,6k,3;3) \leq 10k+2 = \frac{2\cdot 6k + 2\cdot 3 + 3\cdot 6k}{3}$, as desired.
\end{proof}

\begin{figure}[htbp]
    \begin{center}
\begin{tikzpicture}
\begin{scope}[xshift=0.0cm,yshift=0.0cm,scale=1.2]
\draw[step=0.5cm,color=black,fill=gray]  (0.4999,-0.001) grid (3.501,2.001)(3.5,0.0) rectangle (4.0,0.5)(4.5,0.0) rectangle (5.0,0.5)(4.0,0.5) rectangle (4.5,1.0)(3.5,1.0) rectangle (4.0,1.5)(0.0,1.5) rectangle (0.5,2.0)(4.0,1.5) rectangle (4.5,2.0)(-0.5,1.0) rectangle (0.0,1.5) (-1.0,1.5) rectangle (-0.5,2.0) (5.0,0.5) rectangle (5.5,1.0) (5.5,0.0) rectangle (6.0,0.5)(0.5,0.0) rectangle (1,0.5)(1.5,0.0) rectangle (2,0.5)(2.5,0.0) rectangle (3,0.5)(1,0.5) rectangle (1.5,1)(2,0.5) rectangle (2.5,1)(0.5,1) rectangle (1,1.5)(2.5,1) rectangle (3,1.5)(1,1.5) rectangle (1.5,2)(2,1.5) rectangle (2.5,2)(3,1.5) rectangle (3.5,2)(1.5,1.5) rectangle (2,2)(0.5,0) rectangle (3.0,1.5);
\draw   (4.0,0.0) rectangle (4.5,0.5) (5.0,0.0) rectangle (5.5,0.5) (4.5,0.5) rectangle (5.0,1.0) (0.0,1.0) rectangle (0.5,1.5) (-0.5,1.5) rectangle (0.0,2.0)(3.5,1.5) rectangle (4,2)(4,1) rectangle (4.5,1.5)(3.5,0.5) rectangle (4,1);
\draw[ultra thick](0.5,0)--(0.5,2);
\draw[ultra thick](3.5,0)--(3.5,2);
\node[align=left] at (2.5,-0.5) {(a) $H_{12}$ with the infection after taking the closure of the infection in the first five\\columns of the inserted $4\times 6$ grid.};
\end{scope}
\end{tikzpicture}
\end{center}

\begin{center}
\begin{tikzpicture}
\begin{scope}[xshift=0.0cm,yshift=0.0cm,scale=1.2]
\draw[step=0.5cm,color=black]  (0.4999,-0.001) grid (3.501,2.001)(3.5,0.0) rectangle (4.0,0.5)(4.5,0.0) rectangle (5.0,0.5)(4.0,0.5) rectangle (4.5,1.0)(3.5,1.0) rectangle (4.0,1.5)(0.0,1.5) rectangle (0.5,2.0)(4.0,1.5) rectangle (4.5,2.0)(-0.5,1.0) rectangle (0.0,1.5) (-1.0,1.5) rectangle (-0.5,2.0) (5.0,0.5) rectangle (5.5,1.0) (5.5,0.0) rectangle (6.0,0.5)(0.5,0.0) rectangle (1,0.5)(1.5,0.0) rectangle (2,0.5)(2.5,0.0) rectangle (3,0.5)(1,0.5) rectangle (1.5,1)(2,0.5) rectangle (2.5,1)(0.5,1) rectangle (1,1.5)(2.5,1) rectangle (3,1.5)(1,1.5) rectangle (1.5,2)(2,1.5) rectangle (2.5,2)(3,1.5) rectangle (3.5,2);
\draw  (4.0,0.0) rectangle (4.5,0.5) (5.0,0.0) rectangle (5.5,0.5) (4.5,0.5) rectangle (5.0,1.0) (0.0,1.0) rectangle (0.5,1.5) (-0.5,1.5) rectangle (0.0,2.0)(3.5,1.5) rectangle (4,2)(4,1) rectangle (4.5,1.5)(3.5,0.5) rectangle (4,1);
\node[scale=0.4] at (-0.25,1.25) {$1,1,2$};\node[scale=0.4] at (0.25,1.25) {$1,2,2$};\node[scale=0.4] at (3.75,1.25) {$1,9,2$};\node[scale=0.4] at (4.25,1.25) {$1,10,2$};\node[scale=0.4] at (-0.75,1.75) {$2,1,1$};\node[scale=0.4] at (-0.25,1.75) {$1,1,1$};\node[scale=0.4] at (0.25,1.75) {$1,2,1$};\node[scale=0.4] at (0.75,1.75) {$1,3,1$};\node[scale=0.4] at (1.25,1.75) {$1,4,1$};\node[scale=0.4] at (1.75,1.75) {$1,5,1$};\node[scale=0.4] at (2.25,1.75) {$1,6,1$};\node[scale=0.4] at (2.75,1.75) {$1,7,1$};\node[scale=0.4] at (3.25,1.75) {$1,8,1$};\node[scale=0.4] at (0.75,1.25) {$1,3,2$};\node[scale=0.4] at (1.25,1.25) {$1,4,2$};\node[scale=0.4] at (1.75,1.25) {$1,5,2$};\node[scale=0.4] at (2.25,1.25) {$1,6,2$};\node[scale=0.4] at (2.75,1.25) {$1,7,2$};\node[scale=0.4] at (3.25,1.25) {$1,8,2$};\node[scale=0.4] at (0.75,0.75) {$2,3,2$};\node[scale=0.4] at (1.25,0.75) {$2,4,2$};\node[scale=0.4] at (1.75,0.75) {$2,5,2$};\node[scale=0.4] at (2.25,0.75) {$2,6,2$};\node[scale=0.4] at (2.75,0.75) {$2,7,2$};\node[scale=0.4] at (3.25,0.75) {$2,8,2$};\node[scale=0.4] at (0.75,0.25) {$2,3,3$};\node[scale=0.4] at (1.25,0.25) {$2,4,3$};\node[scale=0.4] at (1.75,0.25) {$2,5,3$};\node[scale=0.4] at (2.25,0.25) {$2,6,3$};\node[scale=0.4] at (2.75,0.25) {$2,7,3$};\node[scale=0.4] at (3.25,0.25) {$2,8,3$};\node[scale=0.4] at (3.75,1.75) {$1,9,1$};\node[scale=0.4] at (4.25,1.75) {$1,10,1$};\node[scale=0.4] at (3.75,0.75) {$2,9,2$};\node[scale=0.4] at (4.25,0.75) {$2,10,2$};\node[scale=0.4] at (4.75,0.75) {$2,11,2$};\node[scale=0.4] at (5.25,0.75) {$2,12,2$};\node[scale=0.4] at (3.75,0.25) {$2,9,3$};\node[scale=0.4] at (4.25,0.25) {$2,10,3$};\node[scale=0.4] at (4.75,0.25) {$2,11,3$};\node[scale=0.4] at (5.25,0.25) {$2,12,3$};\node[scale=0.4] at (5.75,0.25) {$1,12,3$};
\draw[ultra thick](0.5,0)--(0.5,2);
\draw[ultra thick](3.5,0)--(3.5,2);
\node at (2.5,-0.25) {(b) A local embedding of $H_{12}$ into $[2]\times[12]\times[3]$.};
\end{scope}
\end{tikzpicture}
\end{center}

\begin{center}
\begin{tikzpicture}
\begin{scope}[xshift=0.0cm,yshift=0.0cm, scale=0.7]
\draw[step=0.5cm,color=black,fill=gray] (-0.001,-0.001) grid (6.001,1.001) (0.0,0.0) rectangle (0.5,0.5)(0.0,0.5) rectangle (0.5,1.0)(0.5,0.5) rectangle (1.0,1.0)(1.0,0.5) rectangle (1.5,1.0)(1.5,0.5) rectangle (2.0,1.0)(2.0,0.5) rectangle (2.5,1.0)(2.5,0.5) rectangle (3.0,1.0)(3.0,0.5) rectangle (3.5,1.0)(3.5,0.5) rectangle (4.0,1.0)(4.0,0.5) rectangle (4.5,1.0)(4.5,0.5) rectangle (5.0,1.0)(5.5,0.5) rectangle (6.0,1.0);
\node[scale=0.7] at (0.75,0.25) {3};\node[scale=0.7] at (1.25,0.25) {4};\node[scale=0.7] at (1.75,0.25) {5};\node[scale=0.7] at (2.25,0.25) {6};\node[scale=0.7] at (2.75,0.25) {7};\node[scale=0.7] at (3.25,0.25) {8};\node[scale=0.7] at (3.75,0.25) {9};\node[scale=0.7] at (4.25,0.25) {10};\node[scale=0.7] at (4.75,0.25) {11};\node[scale=0.7] at (5.25,0.25) {12};\node[scale=0.7] at (5.75,0.25) {13};\node[scale=0.7] at (5.25,0.75) {3};\node[scale=0.7] at (5.75,0.75) {X};
\draw[ultra thick](1,0)--(1,1);
\draw[ultra thick](4,0)--(4,1);
\end{scope}

\begin{scope}[xshift=4.55cm,yshift=0.0cm, scale=0.7]
\draw[step=0.5cm,color=black,fill=gray] (-0.001,-0.001) grid (6.001,1.001) (1.0,0.0) rectangle (1.5,0.5)(1.5,0.0) rectangle (2.0,0.5)(2.0,0.0) rectangle (2.5,0.5)(2.5,0.0) rectangle (3.0,0.5)(3.0,0.0) rectangle (3.5,0.5)(3.5,0.0) rectangle (4.0,0.5)(4.0,0.0) rectangle (4.5,0.5)(4.5,0.0) rectangle (5.0,0.5)(5.0,0.0) rectangle (5.5,0.5)(5.5,0.0) rectangle (6.0,0.5)(0.0,0.5) rectangle (0.5,1.0)(0.5,0.5) rectangle (1.0,1.0)(1.0,0.5) rectangle (1.5,1.0)(1.5,0.5) rectangle (2.0,1.0)(2.0,0.5) rectangle (2.5,1.0)(2.5,0.5) rectangle (3.0,1.0)(3.0,0.5) rectangle (3.5,1.0)(3.5,0.5) rectangle (4.0,1.0)(4.0,0.5) rectangle (4.5,1.0)(4.5,0.5) rectangle (5.0,1.0);
\node[scale=0.7] at (0.25,0.25) {1};\node[scale=0.7] at (0.75,0.25) {2};\node[scale=0.7] at (5.25,0.75) {2};\node[scale=0.7] at (5.75,0.75) {1};
\draw[ultra thick](1,0)--(1,1);
\draw[ultra thick](4,0)--(4,1);
\end{scope}

\begin{scope}[xshift=0.0cm,yshift=-1.05cm, scale=0.7]
\draw[step=0.5cm,color=black,fill=gray] (-0.001,-0.001) grid (6.001,1.001) (0.0,0.0) rectangle (0.5,0.5)(1.0,0.0) rectangle (1.5,0.5)(1.5,0.0) rectangle (2.0,0.5)(2.0,0.0) rectangle (2.5,0.5)(2.5,0.0) rectangle (3.0,0.5)(3.0,0.0) rectangle (3.5,0.5)(3.5,0.0) rectangle (4.0,0.5)(4.0,0.0) rectangle (4.5,0.5)(4.5,0.0) rectangle (5.0,0.5)(5.0,0.0) rectangle (5.5,0.5)(5.5,0.0) rectangle (6.0,0.5)(5.5,0.5) rectangle (6.0,1.0);
\node[scale=0.7] at (0.75,0.25) {3};\node[scale=0.7] at (0.25,0.75) {13};\node[scale=0.7] at (0.75,0.75) {12};\node[scale=0.7] at (1.25,0.75) {11};\node[scale=0.7] at (1.75,0.75) {10};\node[scale=0.7] at (2.25,0.75) {9};\node[scale=0.7] at (2.75,0.75) {8};\node[scale=0.7] at (3.25,0.75) {7};\node[scale=0.7] at (3.75,0.75) {6};\node[scale=0.7] at (4.25,0.75) {5};\node[scale=0.7] at (4.75,0.75) {4};\node[scale=0.7] at (5.25,0.75) {3};\node[scale=0.7] at (0.25,0.25) {X};
\draw[ultra thick](1,0)--(1,1);
\draw[ultra thick](4,0)--(4,1);
\node[align=left] at (6.25,-0.75) {(c) The result of the $3$-neighbour process on the image of $H_{12}$ with the\\ points $(1,12,1)$ and $(2,1,3)$ (each marked with an X) added.};
\end{scope}

\end{tikzpicture}
\end{center}
    \caption{Diagrams used to verify that $m(2,3,6k;3)=10k+2$ for $k\geq2$ in the proof of Proposition~\ref{prop:2,3,0mod6}.}
    \label{fig:230mod6}
\end{figure}

\begin{prop}
\label{prop:2,3,3mod6}
If $a_3\geq 9$ and $a_3\equiv 3\bmod 6$, then $(2,3,a_3)$ is perfect. 
\end{prop}

\begin{proof}
In all diagrams in this proof, we view the grid as $[3]\times[6k+3]\times [2]$ for $k\geq1$. Let $H_9$ be the grid graph depicted in Figure~\ref{fig:239}~(a). Between the two bold vertical lines is a $4\times 6$ grid with the top right corner removed. For $k\geq2$, let $H_{6k+3}$ be the grid graph obtained from $H_9$ by replacing this $4\times 6$ grid with the top right corner removed with a $4\times 6k$ grid with the top right corner removed. 

By Lemma~\ref{lem:fatter}, we can infect a set of $\frac{4(6k-1) + 4 + (6k-1)-2}{3} = 10k-1$ points within a $4\times 6k$ grid whose closure contains all elements of the grid except for those in the first column, the top vertex in the second column and the top right vertex. We infect those $10k-1$ elements in the copy of the $4\times6k$ grid with the top corner removed within $H_{6k+3}$. Also, we infect the vertex on the first row and column of this particular grid graph within $H_{6k+3}$. The infection on the elements of $H_{6k+3}$ outside of this $4\times 6k$ grid are depicted in Figure~\ref{fig:239}~(a). By Lemma~\ref{lem:immuneRegions}, this infection spreads to every element of $H_{6k+3}$. Figure~\ref{fig:239}~(b) describes a local embedding of $H_9$ into $[3]\times[9]\times[2]$; the embedding of $H_{6k+3}$ for general $k$ is analogous. Figure~\ref{fig:239}~(c) verifies that the image of all vertices of $H_9$ under this local embedding percolates under the $3$-neighbour process. In general, the argument goes as follows. Each vertex of the form $(2,i,2)$ for $2\leq i\leq 6k$ is adjacent to  $(2,i,1)$ and $(1,i,2)$, each of which is in the image of $H_{6k+3}$ under the local embedding. Additionally, $(2,2,2)$ is adjacent to $(2,1,2)$, which is also in the local embedding. Using this, we see that every vertex of the form $(2,i,2)$ for $2\leq i\leq 6k$ becomes infected. Also, $(2,6k+2,2)$ becomes infected due to it being adjacent to $(2,6k+3,2),(3,6k+2,2)$ and $(2,6k+2,1)$, all of which are in the image of $H_{6k+3}$. Now that $(2,6k,2)$ and $(2,6k+2,2)$ are infected, $(2,6k+1,2)$ becomes infected as well, due to its adjacency to $(2,6k+1,1)$. At this point, we can invoke Lemma~\ref{lem:sidesSplit} to get that percolation occurs. Thus, 
\[m(2,3,6k+3;3)\leq 10k-1 + 1 + 7 = 10k+7 = \frac{2\cdot (6k+3) + 2\cdot 3 + 3\cdot (6k+3)}{3}\]
as desired.

\begin{figure}[htbp]
\begin{center}
\begin{tikzpicture}
\begin{scope}[xshift=0.0cm,yshift=0.0cm,scale=1.2]
\draw[step=0.5cm,color=black,fill=gray]  (4.0,0.0) rectangle (4.5,0.5)(0.5,0.5) rectangle (1.0,1.0)(1.5,0.5) rectangle (2.0,1.0)(2.5,0.5) rectangle (3.0,1.0)(3.5,0.5) rectangle (4.0,1.0)(4.5,0.5) rectangle (5.0,1.0)(0.0,1.0) rectangle (0.5,1.5)(2.0,1.0) rectangle (2.5,1.5)(3.0,1.0) rectangle (3.5,1.5)(5.0,1.0) rectangle (5.5,1.5)(0.5,1.5) rectangle (1.0,2.0)(1.5,1.5) rectangle (2.0,2.0)(3.5,1.5) rectangle (4.0,2.0)(4.5,1.5) rectangle (5.0,2.0)(1.0,2.0) rectangle (1.5,2.5)(2.0,2.0) rectangle (2.5,2.5)(3.0,2.0) rectangle (3.5,2.5)(1.5,0.5) rectangle (4,2) (2.5,2) rectangle (3,2.5);
\draw[step=0.5cm,color=black](1.0,0.5) rectangle(1.5,1)(2.0,0.5) rectangle(2.5,1.0)(3.0,0.5) rectangle(3.5,1.0)(4.0,0.5) rectangle(4.5,1.0)(0.5,1.0) rectangle(1.0,1.5)(1.0,1.0) rectangle(1.5,1.5)(1.5,1.0) rectangle(2,1.5)(2.5,1.0) rectangle(3,1.5)(3.5,1.0) rectangle(4,1.5)(4,1.0) rectangle(4.5,1.5)(4.5,1.0) rectangle(5,1.5)(1,1.5) rectangle(1.5,2)(2,1.5) rectangle(2.5,2)(2.5,1.5) rectangle(3,2)(3,1.5) rectangle(3.5,2)(4,1.5) rectangle(4.5,2)(1.5,2) rectangle(2,2.5)(2.5,2) rectangle(3,2.5) ;
\draw[ultra thick](1,0.5)--(1,2.5);
\draw[ultra thick](4,0.5)--(4,2.0);
\node[align=left] at (2.75,-0.5){(a) $H_{9}$ with the resulting infection after taking the closure of the infection in the last\\five columns of the grid graph between the bold vertical lines.};
\end{scope}
\end{tikzpicture}
\end{center}

\begin{center}
\begin{tikzpicture}
\begin{scope}[xshift=0.0cm,yshift=0.0cm,scale=1.2]
\draw[step=0.5cm,color=black]  (4.0,0.0) rectangle (4.5,0.5)(0.5,0.5) rectangle (1.0,1.0)(1.5,0.5) rectangle (2.0,1.0)(2.5,0.5) rectangle (3.0,1.0)(3.5,0.5) rectangle (4.0,1.0)(4.5,0.5) rectangle (5.0,1.0)(0.0,1.0) rectangle (0.5,1.5)(2.0,1.0) rectangle (2.5,1.5)(3.0,1.0) rectangle (3.5,1.5)(5.0,1.0) rectangle (5.5,1.5)(0.5,1.5) rectangle (1.0,2.0)(1.5,1.5) rectangle (2.0,2.0)(3.5,1.5) rectangle (4.0,2.0)(4.5,1.5) rectangle (5.0,2.0)(1.0,2.0) rectangle (1.5,2.5)(2.0,2.0) rectangle (2.5,2.5)(3.0,2.0) rectangle (3.5,2.5);
\draw[step=0.5cm,color=black](1.0,0.5) rectangle(1.5,1)(2.0,0.5) rectangle(2.5,1.0)(3.0,0.5) rectangle(3.5,1.0)(4.0,0.5) rectangle(4.5,1.0)(0.5,1.0) rectangle(1.0,1.5)(1.0,1.0) rectangle(1.5,1.5)(1.5,1.0) rectangle(2,1.5)(2.5,1.0) rectangle(3,1.5)(3.5,1.0) rectangle(4,1.5)(4,1.0) rectangle(4.5,1.5)(4.5,1.0) rectangle(5,1.5)(1,1.5) rectangle(1.5,2)(2,1.5) rectangle(2.5,2)(2.5,1.5) rectangle(3,2)(3,1.5) rectangle(3.5,2)(4,1.5) rectangle(4.5,2)(1.5,2) rectangle(2,2.5)(2.5,2) rectangle(3,2.5) ;
\node[scale=0.5] at (1.25,2.25) {$1,2,2$};\node[scale=0.5] at (1.75,2.25) {$1,3,2$};\node[scale=0.5] at (2.25,2.25) {$1,4,2$};\node[scale=0.5] at (2.75,2.25) {$1,5,2$};\node[scale=0.5] at (3.25,2.25) {$1,6,2$};
\node[scale=0.5] at (0.75,1.75) {$1,1,1$};\node[scale=0.5] at (1.25,1.75) {$1,2,1$};\node[scale=0.5] at (1.75,1.75) {$1,3,1$};\node[scale=0.5] at (2.25,1.75) {$1,4,1$};\node[scale=0.5] at (2.75,1.75) {$1,5,1$};\node[scale=0.5] at (3.25,1.75) {$1,6,1$};\node[scale=0.5] at (3.75,1.75) {$1,7,1$};\node[scale=0.5] at (4.25,1.75) {$1,8,1$};\node[scale=0.5] at (4.75,1.75) {$1,9,1$};
\node[scale=0.5] at (0.25,1.25) {$2,1,2$};\node[scale=0.5] at (0.75,1.25) {$2,1,1$};\node[scale=0.5] at (1.25,1.25) {$2,2,1$};\node[scale=0.5] at (1.75,1.25) {$2,3,1$};\node[scale=0.5] at (2.25,1.25) {$2,4,1$};\node[scale=0.5] at (2.75,1.25) {$2,5,1$};\node[scale=0.5] at (3.25,1.25) {$2,6,1$};\node[scale=0.5] at (3.75,1.25) {$2,7,1$};\node[scale=0.5] at (4.25,1.25) {$2,8,1$};\node[scale=0.5] at (4.75,1.25) {$2,9,1$};\node[scale=0.5] at (5.25,1.25) {$2,9,2$};
\node[scale=0.5] at (0.75,0.75) {$3,1,1$};\node[scale=0.5] at (1.25,0.75) {$3,2,1$};\node[scale=0.5] at (1.75,0.75) {$3,3,1$};\node[scale=0.5] at (2.25,0.75) {$3,4,1$};\node[scale=0.5] at (2.75,0.75) {$3,5,1$};\node[scale=0.5] at (3.25,0.75) {$3,6,1$};\node[scale=0.5] at (3.75,0.75) {$3,7,1$};\node[scale=0.5] at (4.25,0.75) {$3,8,1$};\node[scale=0.5] at (4.75,0.75) {$3,9,1$};\node[scale=0.5] at (4.25,0.25) {$3,8,2$};
\draw[ultra thick](1,0.5)--(1,2.5);
\draw[ultra thick](4,0.5)--(4,2.0);
\node at (2.75,-0.5){(b) A local embedding of $H_9$ into $[3]\times[9]\times[2]$.};
\end{scope}
\end{tikzpicture}
\end{center}

\begin{center}
\begin{tikzpicture}
\begin{scope}[xshift=0.0cm,yshift=0.0cm,scale=0.7]
\draw[step=0.5cm,color=black,fill=gray] (-0.001,-0.001) grid (4.501,1.501) (0.0,0.0) rectangle (0.5,0.5)(0.5,0.0) rectangle (1.0,0.5)(1.0,0.0) rectangle (1.5,0.5)(1.5,0.0) rectangle (2.0,0.5)(2.0,0.0) rectangle (2.5,0.5)(2.5,0.0) rectangle (3.0,0.5)(3.0,0.0) rectangle (3.5,0.5)(3.5,0.0) rectangle (4.0,0.5)(4.0,0.0) rectangle (4.5,0.5)(0.0,0.5) rectangle (0.5,1.0)(0.5,0.5) rectangle (1.0,1.0)(1.0,0.5) rectangle (1.5,1.0)(1.5,0.5) rectangle (2.0,1.0)(2.0,0.5) rectangle (2.5,1.0)(2.5,0.5) rectangle (3.0,1.0)(3.0,0.5) rectangle (3.5,1.0)(3.5,0.5) rectangle (4.0,1.0)(4.0,0.5) rectangle (4.5,1.0)(0.0,1.0) rectangle (0.5,1.5)(0.5,1.0) rectangle (1.0,1.5)(1.0,1.0) rectangle (1.5,1.5)(1.5,1.0) rectangle (2.0,1.5)(2.0,1.0) rectangle (2.5,1.5)(2.5,1.0) rectangle (3.0,1.5)(3.0,1.0) rectangle (3.5,1.5)(3.5,1.0) rectangle (4.0,1.5)(4.0,1.0) rectangle (4.5,1.5);
\draw[ultra thick](0.5,0)--(0.5,1.5);
\draw[ultra thick](3.5,0)--(3.5,1.5);
\end{scope}

\begin{scope}[xshift=5.0cm,yshift=0.0cm,scale=0.7]
\draw[step=0.5cm,color=black,fill=gray] (-0.001,-0.001) grid (4.501,1.501) (3.5,0.0) rectangle (4.0,0.5)(0.0,0.5) rectangle (0.5,1.0)(4.0,0.5) rectangle (4.5,1.0)(0.5,1.0) rectangle (1.0,1.5)(1.0,1.0) rectangle (1.5,1.5)(1.5,1.0) rectangle (2.0,1.5)(2.0,1.0) rectangle (2.5,1.5)(2.5,1.0) rectangle (3.0,1.5);
\node[scale=0.7] at (0.25,0.25) {13};\node[scale=0.7] at (0.75,0.25) {12};\node[scale=0.7] at (1.25,0.25) {11};\node[scale=0.7] at (1.75,0.25) {10};\node[scale=0.7] at (2.25,0.25) {9};\node[scale=0.7] at (2.75,0.25) {8};\node[scale=0.7] at (3.25,0.25) {7};\node[scale=0.7] at (4.25,0.25) {1};\node[scale=0.7] at (0.75,0.75) {1};\node[scale=0.7] at (1.25,0.75) {2};\node[scale=0.7] at (1.75,0.75) {3};\node[scale=0.7] at (2.25,0.75) {4};\node[scale=0.7] at (2.75,0.75) {5};\node[scale=0.7] at (3.25,0.75) {6};\node[scale=0.7] at (3.75,0.75) {1};\node[scale=0.7] at (0.25,1.25) {1};\node[scale=0.7] at (3.25,1.25) {7};\node[scale=0.7] at (3.75,1.25) {8};\node[scale=0.7] at (4.25,1.25) {9};
\draw[ultra thick](0.5,0)--(0.5,1.5);
\draw[ultra thick](3.5,0)--(3.5,1.5);
\end{scope}

\node at (4.75,-0.5){(c) The result of the $3$-neighbour process on the image of $H_9$.};
\end{tikzpicture}
\end{center}

    \caption{Diagrams used in the proof of Proposition~\ref{prop:2,3,3mod6}.}
    \label{fig:239}
\end{figure}
\end{proof}

An exhaustive computer search easily concludes that $(2,3,3)$ is not perfect. Therefore, the condition that $a_3$ is larger than $3$ is necessary in Proposition~\ref{prop:2,3,3mod6}. The following Corollary summarizes Propositions~\ref{prop:2,3,0mod6} and~\ref{prop:2,3,3mod6}. 

\begin{cor}
\label{cor:2,3,0mod3}
$(2,3,a_3)$ is perfect if and only if  $a_3\equiv 0\bmod 3$ and $a_3\geq 6$.
\end{cor}

\begin{proof}
If $a_3\not\equiv 0\bmod 3$, then $(2,3,a_3)$ is not class $0$ and therefore cannot be perfect. Also, as discussed above, $(2,3,3)$ is not perfect. The result now follows by Propositions~\ref{prop:2,3,0mod6} and~\ref{prop:2,3,3mod6}.
\end{proof}

\begin{prop}
\label{prop:2,2mod3,2mod3}
Let $a_2,a_3\geq 5$. If $a_2,a_3\equiv 2\bmod 3$ and $a_2\not\equiv a_3\bmod 6$, then $(2,a_2,a_3)$ is perfect. 
\end{prop}

\begin{proof}
Without loss of generality, we assume that $a_2\equiv 5\bmod 6$ and $a_3\equiv 2\bmod 6$. In our diagrams, we view the grid as $[a_2]\times[a_3]\times[2]$ for convenience. Let $H$ be the grid graph obtained from a $(a_2+1)\times (a_3\times 1)$ grid by adding a cell to the left of the bottom left cell and deleting the bottom five cells in the rightmost column. By applying Lemma~\ref{lem:longerCorner2}, and adding an additional infection on the added cell, we get a percolating set $A_0$ in $H$ of cardinality $\frac{(a_2+1)(a_3+1)+(a_2+1)+(a_3+1)-3}{3} = \frac{a_2a_3+2a_2+2a_3}{3}$. See Figure~\ref{fig:22mod32mod3}~(a) for an illustration in the case $a_2=11$ and $a_3=8$. 

We define a local embedding of $H$ into $[a_2]\times[a_3]\times[2]$ as follows. For $1\leq i\leq a_2-4$, the first $a_3$ vertices on the $i$th row are mapped to $(i,1,1),(i,2,1),\dots, (i,a_3,1)$ from left to right and the last vertex on the $i$th row is mapped to $(i,a_3,2)$. For $i\in\{a_2-3,a_2-2\}$, the verties on the $i$th row are mapped to $(i,1,1),(i,2,1),\dots, (i,a_3,1)$ from left to right. For $a_2-1\leq i\leq a_2+1$, the vertices on the $i$th row, excluding the extra vertex added on the last row, are mapped to $(i-1,1,2),(i-1,2,2),\dots, (i-1,a_3,2)$. Finally, the extra vertex on the last row is mapped to $(a_2,1,1)$. See Figure~\ref{fig:22mod32mod3}~(b) and (c) for an illustration of the embedding and the image of $H$ under that embedding, respectively, in the case $a_2=11$ and $a_3=8$. Starting with the image of $H$ fully infected, the vertices $(a_2-1,1,1)$ and $(a_2-3,a_3,2)$ each have three infected neighbours. Thus, they become infected after one additional step of the $3$-neighbour process; each of these vertices is marked with an X in Figure~\ref{fig:22mod32mod3}~(c). Upon infecting these two vertices, we see that there is no path between the two faces in any of the three possible dimensions which avoids the current infection. Thus, percolation occurs by Lemma~\ref{lem:sidesSplit}. We conclude that $m(2,a_2,a_3;3)\leq \frac{a_2a_3+2a_2+2a_3}{3}$ for all such $a_2$ and $a_3$, which is exactly what we set out to prove. 

\begin{figure}[htbp]
\begin{center}
\begin{tikzpicture}[scale=1.2]
\begin{scope}[xshift=0.0cm,yshift=0.0cm]
\draw[step=0.5cm,color=black,fill=gray] (0.4999,-0.001) grid (1.001,0.501)(0.0,0.0) rectangle (0.5,0.5)(1.0,0.0) rectangle (1.5,0.5)(1.5,0.0) rectangle (2.0,0.5)(2.0,0.0) rectangle (2.5,0.5)(2.5,0.0) rectangle (3.0,0.5)(3.0,0.0) rectangle (3.5,0.5)(3.5,0.0) rectangle (4.0,0.5)(4.0,0.0) rectangle (4.5,0.5)(0.5,0.5) rectangle (1.0,1.0)(1.0,0.5) rectangle (1.5,1.0)(1.5,0.5) rectangle (2.0,1.0)(2.0,0.5) rectangle (2.5,1.0)(2.5,0.5) rectangle (3.0,1.0)(3.0,0.5) rectangle (3.5,1.0)(3.5,0.5) rectangle (4.0,1.0)(4.0,0.5) rectangle (4.5,1.0)(0.5,1.0) rectangle (1.0,1.5)(1.0,1.0) rectangle (1.5,1.5)(1.5,1.0) rectangle (2.0,1.5)(2.0,1.0) rectangle (2.5,1.5)(2.5,1.0) rectangle (3.0,1.5)(3.0,1.0) rectangle (3.5,1.5)(3.5,1.0) rectangle (4.0,1.5)(4.0,1.0) rectangle (4.5,1.5)(0.5,1.5) rectangle (1.0,2.0)(1.0,1.5) rectangle (1.5,2.0)(1.5,1.5) rectangle (2.0,2.0)(2.0,1.5) rectangle (2.5,2.0)(2.5,1.5) rectangle (3.0,2.0)(3.0,1.5) rectangle (3.5,2.0)(3.5,1.5) rectangle (4.0,2.0)(4.0,1.5) rectangle (4.5,2.0)(0.5,2.0) rectangle (1.0,2.5)(1.0,2.0) rectangle (1.5,2.5)(1.5,2.0) rectangle (2.0,2.5)(2.0,2.0) rectangle (2.5,2.5)(2.5,2.0) rectangle (3.0,2.5)(3.0,2.0) rectangle (3.5,2.5)(3.5,2.0) rectangle (4.0,2.5)(4.0,2.0) rectangle (4.5,2.5)(0.5,2.5) rectangle (1.0,3.0)(1.0,2.5) rectangle (1.5,3.0)(1.5,2.5) rectangle (2.0,3.0)(2.0,2.5) rectangle (2.5,3.0)(2.5,2.5) rectangle (3.0,3.0)(3.0,2.5) rectangle (3.5,3.0)(3.5,2.5) rectangle (4.0,3.0)(4.0,2.5) rectangle (4.5,3.0)(4.5,2.5) rectangle (5.0,3.0)(0.5,3.0) rectangle (1.0,3.5)(1.0,3.0) rectangle (1.5,3.5)(1.5,3.0) rectangle (2.0,3.5)(2.0,3.0) rectangle (2.5,3.5)(2.5,3.0) rectangle (3.0,3.5)(3.0,3.0) rectangle (3.5,3.5)(3.5,3.0) rectangle (4.0,3.5)(4.0,3.0) rectangle (4.5,3.5)(4.5,3.0) rectangle (5.0,3.5)(0.5,3.5) rectangle (1.0,4.0)(1.0,3.5) rectangle (1.5,4.0)(1.5,3.5) rectangle (2.0,4.0)(2.0,3.5) rectangle (2.5,4.0)(2.5,3.5) rectangle (3.0,4.0)(3.0,3.5) rectangle (3.5,4.0)(3.5,3.5) rectangle (4.0,4.0)(4.0,3.5) rectangle (4.5,4.0)(4.5,3.5) rectangle (5.0,4.0)(0.5,4.0) rectangle (1.0,4.5)(1.0,4.0) rectangle (1.5,4.5)(1.5,4.0) rectangle (2.0,4.5)(2.0,4.0) rectangle (2.5,4.5)(2.5,4.0) rectangle (3.0,4.5)(3.0,4.0) rectangle (3.5,4.5)(3.5,4.0) rectangle (4.0,4.5)(4.0,4.0) rectangle (4.5,4.5)(4.5,4.0) rectangle (5.0,4.5)(0.5,4.5) rectangle (1.0,5.0)(1.0,4.5) rectangle (1.5,5.0)(1.5,4.5) rectangle (2.0,5.0)(2.0,4.5) rectangle (2.5,5.0)(2.5,4.5) rectangle (3.0,5.0)(3.0,4.5) rectangle (3.5,5.0)(3.5,4.5) rectangle (4.0,5.0)(4.0,4.5) rectangle (4.5,5.0)(4.5,4.5) rectangle (5.0,5.0)(0.5,5.0) rectangle (1.0,5.5)(1.0,5.0) rectangle (1.5,5.5)(1.5,5.0) rectangle (2.0,5.5)(2.0,5.0) rectangle (2.5,5.5)(2.5,5.0) rectangle (3.0,5.5)(3.0,5.0) rectangle (3.5,5.5)(3.5,5.0) rectangle (4.0,5.5)(4.0,5.0) rectangle (4.5,5.5)(4.5,5.0) rectangle (5.0,5.5)(0.5,5.5) rectangle (1.0,6.0)(1.0,5.5) rectangle (1.5,6.0)(1.5,5.5) rectangle (2.0,6.0)(2.0,5.5) rectangle (2.5,6.0)(2.5,5.5) rectangle (3.0,6.0)(3.0,5.5) rectangle (3.5,6.0)(3.5,5.5) rectangle (4.0,6.0)(4.0,5.5) rectangle (4.5,6.0)(4.5,5.5) rectangle (5.0,6.0);

\node[align=left] at (2.5,-0.75){(a) The grid graph $H$ with an\\infection obtained from Lemma~\ref{lem:longerCorner2}\\in the case $a_2=11$ and $a_3=8$};
\end{scope}

\begin{scope}[xshift=6.0cm,yshift=0.0cm]
\draw[step=0.5cm,color=black] (0.4999,-0.001) grid (1.001,0.501) (0.0,0.0) rectangle (0.5,0.5)(1.0,0.0) rectangle (1.5,0.5)(1.5,0.0) rectangle (2.0,0.5)(2.0,0.0) rectangle (2.5,0.5)(2.5,0.0) rectangle (3.0,0.5)(3.0,0.0) rectangle (3.5,0.5)(3.5,0.0) rectangle (4.0,0.5)(4.0,0.0) rectangle (4.5,0.5)(0.5,0.5) rectangle (1.0,1.0)(1.0,0.5) rectangle (1.5,1.0)(1.5,0.5) rectangle (2.0,1.0)(2.0,0.5) rectangle (2.5,1.0)(2.5,0.5) rectangle (3.0,1.0)(3.0,0.5) rectangle (3.5,1.0)(3.5,0.5) rectangle (4.0,1.0)(4.0,0.5) rectangle (4.5,1.0)(0.5,1.0) rectangle (1.0,1.5)(1.0,1.0) rectangle (1.5,1.5)(1.5,1.0) rectangle (2.0,1.5)(2.0,1.0) rectangle (2.5,1.5)(2.5,1.0) rectangle (3.0,1.5)(3.0,1.0) rectangle (3.5,1.5)(3.5,1.0) rectangle (4.0,1.5)(4.0,1.0) rectangle (4.5,1.5)(0.5,1.5) rectangle (1.0,2.0)(1.0,1.5) rectangle (1.5,2.0)(1.5,1.5) rectangle (2.0,2.0)(2.0,1.5) rectangle (2.5,2.0)(2.5,1.5) rectangle (3.0,2.0)(3.0,1.5) rectangle (3.5,2.0)(3.5,1.5) rectangle (4.0,2.0)(4.0,1.5) rectangle (4.5,2.0)(0.5,2.0) rectangle (1.0,2.5)(1.0,2.0) rectangle (1.5,2.5)(1.5,2.0) rectangle (2.0,2.5)(2.0,2.0) rectangle (2.5,2.5)(2.5,2.0) rectangle (3.0,2.5)(3.0,2.0) rectangle (3.5,2.5)(3.5,2.0) rectangle (4.0,2.5)(4.0,2.0) rectangle (4.5,2.5)(0.5,2.5) rectangle (1.0,3.0)(1.0,2.5) rectangle (1.5,3.0)(1.5,2.5) rectangle (2.0,3.0)(2.0,2.5) rectangle (2.5,3.0)(2.5,2.5) rectangle (3.0,3.0)(3.0,2.5) rectangle (3.5,3.0)(3.5,2.5) rectangle (4.0,3.0)(4.0,2.5) rectangle (4.5,3.0)(4.5,2.5) rectangle (5.0,3.0)(0.5,3.0) rectangle (1.0,3.5)(1.0,3.0) rectangle (1.5,3.5)(1.5,3.0) rectangle (2.0,3.5)(2.0,3.0) rectangle (2.5,3.5)(2.5,3.0) rectangle (3.0,3.5)(3.0,3.0) rectangle (3.5,3.5)(3.5,3.0) rectangle (4.0,3.5)(4.0,3.0) rectangle (4.5,3.5)(4.5,3.0) rectangle (5.0,3.5)(0.5,3.5) rectangle (1.0,4.0)(1.0,3.5) rectangle (1.5,4.0)(1.5,3.5) rectangle (2.0,4.0)(2.0,3.5) rectangle (2.5,4.0)(2.5,3.5) rectangle (3.0,4.0)(3.0,3.5) rectangle (3.5,4.0)(3.5,3.5) rectangle (4.0,4.0)(4.0,3.5) rectangle (4.5,4.0)(4.5,3.5) rectangle (5.0,4.0)(0.5,4.0) rectangle (1.0,4.5)(1.0,4.0) rectangle (1.5,4.5)(1.5,4.0) rectangle (2.0,4.5)(2.0,4.0) rectangle (2.5,4.5)(2.5,4.0) rectangle (3.0,4.5)(3.0,4.0) rectangle (3.5,4.5)(3.5,4.0) rectangle (4.0,4.5)(4.0,4.0) rectangle (4.5,4.5)(4.5,4.0) rectangle (5.0,4.5)(0.5,4.5) rectangle (1.0,5.0)(1.0,4.5) rectangle (1.5,5.0)(1.5,4.5) rectangle (2.0,5.0)(2.0,4.5) rectangle (2.5,5.0)(2.5,4.5) rectangle (3.0,5.0)(3.0,4.5) rectangle (3.5,5.0)(3.5,4.5) rectangle (4.0,5.0)(4.0,4.5) rectangle (4.5,5.0)(4.5,4.5) rectangle (5.0,5.0)(0.5,5.0) rectangle (1.0,5.5)(1.0,5.0) rectangle (1.5,5.5)(1.5,5.0) rectangle (2.0,5.5)(2.0,5.0) rectangle (2.5,5.5)(2.5,5.0) rectangle (3.0,5.5)(3.0,5.0) rectangle (3.5,5.5)(3.5,5.0) rectangle (4.0,5.5)(4.0,5.0) rectangle (4.5,5.5)(4.5,5.0) rectangle (5.0,5.5)(0.5,5.5) rectangle (1.0,6.0)(1.0,5.5) rectangle (1.5,6.0)(1.5,5.5) rectangle (2.0,6.0)(2.0,5.5) rectangle (2.5,6.0)(2.5,5.5) rectangle (3.0,6.0)(3.0,5.5) rectangle (3.5,6.0)(3.5,5.5) rectangle (4.0,6.0)(4.0,5.5) rectangle (4.5,6.0)(4.5,5.5) rectangle (5.0,6.0);
\node[scale=0.4] at (0.75,5.75) {$1,1,1$};
\node[scale=0.4] at (1.25,5.75) {$1,2,1$};
\node[scale=0.4] at (1.75,5.75) {$1,3,1$};
\node[scale=0.4] at (2.25,5.75) {$1,4,1$};
\node[scale=0.4] at (2.75,5.75) {$1,5,1$};
\node[scale=0.4] at (3.25,5.75) {$1,6,1$};
\node[scale=0.4] at (3.75,5.75) {$1,7,1$};
\node[scale=0.4] at (4.25,5.75) {$1,8,1$};
\node[scale=0.4] at (4.75,5.75) {$1,8,2$};

\node[scale=0.4] at (0.75,5.25) {$2,1,1$};
\node[scale=0.4] at (1.25,5.25) {$2,2,1$};
\node[scale=0.4] at (1.75,5.25) {$2,3,1$};
\node[scale=0.4] at (2.25,5.25) {$2,4,1$};
\node[scale=0.4] at (2.75,5.25) {$2,5,1$};
\node[scale=0.4] at (3.25,5.25) {$2,6,1$};
\node[scale=0.4] at (3.75,5.25) {$2,7,1$};
\node[scale=0.4] at (4.25,5.25) {$2,8,1$};
\node[scale=0.4] at (4.75,5.25) {$2,8,2$};

\node[scale=0.4] at (0.75,4.75) {$3,1,1$};
\node[scale=0.4] at (1.25,4.75) {$3,2,1$};
\node[scale=0.4] at (1.75,4.75) {$3,3,1$};
\node[scale=0.4] at (2.25,4.75) {$3,4,1$};
\node[scale=0.4] at (2.75,4.75) {$3,5,1$};
\node[scale=0.4] at (3.25,4.75) {$3,6,1$};
\node[scale=0.4] at (3.75,4.75) {$3,7,1$};
\node[scale=0.4] at (4.25,4.75) {$3,8,1$};
\node[scale=0.4] at (4.75,4.75) {$3,8,2$};

\node[scale=0.4] at (0.75,4.25) {$4,1,1$};
\node[scale=0.4] at (1.25,4.25) {$4,2,1$};
\node[scale=0.4] at (1.75,4.25) {$4,3,1$};
\node[scale=0.4] at (2.25,4.25) {$4,4,1$};
\node[scale=0.4] at (2.75,4.25) {$4,5,1$};
\node[scale=0.4] at (3.25,4.25) {$4,6,1$};
\node[scale=0.4] at (3.75,4.25) {$4,7,1$};
\node[scale=0.4] at (4.25,4.25) {$4,8,1$};
\node[scale=0.4] at (4.75,4.25) {$4,8,2$};

\node[scale=0.4] at (0.75,3.75) {$5,1,1$};
\node[scale=0.4] at (1.25,3.75) {$5,2,1$};
\node[scale=0.4] at (1.75,3.75) {$5,3,1$};
\node[scale=0.4] at (2.25,3.75) {$5,4,1$};
\node[scale=0.4] at (2.75,3.75) {$5,5,1$};
\node[scale=0.4] at (3.25,3.75) {$5,6,1$};
\node[scale=0.4] at (3.75,3.75) {$5,7,1$};
\node[scale=0.4] at (4.25,3.75) {$5,8,1$};
\node[scale=0.4] at (4.75,3.75) {$5,8,2$};

\node[scale=0.4] at (0.75,3.25) {$6,1,1$};
\node[scale=0.4] at (1.25,3.25) {$6,2,1$};
\node[scale=0.4] at (1.75,3.25) {$6,3,1$};
\node[scale=0.4] at (2.25,3.25) {$6,4,1$};
\node[scale=0.4] at (2.75,3.25) {$6,5,1$};
\node[scale=0.4] at (3.25,3.25) {$6,6,1$};
\node[scale=0.4] at (3.75,3.25) {$6,7,1$};
\node[scale=0.4] at (4.25,3.25) {$6,8,1$};
\node[scale=0.4] at (4.75,3.25) {$6,8,2$};

\node[scale=0.4] at (0.75,2.75) {$7,1,1$};
\node[scale=0.4] at (1.25,2.75) {$7,2,1$};
\node[scale=0.4] at (1.75,2.75) {$7,3,1$};
\node[scale=0.4] at (2.25,2.75) {$7,4,1$};
\node[scale=0.4] at (2.75,2.75) {$7,5,1$};
\node[scale=0.4] at (3.25,2.75) {$7,6,1$};
\node[scale=0.4] at (3.75,2.75) {$7,7,1$};
\node[scale=0.4] at (4.25,2.75) {$7,8,1$};
\node[scale=0.4] at (4.75,2.75) {$7,8,2$};

\node[scale=0.4] at (0.75,2.25) {$8,1,1$};
\node[scale=0.4] at (1.25,2.25) {$8,2,1$};
\node[scale=0.4] at (1.75,2.25) {$8,3,1$};
\node[scale=0.4] at (2.25,2.25) {$8,4,1$};
\node[scale=0.4] at (2.75,2.25) {$8,5,1$};
\node[scale=0.4] at (3.25,2.25) {$8,6,1$};
\node[scale=0.4] at (3.75,2.25) {$8,7,1$};
\node[scale=0.4] at (4.25,2.25) {$8,8,1$};

\node[scale=0.4] at (0.75,1.75) {$9,1,1$};
\node[scale=0.4] at (1.25,1.75) {$9,2,1$};
\node[scale=0.4] at (1.75,1.75) {$9,3,1$};
\node[scale=0.4] at (2.25,1.75) {$9,4,1$};
\node[scale=0.4] at (2.75,1.75) {$9,5,1$};
\node[scale=0.4] at (3.25,1.75) {$9,6,1$};
\node[scale=0.4] at (3.75,1.75) {$9,7,1$};
\node[scale=0.4] at (4.25,1.75) {$9,8,1$};

\node[scale=0.4] at (0.75,1.25) {$9,1,2$};
\node[scale=0.4] at (1.25,1.25) {$9,2,2$};
\node[scale=0.4] at (1.75,1.25) {$9,3,2$};
\node[scale=0.4] at (2.25,1.25) {$9,4,2$};
\node[scale=0.4] at (2.75,1.25) {$9,5,2$};
\node[scale=0.4] at (3.25,1.25) {$9,6,2$};
\node[scale=0.4] at (3.75,1.25) {$9,7,2$};
\node[scale=0.4] at (4.25,1.25) {$9,8,2$};

\node[scale=0.4] at (0.75,0.75) {$10,1,2$};
\node[scale=0.4] at (1.25,0.75) {$10,2,2$};
\node[scale=0.4] at (1.75,0.75) {$10,3,2$};
\node[scale=0.4] at (2.25,0.75) {$10,4,2$};
\node[scale=0.4] at (2.75,0.75) {$10,5,2$};
\node[scale=0.4] at (3.25,0.75) {$10,6,2$};
\node[scale=0.4] at (3.75,0.75) {$10,7,2$};
\node[scale=0.4] at (4.25,0.75) {$10,8,2$};

\node[scale=0.4] at (0.25,0.25) {$11,1,1$};
\node[scale=0.4] at (0.75,0.25) {$11,1,2$};
\node[scale=0.4] at (1.25,0.25) {$11,2,2$};
\node[scale=0.4] at (1.75,0.25) {$11,3,2$};
\node[scale=0.4] at (2.25,0.25) {$11,4,2$};
\node[scale=0.4] at (2.75,0.25) {$11,5,2$};
\node[scale=0.4] at (3.25,0.25) {$11,6,2$};
\node[scale=0.4] at (3.75,0.25) {$11,7,2$};
\node[scale=0.4] at (4.25,0.25) {$11,8,2$};
\node[align=left] at (2.5,-0.5){(b) A local embedding of $H$ into\\$[11]\times[8]\times[2]$.};
\end{scope}
\end{tikzpicture}
\end{center}

\begin{center}
\begin{tikzpicture}
\begin{scope}[xshift=0.0cm,yshift=0.0cm]
\draw[step=0.5cm,color=black,fill=gray] (-0.001,-0.001) grid (4.001,5.501) (0.0,0.0) rectangle (0.5,0.5)(0.0,1.0) rectangle (0.5,1.5)(0.5,1.0) rectangle (1.0,1.5)(1.0,1.0) rectangle (1.5,1.5)(1.5,1.0) rectangle (2.0,1.5)(2.0,1.0) rectangle (2.5,1.5)(2.5,1.0) rectangle (3.0,1.5)(3.0,1.0) rectangle (3.5,1.5)(3.5,1.0) rectangle (4.0,1.5)(0.0,1.5) rectangle (0.5,2.0)(0.5,1.5) rectangle (1.0,2.0)(1.0,1.5) rectangle (1.5,2.0)(1.5,1.5) rectangle (2.0,2.0)(2.0,1.5) rectangle (2.5,2.0)(2.5,1.5) rectangle (3.0,2.0)(3.0,1.5) rectangle (3.5,2.0)(3.5,1.5) rectangle (4.0,2.0)(0.0,2.0) rectangle (0.5,2.5)(0.5,2.0) rectangle (1.0,2.5)(1.0,2.0) rectangle (1.5,2.5)(1.5,2.0) rectangle (2.0,2.5)(2.0,2.0) rectangle (2.5,2.5)(2.5,2.0) rectangle (3.0,2.5)(3.0,2.0) rectangle (3.5,2.5)(3.5,2.0) rectangle (4.0,2.5)(0.0,2.5) rectangle (0.5,3.0)(0.5,2.5) rectangle (1.0,3.0)(1.0,2.5) rectangle (1.5,3.0)(1.5,2.5) rectangle (2.0,3.0)(2.0,2.5) rectangle (2.5,3.0)(2.5,2.5) rectangle (3.0,3.0)(3.0,2.5) rectangle (3.5,3.0)(3.5,2.5) rectangle (4.0,3.0)(0.0,3.0) rectangle (0.5,3.5)(0.5,3.0) rectangle (1.0,3.5)(1.0,3.0) rectangle (1.5,3.5)(1.5,3.0) rectangle (2.0,3.5)(2.0,3.0) rectangle (2.5,3.5)(2.5,3.0) rectangle (3.0,3.5)(3.0,3.0) rectangle (3.5,3.5)(3.5,3.0) rectangle (4.0,3.5)(0.0,3.5) rectangle (0.5,4.0)(0.5,3.5) rectangle (1.0,4.0)(1.0,3.5) rectangle (1.5,4.0)(1.5,3.5) rectangle (2.0,4.0)(2.0,3.5) rectangle (2.5,4.0)(2.5,3.5) rectangle (3.0,4.0)(3.0,3.5) rectangle (3.5,4.0)(3.5,3.5) rectangle (4.0,4.0)(0.0,4.0) rectangle (0.5,4.5)(0.5,4.0) rectangle (1.0,4.5)(1.0,4.0) rectangle (1.5,4.5)(1.5,4.0) rectangle (2.0,4.5)(2.0,4.0) rectangle (2.5,4.5)(2.5,4.0) rectangle (3.0,4.5)(3.0,4.0) rectangle (3.5,4.5)(3.5,4.0) rectangle (4.0,4.5)(0.0,4.5) rectangle (0.5,5.0)(0.5,4.5) rectangle (1.0,5.0)(1.0,4.5) rectangle (1.5,5.0)(1.5,4.5) rectangle (2.0,5.0)(2.0,4.5) rectangle (2.5,5.0)(2.5,4.5) rectangle (3.0,5.0)(3.0,4.5) rectangle (3.5,5.0)(3.5,4.5) rectangle (4.0,5.0)(0.0,5.0) rectangle (0.5,5.5)(0.5,5.0) rectangle (1.0,5.5)(1.0,5.0) rectangle (1.5,5.5)(1.5,5.0) rectangle (2.0,5.5)(2.0,5.0) rectangle (2.5,5.5)(2.5,5.0) rectangle (3.0,5.5)(3.0,5.0) rectangle (3.5,5.5)(3.5,5.0) rectangle (4.0,5.5);
\node at (0.25,0.75) {X};
\end{scope}

\begin{scope}[xshift=4.5cm,yshift=0.0cm]
\draw[step=0.5cm,color=black,fill=gray] (-0.001,-0.001) grid (4.001,5.501) (0.0,0.0) rectangle (0.5,0.5)(0.5,0.0) rectangle (1.0,0.5)(1.0,0.0) rectangle (1.5,0.5)(1.5,0.0) rectangle (2.0,0.5)(2.0,0.0) rectangle (2.5,0.5)(2.5,0.0) rectangle (3.0,0.5)(3.0,0.0) rectangle (3.5,0.5)(3.5,0.0) rectangle (4.0,0.5)(0.0,0.5) rectangle (0.5,1.0)(0.5,0.5) rectangle (1.0,1.0)(1.0,0.5) rectangle (1.5,1.0)(1.5,0.5) rectangle (2.0,1.0)(2.0,0.5) rectangle (2.5,1.0)(2.5,0.5) rectangle (3.0,1.0)(3.0,0.5) rectangle (3.5,1.0)(3.5,0.5) rectangle (4.0,1.0)(0.0,1.0) rectangle (0.5,1.5)(0.5,1.0) rectangle (1.0,1.5)(1.0,1.0) rectangle (1.5,1.5)(1.5,1.0) rectangle (2.0,1.5)(2.0,1.0) rectangle (2.5,1.5)(2.5,1.0) rectangle (3.0,1.5)(3.0,1.0) rectangle (3.5,1.5)(3.5,1.0) rectangle (4.0,1.5)(3.5,2.0) rectangle (4.0,2.5)(3.5,2.5) rectangle (4.0,3.0)(3.5,3.0) rectangle (4.0,3.5)(3.5,3.5) rectangle (4.0,4.0)(3.5,4.0) rectangle (4.0,4.5)(3.5,4.5) rectangle (4.0,5.0)(3.5,5.0) rectangle (4.0,5.5);\node at (3.75,1.75) {X};
\end{scope}
\node[align=left] at (4.25,-0.75){(c) The image of $H$ under the local embedding and two vertices (X)\\that are infected after one time step.};

\end{tikzpicture}
\end{center}
    \caption{Diagrams used in the proof of Proposition~\ref{prop:2,2mod3,2mod3}.}
    \label{fig:22mod32mod3}
\end{figure}
\end{proof}

\begin{figure}[htbp]
\begin{center}
\begin{tikzpicture}
\begin{scope}[xshift=0.0cm,yshift=0.0cm,scale=1.2]
\draw[step=0.5cm,color=black]  (1.0,0.0) rectangle (1.5,0.5)(1.5,0.0) rectangle (2.0,0.5)(2.0,0.0) rectangle (2.5,0.5)(2.5,0.0) rectangle (3.0,0.5)(3.0,0.0) rectangle (3.5,0.5)(3.5,0.0) rectangle (4.0,0.5)(4.0,0.0) rectangle (4.5,0.5)(4.5,0.0) rectangle (5.0,0.5)(5.0,0.0) rectangle (5.5,0.5)(5.5,0.0) rectangle (6.0,0.5)(6.0,0.0) rectangle (6.5,0.5)(1.0,0.5) rectangle (1.5,1.0)(1.5,0.5) rectangle (2.0,1.0)(2.0,0.5) rectangle (2.5,1.0)(2.5,0.5) rectangle (3.0,1.0)(3.0,0.5) rectangle (3.5,1.0)(3.5,0.5) rectangle (4.0,1.0)(4.0,0.5) rectangle (4.5,1.0)(4.5,0.5) rectangle (5.0,1.0)(5.0,0.5) rectangle (5.5,1.0)(5.5,0.5) rectangle (6.0,1.0)(6.0,0.5) rectangle (6.5,1.0)(1.0,1.0) rectangle (1.5,1.5)(1.5,1.0) rectangle (2.0,1.5)(2.0,1.0) rectangle (2.5,1.5)(2.5,1.0) rectangle (3.0,1.5)(3.0,1.0) rectangle (3.5,1.5)(3.5,1.0) rectangle (4.0,1.5)(4.0,1.0) rectangle (4.5,1.5)(4.5,1.0) rectangle (5.0,1.5)(5.0,1.0) rectangle (5.5,1.5)(5.5,1.0) rectangle (6.0,1.5)(6.0,1.0) rectangle (6.5,1.5)(0.5,1.5) rectangle (1.0,2.0)(1.5,1.5) rectangle (2.0,2.0)(2.0,1.5) rectangle (2.5,2.0)(2.5,1.5) rectangle (3.0,2.0)(3.0,1.5) rectangle (3.5,2.0)(3.5,1.5) rectangle (4.0,2.0)(4.0,1.5) rectangle (4.5,2.0)(4.5,1.5) rectangle (5.0,2.0)(5.0,1.5) rectangle (5.5,2.0)(5.5,1.5) rectangle (6.0,2.0)(6.0,2.0) rectangle (6.5,2.5)(0.5,2.5) rectangle (1.0,3.0)(1.0,2.5) rectangle (1.5,3.0)(1.5,2.5) rectangle (2.0,3.0)(2.0,2.5) rectangle (2.5,3.0)(2.5,2.5) rectangle (3.0,3.0)(3.0,2.5) rectangle (3.5,3.0)(3.5,2.5) rectangle (4.0,3.0)(4.0,2.5) rectangle (4.5,3.0)(4.5,2.5) rectangle (5.0,3.0)(5.0,2.5) rectangle (5.5,3.0)(5.5,2.5) rectangle (6.0,3.0)(0.0,3.0) rectangle (0.5,3.5)(0.5,3.0) rectangle (1.0,3.5)(1.0,3.0) rectangle (1.5,3.5)(1.5,3.0) rectangle (2.0,3.5)(2.0,3.0) rectangle (2.5,3.5)(2.5,3.0) rectangle (3.0,3.5)(3.0,3.0) rectangle (3.5,3.5)(3.5,3.0) rectangle (4.0,3.5)(4.0,3.0) rectangle (4.5,3.5)(4.5,3.0) rectangle (5.0,3.5)(5.0,3.0) rectangle (5.5,3.5)(5.5,3.0) rectangle (6.0,3.5)(6.0,3.0) rectangle (6.5,3.5)(6.0,2.5) rectangle (6.5,3)(6.0,1.5) rectangle (6.5,2)
(0.5,2.0) rectangle (1.0,2.5)(1.0,2.0) rectangle (1.5,2.5)(1.5,2.0) rectangle (2.0,2.5)(2.0,2.0) rectangle (2.5,2.5)(2.5,2.0) rectangle (3.0,2.5)(3.0,2.0) rectangle (3.5,2.5)(3.5,2.0) rectangle (4.0,2.5)(4.0,2.0) rectangle (4.5,2.5)(4.5,2.0) rectangle (5.0,2.5)(5.0,2.0) rectangle (5.5,2.5)(5.5,2.0) rectangle (6.0,2.5)(1.0,1.5) rectangle (1.5,2.0);

\node[scale=0.4] at (0.25,3.25) {$1,3,1$};
\node[scale=0.4] at (0.75,3.25) {$1,4,1$};
\node[scale=0.4] at (1.25,3.25) {$1,5,1$};
\node[scale=0.4] at (1.75,3.25) {$1,6,1$};
\node[scale=0.4] at (2.25,3.25) {$1,7,1$};
\node[scale=0.4] at (2.75,3.25) {$1,8,1$};
\node[scale=0.4] at (3.25,3.25) {$1,9,1$};
\node[scale=0.4] at (3.75,3.25) {$1,10,1$};
\node[scale=0.4] at (4.25,3.25) {$1,11,1$};
\node[scale=0.4] at (4.75,3.25) {$1,12,1$};
\node[scale=0.4] at (5.25,3.25) {$1,13,1$};
\node[scale=0.4] at (5.75,3.25) {$1,14,1$};
\node[scale=0.4] at (6.25,3.25) {$1,15,1$};

\node[scale=0.4] at (0.75,2.75) {$2,4,1$};
\node[scale=0.4] at (1.25,2.75) {$2,5,1$};
\node[scale=0.4] at (1.75,2.75) {$2,6,1$};
\node[scale=0.4] at (2.25,2.75) {$2,7,1$};
\node[scale=0.4] at (2.75,2.75) {$2,8,1$};
\node[scale=0.4] at (3.25,2.75) {$2,9,1$};
\node[scale=0.4] at (3.75,2.75) {$2,10,1$};
\node[scale=0.4] at (4.25,2.75) {$2,11,1$};
\node[scale=0.4] at (4.75,2.75) {$2,12,1$};
\node[scale=0.4] at (5.25,2.75) {$2,13,1$};
\node[scale=0.4] at (5.75,2.75) {$2,14,1$};
\node[scale=0.4] at (6.25,2.75) {$2,15,1$};

\node[scale=0.4] at (0.75,2.25) {$3,4,1$};
\node[scale=0.4] at (1.25,2.25) {$3,5,1$};
\node[scale=0.4] at (1.75,2.25) {$3,6,1$};
\node[scale=0.4] at (2.25,2.25) {$3,7,1$};
\node[scale=0.4] at (2.75,2.25) {$3,8,1$};
\node[scale=0.4] at (3.25,2.25) {$3,9,1$};
\node[scale=0.4] at (3.75,2.25) {$3,10,1$};
\node[scale=0.4] at (4.25,2.25) {$3,11,1$};
\node[scale=0.4] at (4.75,2.25) {$3,12,1$};
\node[scale=0.4] at (5.25,2.25) {$3,13,1$};
\node[scale=0.4] at (5.75,2.25) {$3,14,1$};
\node[scale=0.4] at (6.25,2.25) {$3,15,1$};

\node[scale=0.4] at (0.75,1.75) {$4,4,1$};
\node[scale=0.4] at (1.25,1.75) {$4,5,1$};
\node[scale=0.4] at (1.75,1.75) {$4,6,1$};
\node[scale=0.4] at (2.25,1.75) {$4,7,1$};
\node[scale=0.4] at (2.75,1.75) {$4,8,1$};
\node[scale=0.4] at (3.25,1.75) {$4,9,1$};
\node[scale=0.4] at (3.75,1.75) {$4,10,1$};
\node[scale=0.4] at (4.25,1.75) {$4,11,1$};
\node[scale=0.4] at (4.75,1.75) {$4,12,1$};
\node[scale=0.4] at (5.25,1.75) {$4,13,1$};
\node[scale=0.4] at (5.75,1.75) {$4,14,1$};
\node[scale=0.4] at (6.25,1.75) {$4,15,1$};

\node[scale=0.4] at (1.25,1.25) {$5,5,1$};
\node[scale=0.4] at (1.75,1.25) {$5,6,1$};
\node[scale=0.4] at (2.25,1.25) {$5,7,1$};
\node[scale=0.4] at (2.75,1.25) {$5,8,1$};
\node[scale=0.4] at (3.25,1.25) {$5,9,1$};
\node[scale=0.4] at (3.75,1.25) {$5,10,1$};
\node[scale=0.4] at (4.25,1.25) {$5,11,1$};
\node[scale=0.4] at (4.75,1.25) {$5,12,1$};
\node[scale=0.4] at (5.25,1.25) {$5,13,1$};
\node[scale=0.4] at (5.75,1.25) {$5,14,1$};
\node[scale=0.4] at (6.25,1.25) {$5,15,1$};

\node[scale=0.4] at (1.25,0.75) {$5,5,2$};
\node[scale=0.4] at (1.75,0.75) {$5,6,2$};
\node[scale=0.4] at (2.25,0.75) {$5,7,2$};
\node[scale=0.4] at (2.75,0.75) {$5,8,2$};
\node[scale=0.4] at (3.25,0.75) {$5,9,2$};
\node[scale=0.4] at (3.75,0.75) {$5,10,2$};
\node[scale=0.4] at (4.25,0.75) {$5,11,2$};
\node[scale=0.4] at (4.75,0.75) {$5,12,2$};
\node[scale=0.4] at (5.25,0.75) {$5,13,2$};
\node[scale=0.4] at (5.75,0.75) {$5,14,2$};
\node[scale=0.4] at (6.25,0.75) {$5,15,2$};

\node[scale=0.4] at (1.25,0.25) {$6,5,2$};
\node[scale=0.4] at (1.75,0.25) {$6,6,2$};
\node[scale=0.4] at (2.25,0.25) {$6,7,2$};
\node[scale=0.4] at (2.75,0.25) {$6,8,2$};
\node[scale=0.4] at (3.25,0.25) {$6,9,2$};
\node[scale=0.4] at (3.75,0.25) {$6,10,2$};
\node[scale=0.4] at (4.25,0.25) {$6,11,2$};
\node[scale=0.4] at (4.75,0.25) {$6,12,2$};
\node[scale=0.4] at (5.25,0.25) {$6,13,2$};
\node[scale=0.4] at (5.75,0.25) {$6,14,2$};
\node[scale=0.4] at (6.25,0.25) {$6,15,2$};
\node[align=left] at (3.25,-0.5){(a) A local embedding of $H$ into $[6]\times[15]\times[2]$.};
\end{scope}

\begin{scope}[xshift=10.0cm,yshift=0.0cm]
\draw[step=0.5cm,color=black,fill=gray] (-0.001,-0.001) grid (2.501,3.001) (0.5,0.0) rectangle (1.0,0.5)(0.0,0.5) rectangle (0.5,1.0)(1.0,0.5) rectangle (1.5,1.0)(2.0,0.5) rectangle (2.5,1.0)(1.5,1.0) rectangle (2.0,1.5)(2.0,1.0) rectangle (2.5,1.5)(0.0,1.5) rectangle (0.5,2.0)(1.5,1.5) rectangle (2.0,2.0)(2.0,1.5) rectangle (2.5,2.0)(1.5,2.0) rectangle (2.0,2.5)(2.0,2.0) rectangle (2.5,2.5)(0.0,2.5) rectangle (0.5,3.0)(1.0,2.5) rectangle (1.5,3.0)(1.5,2.5) rectangle (2.0,3.0)(2.0,2.5) rectangle (2.5,3.0);
\node at (0.25,0.25) {1};\node at (1.25,0.25) {1};\node at (1.75,0.25) {6};\node at (2.25,0.25) {7};\node at (0.75,0.75) {1};\node at (1.75,0.75) {1};\node at (0.25,1.25) {1};\node at (0.75,1.25) {4};\node at (1.25,1.25) {5};\node at (0.75,1.75) {5};\node at (1.25,1.75) {6};\node at (0.25,2.25) {9};\node at (0.75,2.25) {8};\node at (1.25,2.25) {7};\node at (0.75,2.75) {1};
\end{scope}

\begin{scope}[xshift=13.0cm,yshift=0.0cm]
\draw[step=0.5cm,color=black,fill=gray] (-0.001,-0.001) grid (2.501,3.001) (0.0,0.0) rectangle (0.5,0.5)(1.0,0.0) rectangle (1.5,0.5)(2.0,0.0) rectangle (2.5,0.5)(2.0,0.5) rectangle (2.5,1.0)(0.0,1.0) rectangle (0.5,1.5)(0.5,1.5) rectangle (1.0,2.0)(0.5,2.5) rectangle (1.0,3.0);
\node at (0.75,0.25) {1};\node at (1.75,0.25) {5};\node at (0.25,0.75) {1};\node at (0.75,0.75) {2};\node at (1.25,0.75) {3};\node at (1.75,0.75) {4};\node at (0.75,1.25) {3};\node at (1.25,1.25) {6};\node at (1.75,1.25) {7};\node at (2.25,1.25) {8};\node at (0.25,1.75) {1};\node at (1.25,1.75) {7};\node at (1.75,1.75) {8};\node at (2.25,1.75) {9};\node at (0.25,2.25) {10};\node at (0.75,2.25) {9};\node at (1.25,2.25) {10};\node at (1.75,2.25) {11};\node at (2.25,2.25) {12};\node at (0.25,2.75) {11};\node at (1.25,2.75) {11};\node at (1.75,2.75) {12};\node at (2.25,2.75) {13};
\node[align=left] at (-0.25,-0.75){(b) An infection in the intersection of the\\last six rows and first five columns.};
\end{scope}

\end{tikzpicture}
\end{center}

\begin{center}
\begin{tikzpicture}
\begin{scope}[xshift=0.0cm,yshift=0.0cm]
\draw[step=0.5cm,color=black,fill=gray] (-0.001,-0.001) grid (4.501,6.001) (0.0,0.0) rectangle (0.5,0.5)(0.5,0.0) rectangle (1.0,0.5)(1.0,0.0) rectangle (1.5,0.5)(1.5,0.0) rectangle (2.0,0.5)(2.0,0.0) rectangle (2.5,0.5)(0.0,0.5) rectangle (0.5,1.0)(0.5,0.5) rectangle (1.0,1.0)(1.0,0.5) rectangle (1.5,1.0)(1.5,0.5) rectangle (2.0,1.0)(2.0,0.5) rectangle (2.5,1.0)(2.5,0.5) rectangle (3.0,1.0)(3.0,0.5) rectangle (3.5,1.0)(3.5,0.5) rectangle (4.0,1.0)(4.0,0.5) rectangle (4.5,1.0)(0.0,1.0) rectangle (0.5,1.5)(0.5,1.0) rectangle (1.0,1.5)(1.0,1.0) rectangle (1.5,1.5)(1.5,1.0) rectangle (2.0,1.5)(2.0,1.0) rectangle (2.5,1.5)(2.5,1.0) rectangle (3.0,1.5)(3.0,1.0) rectangle (3.5,1.5)(3.5,1.0) rectangle (4.0,1.5)(4.0,1.0) rectangle (4.5,1.5)(0.0,1.5) rectangle (0.5,2.0)(0.5,1.5) rectangle (1.0,2.0)(1.0,1.5) rectangle (1.5,2.0)(1.5,1.5) rectangle (2.0,2.0)(2.0,1.5) rectangle (2.5,2.0)(2.5,1.5) rectangle (3.0,2.0)(3.0,1.5) rectangle (3.5,2.0)(3.5,1.5) rectangle (4.0,2.0)(4.0,1.5) rectangle (4.5,2.0)(0.0,2.0) rectangle (0.5,2.5)(0.5,2.0) rectangle (1.0,2.5)(1.0,2.0) rectangle (1.5,2.5)(1.5,2.0) rectangle (2.0,2.5)(2.0,2.0) rectangle (2.5,2.5)(2.5,2.0) rectangle (3.0,2.5)(3.0,2.0) rectangle (3.5,2.5)(3.5,2.0) rectangle (4.0,2.5)(4.0,2.0) rectangle (4.5,2.5)(0.0,2.5) rectangle (0.5,3.0)(0.5,2.5) rectangle (1.0,3.0)(1.0,2.5) rectangle (1.5,3.0)(1.5,2.5) rectangle (2.0,3.0)(2.0,2.5) rectangle (2.5,3.0)(2.5,2.5) rectangle (3.0,3.0)(3.0,2.5) rectangle (3.5,3.0)(3.5,2.5) rectangle (4.0,3.0)(4.0,2.5) rectangle (4.5,3.0)(1.0,3.0) rectangle (1.5,3.5)(1.5,3.0) rectangle (2.0,3.5)(2.0,3.0) rectangle (2.5,3.5)(2.5,3.0) rectangle (3.0,3.5)(3.0,3.0) rectangle (3.5,3.5)(3.5,3.0) rectangle (4.0,3.5)(4.0,3.0) rectangle (4.5,3.5)(0.0,3.5) rectangle (0.5,4.0)(1.0,3.5) rectangle (1.5,4.0)(1.5,3.5) rectangle (2.0,4.0)(2.0,3.5) rectangle (2.5,4.0)(2.5,3.5) rectangle (3.0,4.0)(3.0,3.5) rectangle (3.5,4.0)(3.5,3.5) rectangle (4.0,4.0)(4.0,3.5) rectangle (4.5,4.0)(1.0,4.0) rectangle (1.5,4.5)(1.5,4.0) rectangle (2.0,4.5)(2.0,4.0) rectangle (2.5,4.5)(2.5,4.0) rectangle (3.0,4.5)(3.0,4.0) rectangle (3.5,4.5)(3.5,4.0) rectangle (4.0,4.5)(4.0,4.0) rectangle (4.5,4.5)(0.0,4.5) rectangle (0.5,5.0)(1.0,4.5) rectangle (1.5,5.0)(1.5,4.5) rectangle (2.0,5.0)(2.0,4.5) rectangle (2.5,5.0)(2.5,4.5) rectangle (3.0,5.0)(3.0,4.5) rectangle (3.5,5.0)(3.5,4.5) rectangle (4.0,5.0)(4.0,4.5) rectangle (4.5,5.0)(1.0,5.0) rectangle (1.5,5.5)(1.5,5.0) rectangle (2.0,5.5)(2.0,5.0) rectangle (2.5,5.5)(2.5,5.0) rectangle (3.0,5.5)(3.0,5.0) rectangle (3.5,5.5)(3.5,5.0) rectangle (4.0,5.5)(4.0,5.0) rectangle (4.5,5.5)(0.0,5.5) rectangle (0.5,6.0)(1.0,5.5) rectangle (1.5,6.0)(1.5,5.5) rectangle (2.0,6.0)(2.0,5.5) rectangle (2.5,6.0)(2.5,5.5) rectangle (3.0,6.0)(3.0,5.5) rectangle (3.5,6.0)(3.5,5.5) rectangle (4.0,6.0)(4.0,5.5) rectangle (4.5,6.0);
\node at (2.75,0.25) {1};\node at (3.25,0.25) {2};\node at (3.75,0.25) {3};\node at (4.25,0.25) {4};\node at (0.25,3.25) {3};\node at (0.75,3.25) {2};\node at (0.75,3.75) {1};\node at (0.25,4.25) {3};\node at (0.75,4.25) {2};\node at (0.75,4.75) {1};\node at (0.25,5.25) {3};\node at (0.75,5.25) {2};\node at (0.75,5.75) {1};
\end{scope}

\begin{scope}[xshift=5.0cm,yshift=0.0cm]
\draw[step=0.5cm,color=black,fill=gray] (-0.001,-0.001) grid (4.501,6.001) (0.0,0.0) rectangle (0.5,0.5)(0.5,0.0) rectangle (1.0,0.5)(1.0,0.0) rectangle (1.5,0.5)(1.5,0.0) rectangle (2.0,0.5)(2.0,0.0) rectangle (2.5,0.5)(2.5,0.0) rectangle (3.0,0.5)(3.0,0.0) rectangle (3.5,0.5)(3.5,0.0) rectangle (4.0,0.5)(4.0,0.0) rectangle (4.5,0.5)(0.0,0.5) rectangle (0.5,1.0)(0.5,0.5) rectangle (1.0,1.0)(1.0,0.5) rectangle (1.5,1.0)(1.5,0.5) rectangle (2.0,1.0)(2.0,0.5) rectangle (2.5,1.0)(2.5,0.5) rectangle (3.0,1.0)(3.0,0.5) rectangle (3.5,1.0)(3.5,0.5) rectangle (4.0,1.0)(4.0,0.5) rectangle (4.5,1.0)(0.0,1.0) rectangle (0.5,1.5)(0.5,1.0) rectangle (1.0,1.5)(1.0,1.0) rectangle (1.5,1.5)(1.5,1.0) rectangle (2.0,1.5)(2.0,1.0) rectangle (2.5,1.5)(0.0,1.5) rectangle (0.5,2.0)(0.5,1.5) rectangle (1.0,2.0)(1.0,1.5) rectangle (1.5,2.0)(1.5,1.5) rectangle (2.0,2.0)(2.0,1.5) rectangle (2.5,2.0)(0.0,2.0) rectangle (0.5,2.5)(0.5,2.0) rectangle (1.0,2.5)(1.0,2.0) rectangle (1.5,2.5)(1.5,2.0) rectangle (2.0,2.5)(2.0,2.0) rectangle (2.5,2.5)(0.0,2.5) rectangle (0.5,3.0)(0.5,2.5) rectangle (1.0,3.0)(1.0,2.5) rectangle (1.5,3.0)(1.5,2.5) rectangle (2.0,3.0)(2.0,2.5) rectangle (2.5,3.0)(0.5,3.5) rectangle (1.0,4.0)(0.5,4.5) rectangle (1.0,5.0)(0.5,5.5) rectangle (1.0,6.0);
\node at (2.75,1.25) {1};\node at (3.25,1.25) {2};\node at (3.75,1.25) {3};\node at (4.25,1.25) {4};\node at (2.75,1.75) {2};\node at (3.25,1.75) {3};\node at (3.75,1.75) {4};\node at (4.25,1.75) {5};\node at (2.75,2.25) {3};\node at (3.25,2.25) {4};\node at (3.75,2.25) {5};\node at (4.25,2.25) {6};\node at (2.75,2.75) {4};\node at (3.25,2.75) {5};\node at (3.75,2.75) {6};\node at (4.25,2.75) {7};\node at (0.25,3.25) {4};\node at (0.75,3.25) {3};\node at (1.25,3.25) {4};\node at (1.75,3.25) {5};\node at (2.25,3.25) {6};\node at (2.75,3.25) {7};\node at (3.25,3.25) {8};\node at (3.75,3.25) {9};\node at (4.25,3.25) {10};\node at (0.25,3.75) {5};\node at (1.25,3.75) {5};\node at (1.75,3.75) {6};\node at (2.25,3.75) {7};\node at (2.75,3.75) {8};\node at (3.25,3.75) {9};\node at (3.75,3.75) {10};\node at (4.25,3.75) {11};\node at (0.25,4.25) {6};\node at (0.75,4.25) {3};\node at (1.25,4.25) {6};\node at (1.75,4.25) {7};\node at (2.25,4.25) {8};\node at (2.75,4.25) {9};\node at (3.25,4.25) {10};\node at (3.75,4.25) {11};\node at (4.25,4.25) {12};\node at (0.25,4.75) {7};\node at (1.25,4.75) {7};\node at (1.75,4.75) {8};\node at (2.25,4.75) {9};\node at (2.75,4.75) {10};\node at (3.25,4.75) {11};\node at (3.75,4.75) {12};\node at (4.25,4.75) {13};\node at (0.25,5.25) {8};\node at (0.75,5.25) {3};\node at (1.25,5.25) {8};\node at (1.75,5.25) {9};\node at (2.25,5.25) {10};\node at (2.75,5.25) {11};\node at (3.25,5.25) {12};\node at (3.75,5.25) {13};\node at (4.25,5.25) {14};\node at (0.25,5.75) {9};\node at (1.25,5.75) {9};\node at (1.75,5.75) {10};\node at (2.25,5.75) {11};\node at (2.75,5.75) {12};\node at (3.25,5.75) {13};\node at (3.75,5.75) {14};\node at (4.25,5.75) {15};
\node[align=left] at (-0.25,-0.75){(c) Result of the $3$-neighbour process after running it within the image of $H$\\and within the intersection of the last six rows and first five columns.};
\end{scope}

\end{tikzpicture}
\end{center}
    \caption{Diagrams used in the proof of Proposition~\ref{prop:2,0mod3,0mod3}}
    \label{fig:23mod60mod6}
\end{figure}

\begin{prop}
\label{prop:2,0mod3,0mod3}
Let $a_2,a_3\geq 6$. If $a_2,a_3\equiv 0\bmod 3$ and $a_2\not\equiv a_3\bmod 6$, then $(2,a_2,a_3)$ is perfect. 
\end{prop}

\begin{proof}
Without loss of generality, $a_2\equiv 0\bmod 6$ and $a_3\equiv 3\bmod 6$. In our diagrams, we view the grid as $[a_2]\times[a_3]\times[2]$. Let $H$ be the graph obtained from the $(a_2+1)\times(a_3-2)$ grid by deleting the points
\[(a_2-4,1),(a_2-3,1),(a_2-2,1),(a_2-1,1),(a_2,1),(a_2+1,1),(a_2-1,2),(a_2,2),(a_2+1,2).\]
By Lemma~\ref{lem:jagged}, $H$ has a percolating set of cardinality $\frac{(a_2+1)(a_3-2)+(a_2+1)+(a_3-2)-9}{3}$. A local embedding of $H$ into $[a_2]\times[a_3]\times[2]$ in the case that $a_2=6$ and $a_3=15$ is given in Figure~\ref{fig:23mod60mod6}~(a); the local embedding in the general case is analogous. The percolating set consists of the image of the aforementioned infection in $H$ under this local embedding together with the points in the set
\[S:=\{(2i-1,1,1): 1\leq i\leq a_2/2\}\cup\{(a_2,2,1),(a_2-1,3,1)\}\]
\[\cup\{(2i-1,2,2): 1\leq i\leq (a_2-2)/2\}\cup\{(a_2-2,1,2),(a_2,1,2),(a_2,3,2)\}.\]
Consider the set $B_0$ consisting of the elements of the union of $S$ and the image of $H$ under the local embedding contained within $\{a_2-5,\dots,a_2\}\times[5]\times[2]$. Then the elements of $B_0$ contained in $\{a_2-5,\dots,a_2\}\times[5]\times[2]$ percolate within $\{a_2-5,\dots,a_2\}\times[5]\times[2]$, as shown in Figure~\ref{fig:23mod60mod6}~(b). After running the $3$-neighbour process within the image of $H$ and within $\{a_2-5,\dots,a_2\}\times[5]\times[2]$, it is easily observed that the resulting infection percolates. In particular, after running the process for three additional time steps, we reach an infection in which there are no paths between the two faces in any of the three directions. See Figure~\ref{fig:23mod60mod6}~(c) for an illustration of these final stages of the process when $a_2=12$ and $a_3=9$. 
\end{proof}

\begin{prop}
\label{prop:3,3,0mod2}
If $a_3\geq 4$ and $a_3\equiv 0\bmod 2$, then $(3,3,a_3)$ is perfect. 
\end{prop}

\begin{proof}
We view the grid as $[3]\times[a_3]\times [3]$. Define
\[A_0^1:=\{(1,1,1),(3,1,1),(2,3,1),(2,1,2),(1,2,2),(3,1,3),(2,2,3),(3,3,3)\}\text{ and}\]
\[A_0^3:=\{(1,a_3,1),(3,a_3,2),(1,a_3,3)\}.\]
Also, for $1\leq i\leq \frac{a_3-4}{2}$, define 
\[A_0^{2,i}:=\{(1,2+2i,1),(2,2+2i,3),(2,3+2i,1),(3,3+2i,3)\}.\]
Let $A_0:=A_0^1\cup \left(\bigcup_{i=1}^{(a_3-4)/2} A_0^{2,i}\right)\cup A_0^3$. Clearly,
\[|A_0|= 8+4 \cdot \frac{a_3-4}{2}+3=2a_3+3=\frac{3\cdot 3 + 3\cdot a_3+3\cdot a_3}{3}.\]
So, it suffices to show that $A_0$ percolates. The case $a_3=4$ is shown in Figure~\ref{fig:33k}~(a) and so we focus on the case $a_3\geq6$. The case $a_3=8$ is displayed in Figure~\ref{fig:33k}~(b) for reference. 

First, observe that all elements of the grid $[2]\times\{3,\dots,a_3\}\times \{1\}$, except for $(1,3,1)$ and $(2,a_3,1)$, become infected in one step. Likewise, all elements of $\{2,3\}\times[a_3-1]\times\{3\}$ get infected in one step, except for $(2,a_3-1,3)$; note that the infection of $(2,1,3)$ uses the fact that $(2,1,2)$ is in $A_0$. The vertices $(2,1,1),(1,1,2),(3,1,2)$ and $(2,1,1)$ also get infected in just one step. After infecting these vertices, we see that each vertex of the form $(2,i,2)$ for $3\leq i\leq a_3-2$ is adjacent to $(2,i,1)$ and $(2,i,3)$, each of which is currently infected. Additionally, $(2,3,2)$ is adjacent to $(2,2,2)$, which is also infected. Thus, we see that all of $(2,i,2)$ for $3\leq i\leq a_3-2$ become infected, one after the other. 

Each vertex $(3,i,2)$ for $2\leq i\leq a_3-2$ is now adjacent to infected vertices $(2,i,2)$ and $(3,i,3)$. Additionally, $(3,2,2)$ is adjacent to $(3,1,2)$, which is currently infected. So, all of $(3,i,2)$ for $2\leq i\leq a_3-2$ get infected. After this, $(3,a_3-1,2)$ becomes infected due to having infected neighbours $(3,a_3-2,2), (3,a_3,2)$ and $(3,a_3-1,3)$. 

Now, we observe that $(2,2,1)$ becomes infected due to being adjacent to $(2,1,1)$,$(2,3,1)$ and $(2,2,2)$, each of which is either in $A_0$ or was infected after one step. Using this, we get that $(1,2,1)$ and $(1,3,1)$ also get infected.  Also, each of $(3,i,1)$ for $2\leq i\leq a_3-1$ become infected, one by one. 

Currently, all of $[3]\times[a_3]\times\{1\}$ is infected, except for $(2,a_3,1)$ and $(3,a_3,1)$. It is now not hard to see that all of the remaining uninfected vertices of $[3]\times[a_3]\times\{2\}$ (i.e. the top row and the rightmost two vertices on the second row) become infected, one at a time. This triggers the last two vertices in $[3]\times[a_3]\times\{1\}$ to become infected. Finally, it becomes easy to see that all of $[3]\times[a_3]\times\{3\}$ becomes infected. This completes the proof. 
\end{proof}

\begin{figure}[htbp]

\begin{center}
\begin{tikzpicture}
\begin{scope}[xshift=0.0cm,yshift=0.0cm]
\draw[step=0.5cm,color=black,fill=gray] (-0.001,-0.001) grid (2.001,1.501) (0.0,0.0) rectangle (0.5,0.5)(1.0,0.5) rectangle (1.5,1.0)(0.0,1.0) rectangle (0.5,1.5)(1.5,1.0) rectangle (2.0,1.5);
\node at (0.75,0.25) {3};\node at (1.25,0.25) {4};\node at (1.75,0.25) {9};\node at (0.25,0.75) {1};\node at (0.75,0.75) {2};\node at (1.75,0.75) {8};\node at (0.75,1.25) {3};\node at (1.25,1.25) {4};
\end{scope}

\begin{scope}[xshift=2.5cm,yshift=0.0cm]
\draw[step=0.5cm,color=black,fill=gray] (-0.001,-0.001) grid (2.001,1.501) (1.5,0.0) rectangle (2.0,0.5)(0.0,0.5) rectangle (0.5,1.0)(0.5,1.0) rectangle (1.0,1.5);
\node at (0.25,0.25) {1};\node at (0.75,0.25) {2};\node at (1.25,0.25) {3};\node at (0.75,0.75) {1};\node at (1.25,0.75) {4};\node at (1.75,0.75) {7};\node at (0.25,1.25) {1};\node at (1.25,1.25) {5};\node at (1.75,1.25) {6};

\node[align=left] at (1,-0.5){(a) The case $a_3=4$.};
\end{scope}

\begin{scope}[xshift=5.0cm,yshift=0.0cm]
\draw[step=0.5cm,color=black,fill=gray] (-0.001,-0.001) grid (2.001,1.501) (0.0,0.0) rectangle (0.5,0.5)(1.0,0.0) rectangle (1.5,0.5)(0.5,0.5) rectangle (1.0,1.0)(1.5,1.0) rectangle (2.0,1.5);
\node at (0.75,0.25) {1};\node at (1.75,0.25) {9};\node at (0.25,0.75) {1};\node at (1.25,0.75) {5};\node at (1.75,0.75) {8};\node at (0.25,1.25) {8};\node at (0.75,1.25) {7};\node at (1.25,1.25) {6};
\end{scope}

\end{tikzpicture}
\end{center}

\begin{center}
\begin{tikzpicture}
\begin{scope}[xshift=0.0cm,yshift=0.0cm]
\draw[step=0.5cm,color=black,fill=gray] (-0.001,-0.001) grid (4.001,1.501) (0.0,0.0) rectangle (0.5,0.5)(1.0,0.5) rectangle (1.5,1.0)(2.0,0.5) rectangle (2.5,1.0)(3.0,0.5) rectangle (3.5,1.0)(0.0,1.0) rectangle (0.5,1.5)(1.5,1.0) rectangle (2.0,1.5)(2.5,1.0) rectangle (3.0,1.5)(3.5,1.0) rectangle (4.0,1.5);
\node at (0.75,0.25) {3};\node at (1.25,0.25) {4};\node at (1.75,0.25) {5};\node at (2.25,0.25) {6};\node at (2.75,0.25) {7};\node at (3.25,0.25) {8};\node at (3.75,0.25) {13};\node at (0.25,0.75) {1};\node at (0.75,0.75) {2};\node at (1.75,0.75) {1};\node at (2.75,0.75) {1};\node at (3.75,0.75) {12};\node at (0.75,1.25) {3};\node at (1.25,1.25) {4};\node at (2.25,1.25) {1};\node at (3.25,1.25) {1};
\end{scope}

\begin{scope}[xshift=4.5cm,yshift=0.0cm]
\draw[step=0.5cm,color=black,fill=gray] (-0.001,-0.001) grid (4.001,1.501) (3.5,0.0) rectangle (4.0,0.5)(0.0,0.5) rectangle (0.5,1.0)(0.5,1.0) rectangle (1.0,1.5);
\node at (0.25,0.25) {1};\node at (0.75,0.25) {2};\node at (1.25,0.25) {3};\node at (1.75,0.25) {4};\node at (2.25,0.25) {5};\node at (2.75,0.25) {6};\node at (3.25,0.25) {7};\node at (0.75,0.75) {1};\node at (1.25,0.75) {2};\node at (1.75,0.75) {3};\node at (2.25,0.75) {4};\node at (2.75,0.75) {5};\node at (3.25,0.75) {8};\node at (3.75,0.75) {11};\node at (0.25,1.25) {1};\node at (1.25,1.25) {5};\node at (1.75,1.25) {6};\node at (2.25,1.25) {7};\node at (2.75,1.25) {8};\node at (3.25,1.25) {9};\node at (3.75,1.25) {10};
\node[align=left] at (2,-0.5){(b) The case $a_3=8$.};
\end{scope}

\begin{scope}[xshift=9.0cm,yshift=0.0cm]
\draw[step=0.5cm,color=black,fill=gray] (-0.001,-0.001) grid (4.001,1.501) (0.0,0.0) rectangle (0.5,0.5)(1.0,0.0) rectangle (1.5,0.5)(2.0,0.0) rectangle (2.5,0.5)(3.0,0.0) rectangle (3.5,0.5)(0.5,0.5) rectangle (1.0,1.0)(1.5,0.5) rectangle (2.0,1.0)(2.5,0.5) rectangle (3.0,1.0)(3.5,1.0) rectangle (4.0,1.5);
\node at (0.75,0.25) {1};\node at (1.75,0.25) {1};\node at (2.75,0.25) {1};\node at (3.75,0.25) {13};\node at (0.25,0.75) {1};\node at (1.25,0.75) {1};\node at (2.25,0.75) {1};\node at (3.25,0.75) {9};\node at (3.75,0.75) {12};\node at (0.25,1.25) {16};\node at (0.75,1.25) {15};\node at (1.25,1.25) {14};\node at (1.75,1.25) {13};\node at (2.25,1.25) {12};\node at (2.75,1.25) {11};\node at (3.25,1.25) {10};
\end{scope}
\end{tikzpicture}
\end{center}

    \caption{Diagrams used in the proof of Proposition~\ref{prop:3,3,0mod2}}
    \label{fig:33k}
\end{figure}

\begin{prop}
\label{prop:3,4,3mod6}
If $a_3\geq 3$ and $a_3\equiv 3\bmod 6$, then $(3,4,a_3)$ is perfect. 
\end{prop}

\begin{proof}
In the case that $a_3=3$, the result follows from Proposition~\ref{prop:3,3,0mod2}. So, from here forward, we assume that $a_3\geq9$. In our diagrams, we view the grid as $[4]\times[a_3]\times[3]$. Let $H$ be a $[6]\times[a_3-1]$ grid with $(1,1),(3,1)$ and $(6,a_3-1)$ removed. By Lemma~\ref{lem:fatter}, there is an infection of cardinality $\frac{6(a_3-1)+6+(a_3-1)-2}{3} = \frac{7a_3-3}{3}$ in $[6]\times[a_3-1]$ whose closure contains all but the corners $(1,1)$ and $(6,a_3-1)$. By inspecting the proof of Lemma~\ref{lem:fatter}, we observe that this set contains $(2,1)$ and $(4,1)$ and does not contain $(3,1)$.\footnote{In fact, rather than inspecting the proof, one could alternatively apply Lemma~\ref{lem:immuneRegions} to show that any such infection must alternate around the boundary of $[6]\times[a_3-1]$, analogous to the proof \eqref{eq:boundaryAlternates}.} Also, since $(3,1)$ has only three neighbours, it does not become infected until after $(4,2)$ becomes infected. Thus, this construction gives us a percolating set in $H$ of cardinality $\frac{7a_3-3}{3}$.

An illustration of a local embedding of $H$ into $[4]\times[a_3]\times[3]$ in the case $a_3=15$ is given in Figure~\ref{fig:34c}~(a); the local embedding in the general case is analogous. Note that, in contrast to the previous proofs in this section, this local embedding is not injective. Specifically, it maps both of $(2,1)$ and $(4,1)$ to $(2,1,2)$. Recall that both of $(2,1)$ and $(4,1)$ were contained in the initial infection of $H$. Thus, the image of the infection in $H$ under the local embedding has cardinality $\frac{7a_3-3}{3}-1=\frac{7a_3-6}{3}$. We add six additional elements to the infection: $(1,a_3,1),(2,a_3-1,1),(3,a_3,1),(4,a_3-1,1),(4,a_3,2)$ and $(3,a_3-1,3)$. It is tedious, but not difficult, to show that this set percolates. See Figure~\ref{fig:34c}~(b) for an illustration in the case that $a_3=15$; in this figure, the six additional infections are marked with an X and the other grey cells are the elements of the image of $H$ under the local embedding. Thus,
\[m(3,4,a_3;3)\leq \frac{7a_3-6}{3} + 6 = \frac{3a_3+4a_3+3\cdot 4}{3}\]
as desired.
\end{proof}

\begin{figure}[htbp]
\begin{center}
\begin{tikzpicture}
\begin{scope}[xshift=0.0cm,yshift=0.0cm,scale=1.2]
\draw[step=0.5cm,color=black,fill=gray]  (0.4999,-0.001) grid (6.501,3.001)(6.4999,0.4999) grid (7.001,3.001)(-0.001,-0.001) grid (0.501,1.501)(-0.001,1.999) grid (0.501,2.501);
\node[scale=0.4] at (0.75,2.75) {$1,1,1$};
\node[scale=0.4] at (1.25,2.75) {$1,2,1$};
\node[scale=0.4] at (1.75,2.75) {$1,3,1$};
\node[scale=0.4] at (2.25,2.75) {$1,4,1$};
\node[scale=0.4] at (2.75,2.75) {$1,5,1$};
\node[scale=0.4] at (3.25,2.75) {$1,6,1$};
\node[scale=0.4] at (3.75,2.75) {$1,7,1$};
\node[scale=0.4] at (4.25,2.75) {$1,8,1$};
\node[scale=0.4] at (4.75,2.75) {$1,9,1$};
\node[scale=0.4] at (5.25,2.75) {$1,10,1$};
\node[scale=0.4] at (5.75,2.75) {$1,11,1$};
\node[scale=0.4] at (6.25,2.75) {$1,12,1$};
\node[scale=0.4] at (6.75,2.75) {$1,13,1$};

\node[scale=0.4] at (0.75,2.25) {$1,1,2$};
\node[scale=0.4] at (1.25,2.25) {$1,2,2$};
\node[scale=0.4] at (1.75,2.25) {$1,3,2$};
\node[scale=0.4] at (2.25,2.25) {$1,4,2$};
\node[scale=0.4] at (2.75,2.25) {$1,5,2$};
\node[scale=0.4] at (3.25,2.25) {$1,6,2$};
\node[scale=0.4] at (3.75,2.25) {$1,7,2$};
\node[scale=0.4] at (4.25,2.25) {$1,8,2$};
\node[scale=0.4] at (4.75,2.25) {$1,9,2$};
\node[scale=0.4] at (5.25,2.25) {$1,10,2$};
\node[scale=0.4] at (5.75,2.25) {$1,11,2$};
\node[scale=0.4] at (6.25,2.25) {$1,12,2$};
\node[scale=0.4] at (6.75,2.25) {$1,13,2$};

\node[scale=0.4] at (0.75,1.75) {$1,1,3$};
\node[scale=0.4] at (1.25,1.75) {$1,2,3$};
\node[scale=0.4] at (1.75,1.75) {$1,3,3$};
\node[scale=0.4] at (2.25,1.75) {$1,4,3$};
\node[scale=0.4] at (2.75,1.75) {$1,5,3$};
\node[scale=0.4] at (3.25,1.75) {$1,6,3$};
\node[scale=0.4] at (3.75,1.75) {$1,7,3$};
\node[scale=0.4] at (4.25,1.75) {$1,8,3$};
\node[scale=0.4] at (4.75,1.75) {$1,9,3$};
\node[scale=0.4] at (5.25,1.75) {$1,10,3$};
\node[scale=0.4] at (5.75,1.75) {$1,11,3$};
\node[scale=0.4] at (6.25,1.75) {$1,12,3$};
\node[scale=0.4] at (6.75,1.75) {$1,13,3$};

\node[scale=0.4] at (0.75,1.25) {$2,1,3$};
\node[scale=0.4] at (1.25,1.25) {$2,2,3$};
\node[scale=0.4] at (1.75,1.25) {$2,3,3$};
\node[scale=0.4] at (2.25,1.25) {$2,4,3$};
\node[scale=0.4] at (2.75,1.25) {$2,5,3$};
\node[scale=0.4] at (3.25,1.25) {$2,6,3$};
\node[scale=0.4] at (3.75,1.25) {$2,7,3$};
\node[scale=0.4] at (4.25,1.25) {$2,8,3$};
\node[scale=0.4] at (4.75,1.25) {$2,9,3$};
\node[scale=0.4] at (5.25,1.25) {$2,10,3$};
\node[scale=0.4] at (5.75,1.25) {$2,11,3$};
\node[scale=0.4] at (6.25,1.25) {$2,12,3$};
\node[scale=0.4] at (6.75,1.25) {$2,13,3$};

\node[scale=0.4] at (0.75,0.75) {$3,1,3$};
\node[scale=0.4] at (1.25,0.75) {$3,2,3$};
\node[scale=0.4] at (1.75,0.75) {$3,3,3$};
\node[scale=0.4] at (2.25,0.75) {$3,4,3$};
\node[scale=0.4] at (2.75,0.75) {$3,5,3$};
\node[scale=0.4] at (3.25,0.75) {$3,6,3$};
\node[scale=0.4] at (3.75,0.75) {$3,7,3$};
\node[scale=0.4] at (4.25,0.75) {$3,8,3$};
\node[scale=0.4] at (4.75,0.75) {$3,9,3$};
\node[scale=0.4] at (5.25,0.75) {$3,10,3$};
\node[scale=0.4] at (5.75,0.75) {$3,11,3$};
\node[scale=0.4] at (6.25,0.75) {$3,12,3$};
\node[scale=0.4] at (6.75,0.75) {$3,13,3$};

\node[scale=0.4] at (0.75,0.25) {$4,1,3$};
\node[scale=0.4] at (1.25,0.25) {$4,2,3$};
\node[scale=0.4] at (1.75,0.25) {$4,3,3$};
\node[scale=0.4] at (2.25,0.25) {$4,4,3$};
\node[scale=0.4] at (2.75,0.25) {$4,5,3$};
\node[scale=0.4] at (3.25,0.25) {$4,6,3$};
\node[scale=0.4] at (3.75,0.25) {$4,7,3$};
\node[scale=0.4] at (4.25,0.25) {$4,8,3$};
\node[scale=0.4] at (4.75,0.25) {$4,9,3$};
\node[scale=0.4] at (5.25,0.25) {$4,10,3$};
\node[scale=0.4] at (5.75,0.25) {$4,11,3$};
\node[scale=0.4] at (6.25,0.25) {$4,12,3$};

\node[scale=0.4] at (0.25,2.25) {$2,1,2$};
\node[scale=0.4] at (0.25,1.25) {$2,1,2$};
\node[scale=0.4] at (0.25,0.75) {$3,1,2$};
\node[scale=0.4] at (0.25,0.25) {$4,1,2$};
\node[align=left] at (3.5,-0.5){(a) A local embedding of $H$ into $[4]\times[15]\times[3]$.};
\end{scope}
\end{tikzpicture}
\end{center}

\begin{center}
\begin{tikzpicture}
\begin{scope}[xshift=0.0cm,yshift=0.0cm]
\draw[step=0.5cm,color=black,fill=gray] (-0.001,-0.001) grid (7.501,2.001) (6.5,0.0) rectangle (7.0,0.5)(7.0,0.5) rectangle (7.5,1.0)(6.5,1.0) rectangle (7.0,1.5)(0.0,1.5) rectangle (0.5,2.0)(0.5,1.5) rectangle (1.0,2.0)(1.0,1.5) rectangle (1.5,2.0)(1.5,1.5) rectangle (2.0,2.0)(2.0,1.5) rectangle (2.5,2.0)(2.5,1.5) rectangle (3.0,2.0)(3.0,1.5) rectangle (3.5,2.0)(3.5,1.5) rectangle (4.0,2.0)(4.0,1.5) rectangle (4.5,2.0)(4.5,1.5) rectangle (5.0,2.0)(5.0,1.5) rectangle (5.5,2.0)(5.5,1.5) rectangle (6.0,2.0)(6.0,1.5) rectangle (6.5,2.0)(7.0,1.5) rectangle (7.5,2.0);
\node at (0.25,0.25) {29};\node at (0.75,0.25) {28};\node at (1.25,0.25) {27};\node at (1.75,0.25) {26};\node at (2.25,0.25) {25};\node at (2.75,0.25) {24};\node at (3.25,0.25) {23};\node at (3.75,0.25) {22};\node at (4.25,0.25) {21};\node at (4.75,0.25) {20};\node at (5.25,0.25) {19};\node at (5.75,0.25) {18};\node at (6.25,0.25) {17};\node at (7.25,0.25) {1};\node at (0.25,0.75) {26};\node at (0.75,0.75) {25};\node at (1.25,0.75) {24};\node at (1.75,0.75) {23};\node at (2.25,0.75) {22};\node at (2.75,0.75) {21};\node at (3.25,0.75) {20};\node at (3.75,0.75) {19};\node at (4.25,0.75) {18};\node at (4.75,0.75) {17};\node at (5.25,0.75) {16};\node at (5.75,0.75) {15};\node at (6.25,0.75) {14};\node at (6.75,0.75) {1};\node at (0.25,1.25) {25};\node at (0.75,1.25) {24};\node at (1.25,1.25) {23};\node at (1.75,1.25) {22};\node at (2.25,1.25) {21};\node at (2.75,1.25) {20};\node at (3.25,1.25) {19};\node at (3.75,1.25) {18};\node at (4.25,1.25) {17};\node at (4.75,1.25) {16};\node at (5.25,1.25) {15};\node at (5.75,1.25) {14};\node at (6.25,1.25) {13};\node at (7.25,1.25) {1};\node at (6.75,1.75) {1};
\node at (6.75,1.25) {X};
\node at (6.75,0.25) {X};
\node at (7.25,1.75) {X};
\node at (7.25,0.75) {X};
\end{scope}

\begin{scope}[xshift=8.0cm,yshift=0.0cm]
\draw[step=0.5cm,color=black,fill=gray] (-0.001,-0.001) grid (7.501,2.001) (0.0,0.0) rectangle (0.5,0.5)(7.0,0.0) rectangle (7.5,0.5)(0.0,0.5) rectangle (0.5,1.0)(0.0,1.0) rectangle (0.5,1.5)(0.0,1.5) rectangle (0.5,2.0)(0.5,1.5) rectangle (1.0,2.0)(1.0,1.5) rectangle (1.5,2.0)(1.5,1.5) rectangle (2.0,2.0)(2.0,1.5) rectangle (2.5,2.0)(2.5,1.5) rectangle (3.0,2.0)(3.0,1.5) rectangle (3.5,2.0)(3.5,1.5) rectangle (4.0,2.0)(4.0,1.5) rectangle (4.5,2.0)(4.5,1.5) rectangle (5.0,2.0)(5.0,1.5) rectangle (5.5,2.0)(5.5,1.5) rectangle (6.0,2.0)(6.0,1.5) rectangle (6.5,2.0);
\node at (0.75,0.25) {3};\node at (1.25,0.25) {4};\node at (1.75,0.25) {5};\node at (2.25,0.25) {6};\node at (2.75,0.25) {7};\node at (3.25,0.25) {8};\node at (3.75,0.25) {9};\node at (4.25,0.25) {10};\node at (4.75,0.25) {11};\node at (5.25,0.25) {12};\node at (5.75,0.25) {13};\node at (6.25,0.25) {16};\node at (6.75,0.25) {15};\node at (0.75,0.75) {2};\node at (1.25,0.75) {3};\node at (1.75,0.75) {4};\node at (2.25,0.75) {5};\node at (2.75,0.75) {6};\node at (3.25,0.75) {7};\node at (3.75,0.75) {8};\node at (4.25,0.75) {9};\node at (4.75,0.75) {10};\node at (5.25,0.75) {11};\node at (5.75,0.75) {12};\node at (6.25,0.75) {13};\node at (6.75,0.75) {14};\node at (7.25,0.75) {1};\node at (0.75,1.25) {1};\node at (1.25,1.25) {2};\node at (1.75,1.25) {3};\node at (2.25,1.25) {4};\node at (2.75,1.25) {5};\node at (3.25,1.25) {6};\node at (3.75,1.25) {7};\node at (4.25,1.25) {8};\node at (4.75,1.25) {9};\node at (5.25,1.25) {10};\node at (5.75,1.25) {11};\node at (6.25,1.25) {12};\node at (6.75,1.25) {15};\node at (7.25,1.25) {16};\node at (6.75,1.75) {16};\node at (7.25,1.75) {17};
\node at (7.25,0.25) {X};
\end{scope}

\begin{scope}[xshift=0.0cm,yshift=-2.5cm]
\draw[step=0.5cm,color=black,fill=gray] (-0.001,-0.001) grid (7.501,2.001) (0.0,0.0) rectangle (0.5,0.5)(0.5,0.0) rectangle (1.0,0.5)(1.0,0.0) rectangle (1.5,0.5)(1.5,0.0) rectangle (2.0,0.5)(2.0,0.0) rectangle (2.5,0.5)(2.5,0.0) rectangle (3.0,0.5)(3.0,0.0) rectangle (3.5,0.5)(3.5,0.0) rectangle (4.0,0.5)(4.0,0.0) rectangle (4.5,0.5)(4.5,0.0) rectangle (5.0,0.5)(5.0,0.0) rectangle (5.5,0.5)(5.5,0.0) rectangle (6.0,0.5)(0.0,0.5) rectangle (0.5,1.0)(0.5,0.5) rectangle (1.0,1.0)(1.0,0.5) rectangle (1.5,1.0)(1.5,0.5) rectangle (2.0,1.0)(2.0,0.5) rectangle (2.5,1.0)(2.5,0.5) rectangle (3.0,1.0)(3.0,0.5) rectangle (3.5,1.0)(3.5,0.5) rectangle (4.0,1.0)(4.0,0.5) rectangle (4.5,1.0)(4.5,0.5) rectangle (5.0,1.0)(5.0,0.5) rectangle (5.5,1.0)(5.5,0.5) rectangle (6.0,1.0)(6.0,0.5) rectangle (6.5,1.0)(7.0,0.5) rectangle (7.5,1.0)(0.0,1.0) rectangle (0.5,1.5)(0.5,1.0) rectangle (1.0,1.5)(1.0,1.0) rectangle (1.5,1.5)(1.5,1.0) rectangle (2.0,1.5)(2.0,1.0) rectangle (2.5,1.5)(2.5,1.0) rectangle (3.0,1.5)(3.0,1.0) rectangle (3.5,1.5)(3.5,1.0) rectangle (4.0,1.5)(4.0,1.0) rectangle (4.5,1.5)(4.5,1.0) rectangle (5.0,1.5)(5.0,1.0) rectangle (5.5,1.5)(5.5,1.0) rectangle (6.0,1.5)(6.0,1.0) rectangle (6.5,1.5)(0.0,1.5) rectangle (0.5,2.0)(0.5,1.5) rectangle (1.0,2.0)(1.0,1.5) rectangle (1.5,2.0)(1.5,1.5) rectangle (2.0,2.0)(2.0,1.5) rectangle (2.5,2.0)(2.5,1.5) rectangle (3.0,2.0)(3.0,1.5) rectangle (3.5,2.0)(3.5,1.5) rectangle (4.0,2.0)(4.0,1.5) rectangle (4.5,2.0)(4.5,1.5) rectangle (5.0,2.0)(5.0,1.5) rectangle (5.5,2.0)(5.5,1.5) rectangle (6.0,2.0)(6.0,1.5) rectangle (6.5,2.0);
\node at (6.25,0.25) {17};\node at (6.75,0.25) {18};\node at (7.25,0.25) {19};\node at (6.75,0.75) {15};\node at (6.75,1.25) {16};\node at (7.25,1.25) {17};\node at (6.75,1.75) {17};\node at (7.25,1.75) {18};
\node at (7.25,0.75) {X};

\node[align=left] at (7.25,-0.75){(b) The result of the $3$-neighbour process on the image of $H$ and six additional\\infections under the local embedding.};
\end{scope}

\end{tikzpicture}
\end{center}
    \caption{Diagrams used in the proof of Proposition~\ref{prop:3,4,3mod6}.}
    \label{fig:34c}
\end{figure}

\begin{prop}
\label{prop:3,6,0mod2}
If $a_3\geq 2$ and $a_3\equiv 0\bmod 2$, then $(3,6,a_3)$ is perfect. 
\end{prop}

\begin{proof}
In the case that $a_3=2$, the result follows from Proposition~\ref{prop:2,3,0mod6}. So, from here forward, we assume that $a_3\geq4$. We view the grid as $[6]\times[a_3]\times[3]$. Define 
\[A_0^1:=\{(2,1,1),(4,1,1),(6,1,1),(1,1,2),(3,1,2),(6,1,3)\}\text{ and}\]
\[A_0^3:=\{(1,a_3,1),(4,a_3,2),(6,a_3,2),(1,a_3,3),(3,a_3,3),(5,a_3,3)\}.\]
Also, for $1\leq i\leq \frac{a_3-2}{2}$, define
\[A_0^{2,i}:=\{(1,2i,1),(2,2i+1,1),(4,2i,2),(3,2i+1,2),(5,2i,3),(6,2i+1,3)\}.\]
Also, let $A_0:=A_0^1\cup\left(\bigcup_{i=0}^{(a_3-2)/2}A_0^{2,i}\right)\cup A_0^3$. We have
\[|A_0|=6+6\left(\frac{a_3-2}{2}\right)+6=\frac{3\cdot 6+3\cdot a_3+6\cdot a_3}{3}.\]
So, it suffices to show that $A_0$ percolates. See Figure~\ref{fig:360mod2} for a depiction of the case $a_3=8$. 

First, note that every vertex within $[2]\times[a_3]\times\{1\}$ becomes infected in one step except for $(2,a_3,1)$; note that the infection of $(1,1,1)$ uses the fact that $(1,1,2)$ is in $A_0$. Similarly, every vertex of $\{3,4\}\times [a_3]\times\{2\}$ becomes infected after one step, and every vertex of $\{5,6\}\times[a_3]\times\{3\}$ becomes infected after one step, except for $(5,1,3)$. 

At this point, each vertex of the form $(2,i,2)$ for $1\leq i\leq a_3-1$ is adjacent to infected vertices $(3,i,2)$ and $(2,i,1)$. In addition, $(2,1,2)$ is adjacent to $(1,1,2)$, which is in $A_0$. So, all of the vertices $(2,i,2)$ for $1\leq i\leq a_3-1$ become infected. Once these vertices are infected, we see that the vertices $(1,i,2)$ for $2\leq i\leq a_3-1$ become infected in a similar manner. The vertex $(1,a_3,2)$ now has infected neighbours $(1,a_3-1,2),(1,a_3,1)$ and $(1,a_3,3)$ and so it becomes infected, and then $(2,a_3,2)$ becomes infected as well. In a very similar way, the infection spreads to all vertices in $\{5,6\}\times[a_3]\times\{2\}$. Now, all vertices of $[6]\times[a_3]\times\{2\}$ are infected and it is not hard to see that  the other two layers eventually become infected as well.
\end{proof}

\begin{figure}[htbp]
\begin{center}
\begin{tikzpicture}
\begin{scope}[xshift=0.0cm,yshift=0.0cm]
\draw[step=0.5cm,color=black,fill=gray] (-0.001,-0.001) grid (4.001,3.001) (0.0,0.0) rectangle (0.5,0.5)(0.0,1.0) rectangle (0.5,1.5)(0.0,2.0) rectangle (0.5,2.5)(1.0,2.0) rectangle (1.5,2.5)(2.0,2.0) rectangle (2.5,2.5)(3.0,2.0) rectangle (3.5,2.5)(0.5,2.5) rectangle (1.0,3.0)(1.5,2.5) rectangle (2.0,3.0)(2.5,2.5) rectangle (3.0,3.0)(3.5,2.5) rectangle (4.0,3.0);
\node at (0.75,0.25) {13};\node at (1.25,0.25) {14};\node at (1.75,0.25) {15};\node at (2.25,0.25) {16};\node at (2.75,0.25) {17};\node at (3.25,0.25) {18};\node at (3.75,0.25) {19};\node at (0.25,0.75) {11};\node at (0.75,0.75) {12};\node at (1.25,0.75) {13};\node at (1.75,0.75) {14};\node at (2.25,0.75) {15};\node at (2.75,0.75) {16};\node at (3.25,0.75) {17};\node at (3.75,0.75) {18};\node at (0.75,1.25) {3};\node at (1.25,1.25) {4};\node at (1.75,1.25) {5};\node at (2.25,1.25) {6};\node at (2.75,1.25) {7};\node at (3.25,1.25) {8};\node at (3.75,1.25) {13};\node at (0.25,1.75) {1};\node at (0.75,1.75) {2};\node at (1.25,1.75) {3};\node at (1.75,1.75) {4};\node at (2.25,1.75) {5};\node at (2.75,1.75) {6};\node at (3.25,1.75) {7};\node at (3.75,1.75) {12};\node at (0.75,2.25) {1};\node at (1.75,2.25) {1};\node at (2.75,2.25) {1};\node at (3.75,2.25) {11};\node at (0.25,2.75) {1};\node at (1.25,2.75) {1};\node at (2.25,2.75) {1};\node at (3.25,2.75) {1};
\end{scope}

\begin{scope}[xshift=4.5cm,yshift=0.0cm]
\draw[step=0.5cm,color=black,fill=gray] (-0.001,-0.001) grid (4.001,3.001) (3.5,0.0) rectangle (4.0,0.5)(0.5,1.0) rectangle (1.0,1.5)(1.5,1.0) rectangle (2.0,1.5)(2.5,1.0) rectangle (3.0,1.5)(3.5,1.0) rectangle (4.0,1.5)(0.0,1.5) rectangle (0.5,2.0)(1.0,1.5) rectangle (1.5,2.0)(2.0,1.5) rectangle (2.5,2.0)(3.0,1.5) rectangle (3.5,2.0)(0.0,2.5) rectangle (0.5,3.0);
\node at (0.25,0.25) {9};\node at (0.75,0.25) {8};\node at (1.25,0.25) {7};\node at (1.75,0.25) {6};\node at (2.25,0.25) {5};\node at (2.75,0.25) {4};\node at (3.25,0.25) {3};\node at (0.25,0.75) {10};\node at (0.75,0.75) {7};\node at (1.25,0.75) {6};\node at (1.75,0.75) {5};\node at (2.25,0.75) {4};\node at (2.75,0.75) {3};\node at (3.25,0.75) {2};\node at (3.75,0.75) {1};\node at (0.25,1.25) {1};\node at (1.25,1.25) {1};\node at (2.25,1.25) {1};\node at (3.25,1.25) {1};\node at (0.75,1.75) {1};\node at (1.75,1.75) {1};\node at (2.75,1.75) {1};\node at (3.75,1.75) {1};\node at (0.25,2.25) {1};\node at (0.75,2.25) {2};\node at (1.25,2.25) {3};\node at (1.75,2.25) {4};\node at (2.25,2.25) {5};\node at (2.75,2.25) {6};\node at (3.25,2.25) {7};\node at (3.75,2.25) {10};\node at (0.75,2.75) {3};\node at (1.25,2.75) {4};\node at (1.75,2.75) {5};\node at (2.25,2.75) {6};\node at (2.75,2.75) {7};\node at (3.25,2.75) {8};\node at (3.75,2.75) {9};
\end{scope}

\begin{scope}[xshift=9.0cm,yshift=0.0cm]
\draw[step=0.5cm,color=black,fill=gray] (-0.001,-0.001) grid (4.001,3.001) (0.0,0.0) rectangle (0.5,0.5)(1.0,0.0) rectangle (1.5,0.5)(2.0,0.0) rectangle (2.5,0.5)(3.0,0.0) rectangle (3.5,0.5)(0.5,0.5) rectangle (1.0,1.0)(1.5,0.5) rectangle (2.0,1.0)(2.5,0.5) rectangle (3.0,1.0)(3.5,0.5) rectangle (4.0,1.0)(3.5,1.5) rectangle (4.0,2.0)(3.5,2.5) rectangle (4.0,3.0);
\node at (0.75,0.25) {1};\node at (1.75,0.25) {1};\node at (2.75,0.25) {1};\node at (3.75,0.25) {1};\node at (0.25,0.75) {11};\node at (1.25,0.75) {1};\node at (2.25,0.75) {1};\node at (3.25,0.75) {1};\node at (0.25,1.25) {12};\node at (0.75,1.25) {7};\node at (1.25,1.25) {6};\node at (1.75,1.25) {5};\node at (2.25,1.25) {4};\node at (2.75,1.25) {3};\node at (3.25,1.25) {2};\node at (3.75,1.25) {1};\node at (0.25,1.75) {13};\node at (0.75,1.75) {8};\node at (1.25,1.75) {7};\node at (1.75,1.75) {6};\node at (2.25,1.75) {5};\node at (2.75,1.75) {4};\node at (3.25,1.75) {3};\node at (0.25,2.25) {18};\node at (0.75,2.25) {17};\node at (1.25,2.25) {16};\node at (1.75,2.25) {15};\node at (2.25,2.25) {14};\node at (2.75,2.25) {13};\node at (3.25,2.25) {12};\node at (3.75,2.25) {11};\node at (0.25,2.75) {19};\node at (0.75,2.75) {18};\node at (1.25,2.75) {17};\node at (1.75,2.75) {16};\node at (2.25,2.75) {15};\node at (2.75,2.75) {14};\node at (3.25,2.75) {13};
\end{scope}

\end{tikzpicture}
\end{center}
    \caption{Diagram used in the proof of Proposition~\ref{prop:3,6,0mod2}}
    \label{fig:360mod2}
\end{figure}
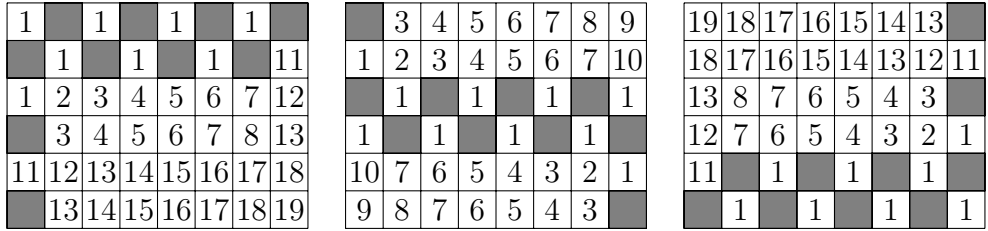

\begin{prop}
\label{prop:3,6,1mod2}
If $a_3\geq 3$ and $a_3\equiv 1\bmod 2$, then $(3,6,a_3)$ is perfect. 
\end{prop}

\begin{proof}
In the case that $a_3=3$, the result follows from Proposition~\ref{prop:3,3,0mod2}. So, from here forward, we assume that $a_3\geq5$. We view the grid as $[6]\times[a_3]\times [3]$. Define
\[A_0^1:=\{(2,1,1),(6,1,1), (1,1,2), (3,1,2),(5,1,2),(4,1,3)\}\text{ and}\]
\begin{align*}
A_0^3:= \{(2,a_3,1),&(4,a_3,1),(6,a_3-1,2),(1,a_3,2),(5,a_3,2),(1,a_3-1,3),\\
&(3,a_3-1,3),(2,a_3,3),(6,a_3,3)\}.
\end{align*}
Also, for $1\leq i\leq \frac{a_3-3}{2}$, define
\[A_0^{2,i}:=\{(1,2i,1),(1,2i+1,1),(6,2i,2),(5,2i+1,2),(3,2i,3),(4,2i+1,3)\}.\]
Now, let $A_0:=A_0^1\cup\left(\bigcup_{i=2}^{(a_3-1)/2}A_0^{2,i}\right)\cup A_0^3$. We have
\[|A_0|=6+6\left(\frac{a_3-3}{2}\right) + 9=\frac{3\cdot 6 + 3\cdot a_3+6\cdot a_3}{3}.\]
Let us show that $A_0$ percolates. See Figure~\ref{fig:361mod2} for a depiction of the case $a_3=9$. 

First, observe that all vertices of $[2]\times[a_3-2]\times\{1\}$ becomes infected in one step, apart from $(1,a_3-2,1)$; the infection of $(1,1,1)$ uses the fact that $(1,1,2)$ is in $A_0$. Similarly, all of $\{5,6\}\times[a_3]\times\{2\}$ becomes infected in one step and all of $\{3,4\}\times[a_3-1]\times\{3\}$ except for $(4,a_3-1,3)$ becomes infected. 

At this point, each vertex of the form $(4,i,2)$ for $1\leq i\leq a_3-2$ is adjacent to infected vertices $(5,i,2)$ and $(4,i,3)$. Additionally, $(4,1,2)$ is adjacent to $(3,1,2)$, which is in $A_0$. So, each vertex of the form $(4,i,2)$ for $1\leq i\leq a_3-2$ becomes infected, one after the other. Likewise, each vertex of the form $(3,i,2)$ for $1\leq i\leq a_3-2$ becomes infected, as does each vertex of the form $(2,i,2)$ for $1\leq i\leq a_3-2$. Each vertex of the form $(1,i,2)$ for $1\leq i\leq a_3-3$ also gets infected; this time, the infection does not spread immediately to $(1,a_3-2,2)$ because $(1,a_3-2,1)$ is  not currently infected. 

At this point, we have shown that every vertex in $[6]\times[a_3]\times \{2\}$ becomes infected, except for $(1,a_3-2,1)$ and seven vertices in $[6]\times\{a_3-1,a_3\}\times\{2\}$. It is not hard to see that these eight vertices all become infected; see Figure~\ref{fig:361mod2} for an illustration of this. Now, once all of $[6]\times[a_3]\times \{2\}$ is infected, it is not hard to see that the other two levels also become infected. This completes the proof. 
\end{proof}

\begin{figure}[htbp]
\begin{center}
\begin{tikzpicture}
\begin{scope}[xshift=0.0cm,yshift=0.0cm]
\draw[step=0.5cm,color=black,fill=gray] (-0.001,-0.001) grid (4.501,3.001) (0.0,0.0) rectangle (0.5,0.5)(4.0,1.0) rectangle (4.5,1.5)(0.0,2.0) rectangle (0.5,2.5)(1.0,2.0) rectangle (1.5,2.5)(2.0,2.0) rectangle (2.5,2.5)(3.0,2.0) rectangle (3.5,2.5)(4.0,2.0) rectangle (4.5,2.5)(0.5,2.5) rectangle (1.0,3.0)(1.5,2.5) rectangle (2.0,3.0)(2.5,2.5) rectangle (3.0,3.0);
\node at (0.75,0.25) {27};\node at (1.25,0.25) {28};\node at (1.75,0.25) {29};\node at (2.25,0.25) {30};\node at (2.75,0.25) {31};\node at (3.25,0.25) {32};\node at (3.75,0.25) {33};\node at (4.25,0.25) {34};\node at (0.25,0.75) {25};\node at (0.75,0.75) {26};\node at (1.25,0.75) {27};\node at (1.75,0.75) {28};\node at (2.25,0.75) {29};\node at (2.75,0.75) {30};\node at (3.25,0.75) {31};\node at (3.75,0.75) {32};\node at (4.25,0.75) {33};\node at (0.25,1.25) {24};\node at (0.75,1.25) {23};\node at (1.25,1.25) {22};\node at (1.75,1.25) {21};\node at (2.25,1.25) {20};\node at (2.75,1.25) {19};\node at (3.25,1.25) {18};\node at (3.75,1.25) {17};\node at (0.25,1.75) {23};\node at (0.75,1.75) {22};\node at (1.25,1.75) {21};\node at (1.75,1.75) {20};\node at (2.25,1.75) {19};\node at (2.75,1.75) {18};\node at (3.25,1.75) {17};\node at (3.75,1.75) {16};\node at (4.25,1.75) {15};\node at (0.75,2.25) {1};\node at (1.75,2.25) {1};\node at (2.75,2.25) {1};\node at (3.75,2.25) {11};\node at (0.25,2.75) {1};\node at (1.25,2.75) {1};\node at (2.25,2.75) {1};\node at (3.25,2.75) {13};\node at (3.75,2.75) {14};\node at (4.25,2.75) {15};
\end{scope}

\begin{scope}[xshift=5.0cm,yshift=0.0cm]
\draw[step=0.5cm,color=black,fill=gray] (-0.001,-0.001) grid (4.501,3.001) (0.5,0.0) rectangle (1.0,0.5)(1.5,0.0) rectangle (2.0,0.5)(2.5,0.0) rectangle (3.0,0.5)(3.5,0.0) rectangle (4.0,0.5)(0.0,0.5) rectangle (0.5,1.0)(1.0,0.5) rectangle (1.5,1.0)(2.0,0.5) rectangle (2.5,1.0)(3.0,0.5) rectangle (3.5,1.0)(4.0,0.5) rectangle (4.5,1.0)(0.0,1.5) rectangle (0.5,2.0)(0.0,2.5) rectangle (0.5,3.0)(4.0,2.5) rectangle (4.5,3.0);
\node at (0.25,0.25) {1};\node at (1.25,0.25) {1};\node at (2.25,0.25) {1};\node at (3.25,0.25) {1};\node at (4.25,0.25) {1};\node at (0.75,0.75) {1};\node at (1.75,0.75) {1};\node at (2.75,0.75) {1};\node at (3.75,0.75) {1};\node at (0.25,1.25) {1};\node at (0.75,1.25) {2};\node at (1.25,1.25) {3};\node at (1.75,1.25) {4};\node at (2.25,1.25) {5};\node at (2.75,1.25) {6};\node at (3.25,1.25) {7};\node at (3.75,1.25) {12};\node at (4.25,1.25) {13};\node at (0.75,1.75) {3};\node at (1.25,1.75) {4};\node at (1.75,1.75) {5};\node at (2.25,1.75) {6};\node at (2.75,1.75) {7};\node at (3.25,1.75) {8};\node at (3.75,1.75) {11};\node at (4.25,1.75) {14};\node at (0.25,2.25) {1};\node at (0.75,2.25) {4};\node at (1.25,2.25) {5};\node at (1.75,2.25) {6};\node at (2.25,2.25) {7};\node at (2.75,2.25) {8};\node at (3.25,2.25) {9};\node at (3.75,2.25) {10};\node at (4.25,2.25) {1};\node at (0.75,2.75) {5};\node at (1.25,2.75) {6};\node at (1.75,2.75) {7};\node at (2.25,2.75) {8};\node at (2.75,2.75) {9};\node at (3.25,2.75) {12};\node at (3.75,2.75) {11};
\end{scope}

\begin{scope}[xshift=10.0cm,yshift=0.0cm]
\draw[step=0.5cm,color=black,fill=gray] (-0.001,-0.001) grid (4.501,3.001) (4.0,0.0) rectangle (4.5,0.5)(0.0,1.0) rectangle (0.5,1.5)(1.0,1.0) rectangle (1.5,1.5)(2.0,1.0) rectangle (2.5,1.5)(3.0,1.0) rectangle (3.5,1.5)(0.5,1.5) rectangle (1.0,2.0)(1.5,1.5) rectangle (2.0,2.0)(2.5,1.5) rectangle (3.0,2.0)(3.5,1.5) rectangle (4.0,2.0)(4.0,2.0) rectangle (4.5,2.5)(3.5,2.5) rectangle (4.0,3.0);
\node at (0.25,0.25) {26};\node at (0.75,0.25) {25};\node at (1.25,0.25) {24};\node at (1.75,0.25) {23};\node at (2.25,0.25) {22};\node at (2.75,0.25) {21};\node at (3.25,0.25) {20};\node at (3.75,0.25) {19};\node at (0.25,0.75) {25};\node at (0.75,0.75) {24};\node at (1.25,0.75) {23};\node at (1.75,0.75) {22};\node at (2.25,0.75) {21};\node at (2.75,0.75) {20};\node at (3.25,0.75) {19};\node at (3.75,0.75) {18};\node at (4.25,0.75) {17};\node at (0.75,1.25) {1};\node at (1.75,1.25) {1};\node at (2.75,1.25) {1};\node at (3.75,1.25) {13};\node at (4.25,1.25) {16};\node at (0.25,1.75) {1};\node at (1.25,1.75) {1};\node at (2.25,1.75) {1};\node at (3.25,1.75) {1};\node at (4.25,1.75) {15};\node at (0.25,2.25) {16};\node at (0.75,2.25) {15};\node at (1.25,2.25) {14};\node at (1.75,2.25) {13};\node at (2.25,2.25) {12};\node at (2.75,2.25) {11};\node at (3.25,2.25) {10};\node at (3.75,2.25) {1};\node at (0.25,2.75) {19};\node at (0.75,2.75) {18};\node at (1.25,2.75) {17};\node at (1.75,2.75) {16};\node at (2.25,2.75) {15};\node at (2.75,2.75) {14};\node at (3.25,2.75) {13};\node at (4.25,2.75) {1};
\end{scope}

\end{tikzpicture}
\end{center}
    \caption{Diagram used in the proof of Proposition~\ref{prop:3,6,1mod2}}
    \label{fig:361mod2}
\end{figure}
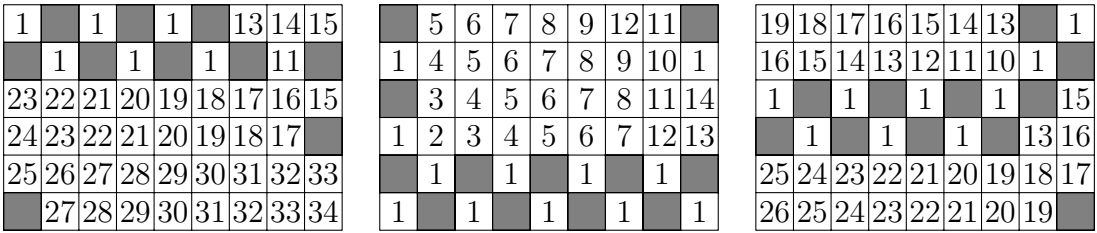

We can now characterize the perfect triples of the form $(3,6,a_3)$. 

\begin{cor}
\label{cor:3,6,c}
$(3,6,a_3)$ is perfect if and only if  $a_3\geq2$.
\end{cor}

\begin{proof}
The fact that $(3,6,1)$ is not perfect follows from Theorem~\ref{thm:thickness1}. We have that $(3,6,a_3)$ is perfect for all $a_3\geq2$ by either Proposition~\ref{prop:3,6,0mod2} or ~\ref{prop:3,6,1mod2} depending on the parity of $a_3$. 
\end{proof}

\begin{prop}
\label{prop:3,3mod6,1mod2}
If $a_2,a_3\geq 3$, $a_2\equiv 1\bmod 2$ and $a_3\equiv 3\bmod 6$, then $(3,a_2,a_3)$ is perfect. 
\end{prop}

\begin{proof}
In our diagrams, we view the grid as $[a_2]\times[a_3]\times[3]$. Let $H$ be the grid graph obtained from $(a_2+2)\times(a_3+2)$ by deleting all of the points $(x,y)$ such that either $x+y\geq a_2+a_3+1$ or $(x,y)=(a_2,a_3)$; see Figure~\ref{fig:3bc}~(a) for the case $a_2=5$ and $a_3=9$. By Lemma~\ref{lem:cutcorner}, there is a percolating set in $H$ of cardinality $\frac{(a_2+2)(a_3+2)+(a_2+2)+(a_3+2)-11}{3} = \frac{a_2a_3+3a_2+3a_3-3}{3}$. A local embedding of $H$ into $[a_2]\times [a_3]\times[3]$ in the case $a_2=5$ and $a_3=9$ is provided in Figure~\ref{fig:3bc}~(a); the local embedding in the general case is analogous. The percolating set that we consider is the image of the aforementioned percolating set in $H$ under the local embedding together with the point $(a_2,a_3,3)$; the case $a_2=5$ and $a_3=9$ is shown in Figure~\ref{fig:3bc}~(b), where the additional point is labelled with an X. The cells in that figure labelled with Y are infected after one step of the $3$-neighbour process. At this point, we can invoke~\ref{lem:sidesSplit} to get that this set percolates. The cardinality is $\frac{a_2a_3+3a_2+3a_3}{3}$, as desired.
\end{proof}

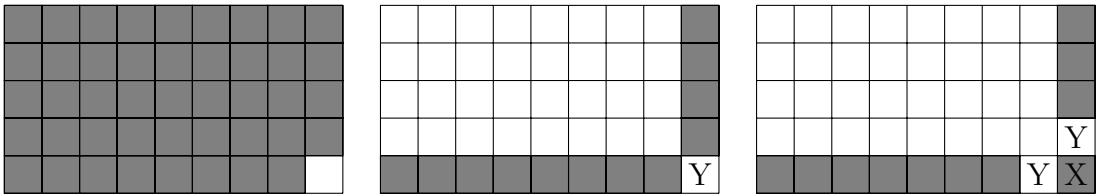
\begin{figure}[htbp]
\begin{center}
\begin{tikzpicture}
\begin{scope}[xshift=0.0cm,yshift=0.0cm,scale=1.2]
\draw[step=0.5cm,color=black]  (0.0,0.0) rectangle (0.5,0.5)(0.5,0.0) rectangle (1.0,0.5)(1.0,0.0) rectangle (1.5,0.5)(1.5,0.0) rectangle (2.0,0.5)(2.0,0.0) rectangle (2.5,0.5)(2.5,0.0) rectangle (3.0,0.5)(3.0,0.0) rectangle (3.5,0.5)(0.5,0.5) rectangle (1.0,1.0)(1.0,0.5) rectangle (1.5,1.0)(1.5,0.5) rectangle (2.0,1.0)(2.0,0.5) rectangle (2.5,1.0)(2.5,0.5) rectangle (3.0,1.0)(3.0,0.5) rectangle (3.5,1.0)(3.5,0.5) rectangle (4.0,1.0)(0.0,1.0) rectangle (0.5,1.5)(0.5,1.5) rectangle (1.0,2.0)(1.0,1.5) rectangle (1.5,2.0)(1.5,1.5) rectangle (2.0,2.0)(2.0,1.5) rectangle (2.5,2.0)(2.5,1.5) rectangle (3.0,2.0)(3.0,1.5) rectangle (3.5,2.0)(3.5,1.5) rectangle (4.0,2.0)(4.0,1.5) rectangle (4.5,2.0)(4.5,1.5) rectangle (5.0,2.0)(0.0,2.0) rectangle (0.5,2.5)(0.5,2.0) rectangle (1.0,2.5)(1.0,2.0) rectangle (1.5,2.5)(1.5,2.0) rectangle (2.0,2.5)(2.0,2.0) rectangle (2.5,2.5)(2.5,2.0) rectangle (3.0,2.5)(3.0,2.0) rectangle (3.5,2.5)(3.5,2.0) rectangle (4.0,2.5)(4.0,2.0) rectangle (4.5,2.5)(4.5,2.0) rectangle (5.0,2.5)(5.0,2.0) rectangle (5.5,2.5)(0.0,2.5) rectangle (0.5,3.0)(0.5,2.5) rectangle (1.0,3.0)(1.0,2.5) rectangle (1.5,3.0)(1.5,2.5) rectangle (2.0,3.0)(2.0,2.5) rectangle (2.5,3.0)(2.5,2.5) rectangle (3.0,3.0)(3.0,2.5) rectangle (3.5,3.0)(3.5,2.5) rectangle (4.0,3.0)(4.0,2.5) rectangle (4.5,3.0)(4.5,2.5) rectangle (5.0,3.0)(5.0,2.5) rectangle (5.5,3.0)(0.0,3.0) rectangle (0.5,3.5)(0.5,3.0) rectangle (1.0,3.5)(1.0,3.0) rectangle (1.5,3.5)(1.5,3.0) rectangle (2.0,3.5)(2.0,3.0) rectangle (2.5,3.5)(2.5,3.0) rectangle (3.0,3.5)(3.0,3.0) rectangle (3.5,3.5)(3.5,3.0) rectangle (4.0,3.5)(4.0,3.0) rectangle (4.5,3.5)(4.5,3.0) rectangle (5.0,3.5)(5.0,3.0) rectangle (5.5,3.5)(0.0,0.5) rectangle (0.5,1)(0.0,1.5) rectangle (0.5,2)(0.5,1) rectangle (1,1.5)(1,1) rectangle (1.5,1.5)(1.5,1) rectangle (2,1.5)(2,1) rectangle (2.5,1.5)(2.5,1) rectangle (3,1.5)(3,1) rectangle (3.5,1.5)(3.5,1) rectangle (4,1.5);
\node[scale=0.4] at (0.25,3.25) {$1,1,1$};
\node[scale=0.4] at (0.75,3.25) {$1,2,1$};
\node[scale=0.4] at (1.25,3.25) {$1,3,1$};
\node[scale=0.4] at (1.75,3.25) {$1,4,1$};
\node[scale=0.4] at (2.25,3.25) {$1,5,1$};
\node[scale=0.4] at (2.75,3.25) {$1,6,1$};
\node[scale=0.4] at (3.25,3.25) {$1,7,1$};
\node[scale=0.4] at (3.75,3.25) {$1,8,1$};
\node[scale=0.4] at (4.25,3.25) {$1,9,1$};
\node[scale=0.4] at (4.75,3.25) {$1,9,2$};
\node[scale=0.4] at (5.25,3.25) {$1,9,3$};

\node[scale=0.4] at (0.25,2.75) {$2,1,1$};
\node[scale=0.4] at (0.75,2.75) {$2,2,1$};
\node[scale=0.4] at (1.25,2.75) {$2,3,1$};
\node[scale=0.4] at (1.75,2.75) {$2,4,1$};
\node[scale=0.4] at (2.25,2.75) {$2,5,1$};
\node[scale=0.4] at (2.75,2.75) {$2,6,1$};
\node[scale=0.4] at (3.25,2.75) {$2,7,1$};
\node[scale=0.4] at (3.75,2.75) {$2,8,1$};
\node[scale=0.4] at (4.25,2.75) {$2,9,1$};
\node[scale=0.4] at (4.75,2.75) {$2,9,2$};
\node[scale=0.4] at (5.25,2.75) {$2,9,3$};

\node[scale=0.4] at (0.25,2.25) {$3,1,1$};
\node[scale=0.4] at (0.75,2.25) {$3,2,1$};
\node[scale=0.4] at (1.25,2.25) {$3,3,1$};
\node[scale=0.4] at (1.75,2.25) {$3,4,1$};
\node[scale=0.4] at (2.25,2.25) {$3,5,1$};
\node[scale=0.4] at (2.75,2.25) {$3,6,1$};
\node[scale=0.4] at (3.25,2.25) {$3,7,1$};
\node[scale=0.4] at (3.75,2.25) {$3,8,1$};
\node[scale=0.4] at (4.25,2.25) {$3,9,1$};
\node[scale=0.4] at (4.75,2.25) {$3,9,2$};
\node[scale=0.4] at (5.25,2.25) {$3,9,3$};

\node[scale=0.4] at (0.25,1.75) {$4,1,1$};
\node[scale=0.4] at (0.75,1.75) {$4,2,1$};
\node[scale=0.4] at (1.25,1.75) {$4,3,1$};
\node[scale=0.4] at (1.75,1.75) {$4,4,1$};
\node[scale=0.4] at (2.25,1.75) {$4,5,1$};
\node[scale=0.4] at (2.75,1.75) {$4,6,1$};
\node[scale=0.4] at (3.25,1.75) {$4,7,1$};
\node[scale=0.4] at (3.75,1.75) {$4,8,1$};
\node[scale=0.4] at (4.25,1.75) {$4,9,1$};
\node[scale=0.4] at (4.75,1.75) {$4,9,2$};

\node[scale=0.4] at (0.25,1.25) {$5,1,1$};
\node[scale=0.4] at (0.75,1.25) {$5,2,1$};
\node[scale=0.4] at (1.25,1.25) {$5,3,1$};
\node[scale=0.4] at (1.75,1.25) {$5,4,1$};
\node[scale=0.4] at (2.25,1.25) {$5,5,1$};
\node[scale=0.4] at (2.75,1.25) {$5,6,1$};
\node[scale=0.4] at (3.25,1.25) {$5,7,1$};
\node[scale=0.4] at (3.75,1.25) {$5,8,1$};

\node[scale=0.4] at (0.25,0.75) {$5,1,2$};
\node[scale=0.4] at (0.75,0.75) {$5,2,2$};
\node[scale=0.4] at (1.25,0.75) {$5,3,2$};
\node[scale=0.4] at (1.75,0.75) {$5,4,2$};
\node[scale=0.4] at (2.25,0.75) {$5,5,2$};
\node[scale=0.4] at (2.75,0.75) {$5,6,2$};
\node[scale=0.4] at (3.25,0.75) {$5,7,2$};
\node[scale=0.4] at (3.75,0.75) {$5,8,2$};

\node[scale=0.4] at (0.25,0.25) {$5,1,3$};
\node[scale=0.4] at (0.75,0.25) {$5,2,3$};
\node[scale=0.4] at (1.25,0.25) {$5,3,3$};
\node[scale=0.4] at (1.75,0.25) {$5,4,3$};
\node[scale=0.4] at (2.25,0.25) {$5,5,3$};
\node[scale=0.4] at (2.75,0.25) {$5,6,3$};
\node[scale=0.4] at (3.25,0.25) {$5,7,3$};
\node[align=left] at (2.75,-0.5){(a) A local embedding of $H$ into $[5]\times[9]\times[3]$.};
\end{scope}
\end{tikzpicture}
\end{center}

\begin{center}
\begin{tikzpicture}
\begin{scope}[xshift=0.0cm,yshift=0.0cm]
\draw[step=0.5cm,color=black,fill=gray] (-0.001,-0.001) grid (4.501,2.501) (0.0,0.0) rectangle (0.5,0.5)(0.5,0.0) rectangle (1.0,0.5)(1.0,0.0) rectangle (1.5,0.5)(1.5,0.0) rectangle (2.0,0.5)(2.0,0.0) rectangle (2.5,0.5)(2.5,0.0) rectangle (3.0,0.5)(3.0,0.0) rectangle (3.5,0.5)(3.5,0.0) rectangle (4.0,0.5)(0.0,0.5) rectangle (0.5,1.0)(0.5,0.5) rectangle (1.0,1.0)(1.0,0.5) rectangle (1.5,1.0)(1.5,0.5) rectangle (2.0,1.0)(2.0,0.5) rectangle (2.5,1.0)(2.5,0.5) rectangle (3.0,1.0)(3.0,0.5) rectangle (3.5,1.0)(3.5,0.5) rectangle (4.0,1.0)(4.0,0.5) rectangle (4.5,1.0)(0.0,1.0) rectangle (0.5,1.5)(0.5,1.0) rectangle (1.0,1.5)(1.0,1.0) rectangle (1.5,1.5)(1.5,1.0) rectangle (2.0,1.5)(2.0,1.0) rectangle (2.5,1.5)(2.5,1.0) rectangle (3.0,1.5)(3.0,1.0) rectangle (3.5,1.5)(3.5,1.0) rectangle (4.0,1.5)(4.0,1.0) rectangle (4.5,1.5)(0.0,1.5) rectangle (0.5,2.0)(0.5,1.5) rectangle (1.0,2.0)(1.0,1.5) rectangle (1.5,2.0)(1.5,1.5) rectangle (2.0,2.0)(2.0,1.5) rectangle (2.5,2.0)(2.5,1.5) rectangle (3.0,2.0)(3.0,1.5) rectangle (3.5,2.0)(3.5,1.5) rectangle (4.0,2.0)(4.0,1.5) rectangle (4.5,2.0)(0.0,2.0) rectangle (0.5,2.5)(0.5,2.0) rectangle (1.0,2.5)(1.0,2.0) rectangle (1.5,2.5)(1.5,2.0) rectangle (2.0,2.5)(2.0,2.0) rectangle (2.5,2.5)(2.5,2.0) rectangle (3.0,2.5)(3.0,2.0) rectangle (3.5,2.5)(3.5,2.0) rectangle (4.0,2.5)(4.0,2.0) rectangle (4.5,2.5);

\end{scope}

\begin{scope}[xshift=5.0cm,yshift=0.0cm]
\draw[step=0.5cm,color=black,fill=gray] (-0.001,-0.001) grid (4.501,2.501) (0.0,0.0) rectangle (0.5,0.5)(0.5,0.0) rectangle (1.0,0.5)(1.0,0.0) rectangle (1.5,0.5)(1.5,0.0) rectangle (2.0,0.5)(2.0,0.0) rectangle (2.5,0.5)(2.5,0.0) rectangle (3.0,0.5)(3.0,0.0) rectangle (3.5,0.5)(3.5,0.0) rectangle (4.0,0.5)(4.0,0.5) rectangle (4.5,1.0)(4.0,1.0) rectangle (4.5,1.5)(4.0,1.5) rectangle (4.5,2.0)(4.0,2.0) rectangle (4.5,2.5);
\node at (4.25,0.25) {Y};
\node[align=left] at (2.25,-0.75){(b) The image of $H$ under the local embedding with one infection added (X)\\and three vertices (Y) that are infected after one step.};
\end{scope}

\begin{scope}[xshift=10.0cm,yshift=0.0cm]
\draw[step=0.5cm,color=black,fill=gray] (-0.001,-0.001) grid (4.501,2.501) (0.0,0.0) rectangle (0.5,0.5)(0.5,0.0) rectangle (1.0,0.5)(1.0,0.0) rectangle (1.5,0.5)(1.5,0.0) rectangle (2.0,0.5)(2.0,0.0) rectangle (2.5,0.5)(2.5,0.0) rectangle (3.0,0.5)(3.0,0.0) rectangle (3.5,0.5)(4.0,1.0) rectangle (4.5,1.5)(4.0,1.5) rectangle (4.5,2.0)(4.0,2.0) rectangle (4.5,2.5)(4.0,0.0) rectangle (4.5,0.5);
\node at (3.75,0.25) {Y};\node at (4.25,0.75) {Y};\node at (4.25,0.25) {X};
\end{scope}

\end{tikzpicture}
\end{center}
    \caption{Diagrams used in the proof of Proposition~\ref{prop:3,3mod6,1mod2}.}
    \label{fig:3bc}
\end{figure}

The following corollary characterizes the perfect triples of the form $(3,3,a_3)$. 

\begin{cor}
\label{cor:3,3,c}
$(3,3,a_3)$ is perfect if and only if  $a_3\neq 2$.
\end{cor}

\begin{proof}
If $a_3=1$, then $(3,3,a_3)$ is perfect by the second construction in Figure~\ref{fig:Purina}. The fact that $(3,3,2)$ is not perfect follows from Corollary~\ref{cor:2,3,0mod3}. If $a_3\neq2$, then $(3,3,a_3)$ is perfect by either Proposition~\ref{prop:3,3mod6,1mod2} or~\ref{prop:3,3,0mod2}, depending on the parity of $a_3$.
\end{proof}

\section{Recursive Constructions}
\label{sec:recursion}

To complete the proof of Theorem~\ref{thm:largeEnough}, we will blend together the constructions from the previous section and a few sporadic examples from Appendix~\ref{app:sporadic} using a recursive strategy to show that we can handle all sufficiently large choices of parameters. This recursive strategy makes use of the so-called \emph{modified bootstrap process} in $[n]^d$, which is similar to the $d$-neighbour process except that a vertex  becomes infected if it is adjacent to infected vertices via at least one edge in all of the $d$ possible directions. For instance, in two dimensions, a vertex becomes infected if it has at least one infected neighbour to the North or South  and at least one to the West or East, but does not become infected if it only has one to the West and one to the East.  This process is fairly well studied; in particular, Holroyd~\cite{Holroyd03,Holroyd06} obtained sharp estimates on the critical probability of this process. 

We state the following lemma in full generality, but remark that we will only apply it in the cases $d=3$ and $n=2$.  This special case is highlighted by the corollary that follows it.

\begin{lem}
\label{lem:recursion}
For $n\geq1$ and $d\geq1$, let $A=(a_{i,j})$ be a $d\times n$ matrix of positive integers and, for $1\leq i\leq d$, let $a_i=\sum_{j=1}^na_{i,j}$. Let $S$ be a percolating set for the modified bootstrap process in $[n]^d$ and, for each $x=(x_1,\dots,x_d)\in S$, let $T_{x}$ be a percolating set for the $d$-neighbour process in $\prod_{i=1}^d[a_{i,x_i}]$. Then
\[m(a_1,a_2,\dots,a_d;d)\leq \sum_{x\in S}\left|T_{x}\right|.\]
\end{lem}

\begin{proof}
We view the cube $\prod_{i=1}^d[a_i]$ as being partitioned into $n^d$ cubes by viewing the $i$th dimension as being partitioned into intervals of length $a_{i,1},a_{i,2},\dots,a_{i,n}$ and considering products of these intervals. Formally, for $(x_1,x_2,\dots,x_d)\in [n]^d$, we let $G_{x}$ be the grid with vertex set
\[\prod_{i=1}^d\left\{1+\sum_{j=1}^{x_i-1}a_{i,j},\dots,\sum_{j=1}^{x_i}a_{i,j}\right\}.\]
We note that $G_{x}$ is isomorphic to the cube $\prod_{i=1}^d[a_{i,x_i}]$. 

Now, for every $x\in S$, suppose that we infect the vertices of $G_{x}$ corresponding to the vertices of $T_{x}$ under the natural isomorphism from $\prod_{i=1}^d[a_{i,x_i}]$ to $G_{x}$. Let $A_0$ be the set of all such vertices. To see that $A_0$ percolates in $\prod_{i=1}^d[a_i]$, we start by running the $d$-neighbour process for $t$ steps where $t$ is the minimum integer such that $A_t\supseteq\bigcup_{x\in S}G_{x}$. Such an integer $t$ clearly exists by construction of $A_0$.

Our final task is to show that $A_t$ percolates in $\prod_{i=1}^d[a_i]$. Let $x_1,\dots,x_{n^d}$ be the vertices of $[n]^d$ listed so that $x_1,\dots,x_{|S|}$ are the vertices of $S$ and, for $|S|+1\leq i\leq n^d$, $x_i$ has neighbours in $\{x_1,\dots,x_{i-1}\}$ via edges in at least $d$ different directions. Now, for $1\leq i\leq n^d$, we show that every vertex of $G_{x_i}$ is eventually infected by induction on $i$. As we concluded in the previous paragraph, every vertex of $G_{x_i}$ for $1\leq i\leq |S|$ is in $A_t$; this settles the base case. For $i\geq |S|+1$, without loss of generality, $\{x_1,\dots,x_{i-1}\}$ contains the neighbour of $x_i$ whose $j$th coordinate is smaller than the $j$th coordinate of $x_i$ for all $j\in [d]$. Let $H_{x_i}$ be the grid with vertex set 
\[\prod_{i=1}^d\left\{\sum_{j=1}^{x_i-1}a_{i,j},\dots,\sum_{j=1}^{x_i}a_{i,j}\right\}.\]
By induction on $i$, for every $j\in [d]$, at least one of the faces of $H_{x_i}$ in direction $j$ eventually becomes fully infected. Thus, by Lemma~\ref{lem:sidesSplit}, $H_{x_i}$ eventually becomes fully infected, and so $G_{x_i}$ becomes fully infected. This completes the proof. 
\end{proof}

The following corollary describes the main way in which we apply Lemma~\ref{lem:recursion}. 

\begin{cor}
\label{cor:recursive2}
Let $A=(a_{i,j})$ be a $3\times 2$ matrix of positive integers. Let $a_i=a_{i,1}+a_{i,2}$ for $1\leq i\leq 3$. Then $m(a_1,a_2,a_3;3)$ is at most
\[m(a_{1,1},a_{2,1}, a_{3,1};3)+m(a_{1,1},a_{2,2}, a_{3,2};3)+m(a_{1,2},a_{2,1}, a_{3,2};3)+m(a_{1,2},a_{2,2}, a_{3,1};3).\]
Moreover, if $(a_{1,1},a_{2,1}, a_{3,1})$ is optimal and each of the triples $(a_{1,1},a_{2,2}, a_{3,2}), (a_{1,2},a_{2,1}, a_{3,2})$, $(a_{1,2},a_{2,2}, a_{3,1})$ is perfect, then $(a_1,a_2,a_3)$ is optimal. 
\end{cor}

\begin{proof}
The general upper bound on $m(a_1,a_2,a_3;3)$ follows from Lemma~\ref{lem:recursion} since the vertices $(1,1,1),(1,2,2),(2,1,2)$ and $(2,2,1)$ form a set which percolates under the modified bootstrap process in $[2]^3$. Regarding the ``moreover'' part of the statement, we observe that, if $(a_{1,1},a_{2,1}, a_{3,1})$ is optimal and all of $(a_{1,1},a_{2,2}, a_{3,2})$, $(a_{1,2},a_{2,1}, a_{3,2})$ and $(a_{1,2},a_{2,2}, a_{3,1})$ are perfect, then $\left\lceil\frac{a_{1}a_{2}+a_{1}a_{3}+a_{2},a_{3}}{3}\right\rceil$ can be rewritten as
\[\left\lceil\frac{(a_{1,1}+a_{1,2})(a_{2,1}+a_{2,2})+(a_{1,1}+a_{1,2})(a_{3,1}+a_{3,2})+(a_{2,1}+a_{2,2})(a_{3,1}+a_{3,2})}{3}\right\rceil\]
\[=\left\lceil\frac{a_{1,1}a_{2,1}+a_{1,1}a_{3,1}+a_{2,1},a_{3,1}}{3}\right\rceil +\frac{a_{1,1}a_{2,2}+a_{1,1}a_{3,2}+a_{2,2},a_{3,2}}{3}\]
\[+\frac{a_{1,2}a_{2,1}+a_{1,2}a_{3,2}+a_{2,1},a_{3,2}}{3}+\frac{a_{1,2}a_{2,2}+a_{1,2}a_{3,1}+a_{2,2},a_{3,1}}{3}.\]
This completes the proof.
\end{proof}

We remark that, in order to apply Lemma~\ref{lem:recursion} to obtain optimal bounds (as in Corollary~\ref{cor:recursive2}), one requires a percolating set of cardinality $n^{d-1}$ for the modified bootstrap process in $[n]^d$. For $n=2$, this is a triviality, as we saw in the proof of Corollary~\ref{cor:recursive2}. For larger $n$, it seems likely that the construction in the proofs of~\cite{Winkler07,PrzykuckiShelton20} of the matching upper bound to \eqref{eq:n^d-1} actually percolates under the modified process; however, we have not verified this. If so, then their construction can be used in Lemma~\ref{lem:recursion}. However, Corollary~\ref{cor:recursive2} will be sufficient for us.

In the special case that $(a_1,a_2,a_3)$ is class zero, we will prove the following stronger form of Theorem~\ref{thm:largeEnough}. 

\begin{thm}
\label{thm:divisible}
If $a_1,a_2,a_3\geq 5$ and $(a_1,a_2,a_3)$ is class $0$, then
\[m(a_1,a_2,a_3;3)=\frac{a_1a_2+a_1a_3+a_2a_3}{3}.\]
\end{thm}

As a simple application of Corollary~\ref{cor:recursive2}, let us now derive Theorem~\ref{thm:largeEnough} from Theorem~\ref{thm:divisible}; the rest of the section will then be devoted to proving Theorem~\ref{thm:divisible}.

\begin{proof}[Proof of Theorem~\ref{thm:largeEnough}]
Let $a_1,a_2,a_3\geq11$ be given and choose $a_{1,1},a_{2,1},a_{3,1}\in \{5,6,7\}$ such that $a_{i,1}\equiv a_i\bmod 3$ for all $i\in [3]$. Also, define $a_{i,2}:=a_i-a_{i,1}$ for $i\in [3]$. Our goal is to argue that all of $(a_{1,1},a_{2,2},a_{3,2}),(a_{1,2},a_{2,1},a_{3,2})$ and $(a_{1,2},a_{2,2},a_{3,1})$ are perfect and that $(a_{1,1},a_{2,1},a_{3,1})$ is optimal. Given this, we will be done by Corollary~\ref{cor:recursive2}. 

By definition, each of $a_{1,2},a_{2,2}$ and $a_{3,2}$ is a multiple of three. Therefore, 
\[(a_{1,1},a_{2,2},a_{3,2}),(a_{1,2},a_{2,1},a_{3,2})\text{ and }(a_{1,2},a_{2,2},a_{3,1})\]
are all class zero. Also, by definition, and the fact that $a_1,a_2,a_3\geq11$, all of the integers $a_{i,j}$ for $(i,j)\in [3]\times [2]$ are at least $5$. Therefore, by Theorem~\ref{thm:divisible}, $(a_{1,1},a_{2,2},a_{3,2}),(a_{1,2},a_{2,1},a_{3,2})$ and $(a_{1,2},a_{2,2},a_{3,1})$ are perfect.

Thus, all that remains is to show that $(a,b,c)$ is optimal whenever $a,b,c\in\{5,6,7\}$. If $a=b=c$ or two of $a,b,c$ are equal to $6$, then $(a,b,c)$ is class $0$ and so, in these cases, the result follows from Theorem~\ref{thm:divisible}.\footnote{Note that the case $a=b=c$ also follows from Theorem~\ref{thm:nnn}.} Up to symmetry, the remaining cases are $(5,5,6), (5,5,7), (5,6,7), (5,7,7)$ and $(6,7,7)$. Ad hoc constructions (found by computer) for each of the first four cases are provided in the appendix as Constructions~\ref{const:556},~\ref{const:557},~\ref{const:567} and~\ref{const:577}, respectively. For the case $(a,b,c)=(6,7,7)$, we apply Corollary~\ref{cor:recursive2} with 
\[\begin{bmatrix}a_{1,1} & a_{1,2}\\ a_{2,1} & a_{2,2}\\a_{3,1} & a_{3,2}\end{bmatrix} = \begin{bmatrix}3 & 3\\ 3 & 4\\3 & 4\end{bmatrix}.\]
For this, we use the fact that the tuples $(3,3,3),(3,3,4)$ and $(3,4,3)$ are perfect by Corollary~\ref{cor:3,3,c} and $(3,4,4)$ is optimal by Construction~\ref{const:344}.
\end{proof}

Our final goal is to prove Theorem~\ref{thm:divisible}. The following lemma allows us to reduce to the case that one of $a_1,a_2,a_3$ is in $\{5,6,7\}$.

\begin{lem}
\label{lem:add3}
Suppose $a_1,a_2,a_3\in\mathbb{N}\setminus\{2\}$.  If $(a_1,a_2,a_3)$ is optimal, then $(a_1+3,a_2+3,a_3+3)$ is optimal.  Moreover, if $(a_1,a_2,a_3)$ is perfect, then $(a_1+3,a_2+3,a_3+3)$ is perfect.
\end{lem}

\begin{proof}
Define $a_{i,1}:=a_i$ for $i\in [3]$. By Corollary~\ref{cor:3,3,c}, all of the triples $(a_{1,1},3,3),(3,a_{2,1},3)$ and $(3,3,a_{3,1})$ are perfect. So, setting $a_{i,2}:=3$ for $i\in[3]$, we are done by Corollary~\ref{cor:recursive2}.
\end{proof}

We now deal with the cases in which $\min\{a_1,a_2,a_3\}$ is equal to $5,6$ or $7$ separately.

\begin{lem}
\label{lem:thickness5}
If $(5,a_2,a_3)$ is class 0 and $a_2,a_3\geq5$, then $(5,a_2,a_3)$ is perfect. 
\end{lem}

\begin{proof}
In order for $(5,a_2,a_3)$ to be class $0$, we must have $a_2 \equiv a_3 \equiv 0$ or $2 \bmod{3}$.  Our argument applies Corollary~\ref{cor:recursive2} to various choices of four perfect triples according to the values of $a_2$ and $a_3$.

Suppose $a_2,a_3 \equiv 0\bmod{3}$ and each is at least $9$. If $a_2 \not\equiv a_3 \bmod{2}$, we set up an application of Corollary~\ref{cor:recursive2} using the triples
\[(2,6,6),(2,a_2-6,a_3-6),(3,6,a_3-6),(3,a_2-6,6).\]
The first of these is perfect from doubling $(1,3,3)$ via Corollary~\ref{cor:recursive2}.  The second triple in the list is perfect from either  Corollary~\ref{cor:2,3,0mod3} (if one of $a_2$ or $a_3$ equals 9) or Proposition~\ref{prop:2,0mod3,0mod3} (if both are at least $12$).  Finally, the last two triples are perfect from Corollary~\ref{cor:3,6,c}.  Using these in Corollary~\ref{cor:recursive2}, it follows that $(5,a_2,a_3)$ is perfect.
On the other hand, if $a_2 \equiv a_3 \bmod{2}$, we instead use  
\[(2,3,6),(2,a_2-3,a_3-6),(3,3,a_3-6),(3,a_2-3,6),\]
appealing to Corollary~\ref{cor:2,3,0mod3}, Proposition~\ref{prop:2,0mod3,0mod3}, Corollary~\ref{cor:3,3,c} and Corollary~\ref{cor:3,6,c}. 
This list of triples also works for $a_2=6$ and any $a_3 \ge 12$, where here Corollary~\ref{cor:2,3,0mod3} is used for the second triple.  Up to reordering, the remaining cases are $(a_1,a_2,a_3) \in \{(5,6,6),(5,6,9)\}$.  The first of these is obtained using the perfect triples
$(1,3,3),(4,3,3)$ (two of each) in Corollary~\ref{cor:recursive2}.  Finally, a construction for $(5,6,9)$ is given in Construction~\ref{const:569}.

Suppose $a_2,a_3 \equiv 2\bmod{3}$.  If $a_2,a_3 \ge 8$ and $a_2 \not \equiv a_3 \bmod{2}$,
then we apply Corollary~\ref{cor:recursive2} using the triples
\[(2,2,2),(2,a_2-2,a_3-2),(3,2,a_3-2),(3,a_2-2,2).\]
The second triple is perfect from Proposition~\ref{prop:2,2mod3,2mod3} and the last two are perfect via Corollary~\ref{cor:2,3,0mod3}. 
The same construction also works for $a_2=a_3=8$; in this case, the second triple is $(2,6,6)$, which we recall is also perfect.

On the other hand, if $a_2,a_3$ have the same parity with one strictly larger than $8$, say $a_3 \ge 11$, then we apply Corollary~\ref{cor:recursive2} using the perfect triples
\[(2,2,5),(2,a_2-2,a_3-5),(3,2,a_3-5),(3,a_2-2,5).\]
A construction for $(2,2,5)$ appears in Construction~\ref{const:225}, and otherwise the same results as above apply.  These triples also work for $a_2=5$, provided $a_3 \ge 11$.
For the remaining cases, $(5,5,5)$ is perfect by Theorem~\ref{thm:nnn}, and a construction for $(5,5,8)$ is given in Construction~\ref{const:558}.
\end{proof}

\begin{lem}
\label{lem:thickness6}
If $(6,a_2,a_3)$ is class 0 and $a_2,a_3\geq6$, then $(6,a_2,a_3)$ is perfect. 
\end{lem}

\begin{proof}
In order for $(6,a_2,a_3)$ to be class 0, at least one of $a_2$ or $a_3$ must be divisible by three. First, suppose that $a_2$ and $a_3$ are both even. Then, by Proposition~\ref{prop:3,3mod6,1mod2} and the fact that $a_2,a_3\geq6$ and at least one of $a_2$ or $a_3$ is a multiple of $6$, we get that each of the following four triples is perfect:
\[(3,3,3), (3,a_2-3,a_3-3), (3,3,a_3-3), (3,a_2-3,3).\]
It follows that $(6,a_2,a_3)$ is perfect by Corollary~\ref{cor:recursive2}. 

Next, suppose that $a_2$ is even with $a_2 \ge 6$, and $a_3$ is odd with $a_3\geq 9$.  Consider the four triples
\[(3,3,6), (3,a_2-3,a_3-6), (3,3,a_3-6), (3,a_2-3,6).\]
These are seen to be perfect from, respectively, Proposition~\ref{prop:3,3,0mod2},  Proposition~\ref{prop:3,3mod6,1mod2} (noting that at least one of $a_2-3$ or $a_3-6$ is $3 \bmod{6}$),  
Proposition~\ref{prop:3,3,0mod2}, and Corollary~\ref{cor:3,6,c}.  Again, $(6,a_2,a_3)$ is perfect by Corollary~\ref{cor:recursive2}. 

Suppose now that $a_2$ and $a_3$ are both odd and both at least 9.  We again apply Corollary~\ref{cor:recursive2}, this time on the triples
\[(3,6,6), (3,a_2-6,a_3-6), (3,6,a_3-6), (3,a_2-6,6).\]
These are seen to be perfect from the same set of results as above. 

The only cases left to consider have one parameter equal to $7$; suppose without loss of generality that $a_2=7$. Since $(6,7,a_3)$ is class 0, $a_3$ must be divisible by 3. If $a_3$ is even, then by Propositions~\ref{prop:3,3mod6,1mod2} and~\ref{prop:3,4,3mod6}, each of the following four triples is perfect:
\[(3,3,3), (3,4,a_3-3), (3,3,a_3-3), (3,4,3).\]
If $a_3$ is odd, then the triples
\[(3,3,6), (3,4,a_3-6), (3,3,a_3-6), (3,4,6)\]
are perfect by, Propositions~\ref{prop:3,3,0mod2},~\ref{prop:3,4,3mod6}, \ref{prop:3,3mod6,1mod2} and \ref{prop:3,6,0mod2}, respectively. In either case, $(6,7,a_3)$ is perfect by Corollary~\ref{cor:recursive2}. This completes the proof of the lemma.
\end{proof}

\begin{lem}
\label{lem:thickness7}
If $(7,a_2,a_3)$ is class 0 and $a_2,a_3\geq7$, then $(7,a_2,a_3)$ is perfect. 
\end{lem}

\begin{proof}
For $(7,a_2,a_3)$ to be class $0$, we must have $a_2 \equiv a_3 \equiv 0$ or $1 \bmod{3}$.  We proceed as in the proofs of the previous two lemmas.

Suppose $a_2,a_3 \equiv 0\bmod{3}$ and each is at least $9$. We apply Corollary~\ref{cor:recursive2} using the triples
\[(5,3,3),(5,a_2-3,a_3-3),(2,3,a_3-3),(2,a_2-3,3).\]
The first triple is perfect via Corollary~\ref{cor:3,3,c}, and the second is perfect from Lemma~\ref{lem:thickness5}. The third and fourth triples are perfect from Corollary~\ref{cor:2,3,0mod3}.

Next, we consider the case $a_2,a_3 \equiv 1\bmod{3}$.  
Suppose first that $a_2=7$.  If $a_3 \equiv 1 \bmod{6}$, we decompose into triples
\[(4,4,4),(3,3,4),(3,4,a_3-4),(4,3,a_3-4).\]
These are perfect from Theorem~\ref{thm:nnn},  Corollary~\ref{cor:3,3,c}, and  
Proposition~\ref{prop:3,4,3mod6} (twice), respectively.  If $a_3 \equiv 4 \bmod{6}$, we proceed similarly using
\[(4,4,7),(3,3,7),(3,4,a_3-7),(4,3,a_3-7).\]
The first of these is perfect from Corollary~\ref{cor:recursive2} applied to $(2,2,2)$ and $(2,2,5)$; the second is perfect again via Corollary~\ref{cor:3,3,c}.

In view of the above, we may assume in what follows that $a_2,a_3 \ge 10$.  If $a_2 \not\equiv a_3 \bmod{2}$, we use a decomposition into triples
\[(2,5,5),(2,a_2-5,a_3-5),(5,5,a_3-5),(5,a_2-5,5).\]
Construction~\ref{const:255} shows that $(2,5,5)$ is perfect, the second triple is perfect from Proposition~\ref{prop:2,2mod3,2mod3}, and the last two are perfect from Lemma~\ref{lem:thickness5}.
Suppose now that $a_2,a_3$ have the same parity and $(a_2,a_3) \neq (10,10)$, say without loss of generality $a_3 \ge 13$.  In this case, we use perfect triples
\[(2,5,8),(2,a_2-5,a_3-8),(5,5,a_3-8),(5,a_2-5,8),\]
with similar ingredient results as above. The remaining case $(7,10,10)$ is settled using perfect triples $(2,5,5),(5,5,5)$ (twice each).
This completes the proof of the lemma.
\end{proof}

We are now ready to establish our result for perfect triples.

\begin{proof}[Proof of Theorem~\ref{thm:divisible}]
We assume $a_1 = \min\{a_1,a_2,a_3\}$ and proceed by induction on $a_1$.  If $a_1 \in \{5,6,7\}$, we are done by Lemmas~\ref{lem:thickness5}, \ref{lem:thickness6} and \ref{lem:thickness7}.  Suppose $a_1 \ge 8$ and that the result holds for all class $0$ triples $(a_1',a_2',a_3')$ with $5 \le \min\{a_1',a_2',a_3'\} < a_1$.  In particular,
$(a_1-3,a_2-3,a_3-3)$ is perfect.  It follows from Lemma~\ref{lem:add3} that $(a_1,a_2,a_3)$ is perfect.
\end{proof}

\section{Conclusion}
\label{sec:concl}

We conjecture that Theorem~\ref{thm:largeEnough} generalizes to higher dimensions. 

\begin{conj}
\label{conj:higherDim}
For every $d\geq4$, there exists a positive integer $N_d$ such that if
\[a_1,a_2,\dots,a_d\geq N_d,\]
then 
\[m(a_1,\dots,a_d;d) =\left\lceil\frac{\sum_{j=1}^d\prod_{i\neq j}^da_i}{d}\right\rceil.\]
\end{conj}

Also, in the case $d=3$, we conjecture that Theorem~\ref{thm:largeEnough} can be extended to nearly any choice of parameters $a_1,a_2,a_3$. 

\begin{conj}
\label{conj:a1a2a3>=3}
If $a_1,a_2,a_3\geq3$, then
\[m(a_1,a_2,a_3;3)= \left\lceil\frac{a_1a_2 + a_1a_3+a_2a_3}{3}\right\rceil.\]
\end{conj}

When one or more of the parameters are very small, the conclusion of Conjecture~\ref{conj:a1a2a3>=3} no longer holds. For example, Theorem~\ref{thm:thickness1} implies that, if $a_1=1$, then it it fails to hold for a vast majority of $a_2,a_3\geq 1$. It is also not hard to show that any percolating set for the $3$-neighbour process in $[2]\times[2]\times[n]$ must contain at least three infected points on any two consecutive levels. Therefore,
\[m(2,2,n;3) = 3n/2 -O(1)\]
whereas the lower bound from Proposition~\ref{prop:SA} is $4n/3+O(1)$, and so $(2,2,n)$ is not optimal for large $n$. At the same time, it seems that the conclusion of Conjecture~\ref{conj:a1a2a3>=3} could be true when $a_1=2$, $a_2>2$ and $a_3$ is relatively large. The smallest class $0$ tuple of the form $(2,a_2,a_3)$ which we have been unable to show is perfect is $(2,5,17)$. We will not speculate on the exact characterization of optimal triples $(a_1,a_2,a_3)$, but note that it would be interesting to determine it. 

While Corollary~\ref{cor:sqGrids} determines $m(n,n;3)$ exactly for all $n$, it leaves open the problem of determining the optimal size of a percolating set in a general $2$-dimensional rectangular grid. 

\begin{prob}
\label{prob:ab3}
Determine $m(a_1,a_2;3)$ for all $a_1,a_2\geq1$. 
\end{prob}

In~\cite{ShapiroStephens91}, it was shown that the number of optimal percolating sets for the modified bootstrap process in $[n]^2$ is counted by the Schr\"oder Numbers (OEIS A006318~\cite{SchroederOEIS}). The related problem of bounding the number of optimal percolating sets for the $2$-neighbour process in $[n]^2$ is an intriguing open problem which, to our knowledge, was first posed by Bollob\'as (unpublished). In another direction, it would be interesting to bound the number of optimal percolating sets for the modified bootstrap process in $[n]^3$.

\begin{prob}
\label{prob:countModified}
Enumerate or approximate the number percolating sets of cardinality $n^2$ for the modified bootstrap process in $[n]^3$. 
\end{prob}

Of course, the generalization of Problem~\ref{prob:countModified} to $[n]^d$ for all $d\geq3$ is also interesting. In another direction, one could consider the asymptotics for a fixed value of $n$ as $d\to\infty$. Note that every percolating set of order $n^2$ for the modified bootstrap process in $[n]^3$ contains exactly one point in every line parallel to any of the three axes. That is, each percolating set corresponds to a Latin square. This is analogous to the $2$-dimensional case in which each optimal percolating set for the modified process corresponds to a permutation matrix. However, if the two dimensional case is anything to go by, then one should expect the number of percolating sets to be far smaller than the number of Latin squares (the growth rate of the Schr\"oder Numbers is $(3+2\sqrt{2})^{(1+o(1))n}$~\cite{SchroederOEIS}).  Table~\ref{table:modified} lists the number of optimal percolating sets for the modified process in $[n]^3$ for small values of $n$; these were found via an exhaustive computer search. For comparison, we also include the number of Latin squares of order $n$, as reported in OEIS A002860~\cite{LatinOEIS}.

\begin{table}[htbp]
\begin{center}
\begin{tabular}{l|l|l}
$n$ & \# of percolating sets & \# of Latin squares\\\hline 
1 & 1 & 1\\
2 & 2 & 2\\
3 & 12 & 12\\
4 & 256 & 576\\
5 & 2688 & 161280\\
6 & 148958 & 812851200
\end{tabular}
\end{center}
\caption{The number of percolating sets of cardinality $n^2$ for the modified bootstrap process in $[n]^3$ and the number of Latin squares of order $n$ for $n\in[6]$.}
\label{table:modified}
\end{table}

Let us conclude by mentioning a consequence of Theorem~\ref{thm:largeEnough} and one more open problem. For $a_1,\dots,a_d\geq3$, define $T(a_1,\dots,a_d)$ to be the graph obtained from $\prod_{i=1}^d[a_i]$ by adding an edge from $u=(u_1,\dots,u_d)$ to $v=(v_1,\dots,v_d)$ if there exists $i\in[d]$ such that $u_i=1$, $v_i=a_i$ and $u_j=v_j$ for all $j\neq i$. The graph $T(a_1,\dots,a_d)$ is referred to as a \emph{$d$-dimensional torus}. In~\cite[Theorem~2]{Benevides+21+}, it is shown that
\begin{equation}\label{eq:nntorus}m(T(n,n);3)=\left\lceil\frac{n^2+1}{3}\right\rceil\end{equation}
for all $n\geq3$. A general lower bound on $m(T(a_1,\dots,a_d);r)$ can be derived from the combination of~\cite[Equation~(1.5)]{MorrisonNoel18} and~\cite[Theorem~8]{HambardzumyanHatamiQian20}. In the case $d=r=3$, this bound reduces to
\begin{equation}\label{eq:torusLower}m(T(a_1,a_2,a_3);3)\geq \left\lceil\frac{(a_1-1)(a_2-1) + (a_1-1)(a_3-1)+(a_2-1)(a_3-1)}{3}\right\rceil+1\end{equation}
for all $a_1,a_2,a_3\geq3$. As a corollary of Theorem~\ref{thm:largeEnough}, we show that this bound is within 1 of the truth, provided that $a_1,a_2,a_3$ are large enough.

\begin{cor}
\label{cor:torusUpper}
If $a_1,a_2,a_3\geq12$, then
\[m(T(a_1,a_2,a_3);3)\leq \left\lceil\frac{(a_1-1)(a_2-1) + (a_1-1)(a_3-1)+(a_2-1)(a_3-1)}{3}\right\rceil+2.\]
\end{cor}

\begin{proof}
By Theorem~\ref{thm:largeEnough}, there is a percolating set $A_0'$ of cardinality
\[\left\lceil\frac{(a_1-1)(a_2-1) + (a_1-1)(a_3-1)+(a_2-1)(a_3-1)}{3}\right\rceil\]
for the $3$-neighbour process in $[a_1-1]\times[a_2-1]\times[a_3-1]$. We let $A_0$ be the set obtained from $A_0'$ by adding the points $(a_1,1,1)$ and $(a_1,a_2,a_3)$. We show that $A_0$ percolates in $T(a_1,a_2,a_3)$. See Figure~\ref{fig:torus} for an illustration of four key stages of the infection. 

Using $A_0'$, we see that all vertices of $[a_1-1]\times[a_2-1]\times[a_3-1]$ eventually become infected. Now, each vertex of the form $(a_1,i,j)$ for $1\leq i\leq a_2-1$ and $1\leq j\leq a_3-1$ is adjacent to the infected vertices $(a_1-1,i,j)$ and $(1,i,j)$ in the torus. Additionally, $(a_1,1,1)$ is, itself, infected. So, the infection eventually spreads to every vertex of the form $(a_1,i,j)$ for $1\leq i\leq a_2-1$ and $1\leq j\leq a_3-1$. 

At this point, each vertex of the form $(a_1,a_2,j)$ for $1\leq j\leq a_3-1$ has infected neighbours $(a_1,a_2-1,j)$ and $(a_1,1,j)$. Additionally, $(a_1,a_2,a_3-1)$ is adjacent to the vertex $(a_1,a_2,a_3)$, which is in $A_0$. So, each vertex of the form $(a_1,a_2,j)$ for $1\leq j\leq a_3-1$ becomes infected. Likewise, each vertex of the form $(a_1,i,a_3)$ for $1\leq i\leq a_2-1$ becomes infected. 

Now, each vertex of the form $(i,a_2,j)$ for $1\leq i\leq a_1-1$ and $1\leq j\leq a_3-1$ has two infected neighbours, namely $(i,a_2-1,j)$ and $(i,1,j)$. Moreover, each of the vertices $(a_1-1,a_2,j)$ for $1\leq j\leq a_3-1$ has another infected neighbour, namely $(a_1,a_2,j)$. So, all vertices of the form $(i,a_2,j)$ for $1\leq i\leq a_1-1$ and $1\leq j\leq a_3-1$ become infected. Symmetrically, every vertex of the form $(i,j,a_3)$ for $1\leq i\leq a_1-1$ and $1\leq j\leq a_2-1$ gets infected. The remaining healthy vertices are $(i,a_2,a_3)$ for $1\leq i\leq a_1-1$. Each of these vertices currently has more than three infected neighbours and so they all become infected. 

\begin{figure}
\begin{center}
\includegraphics[scale=0.25]{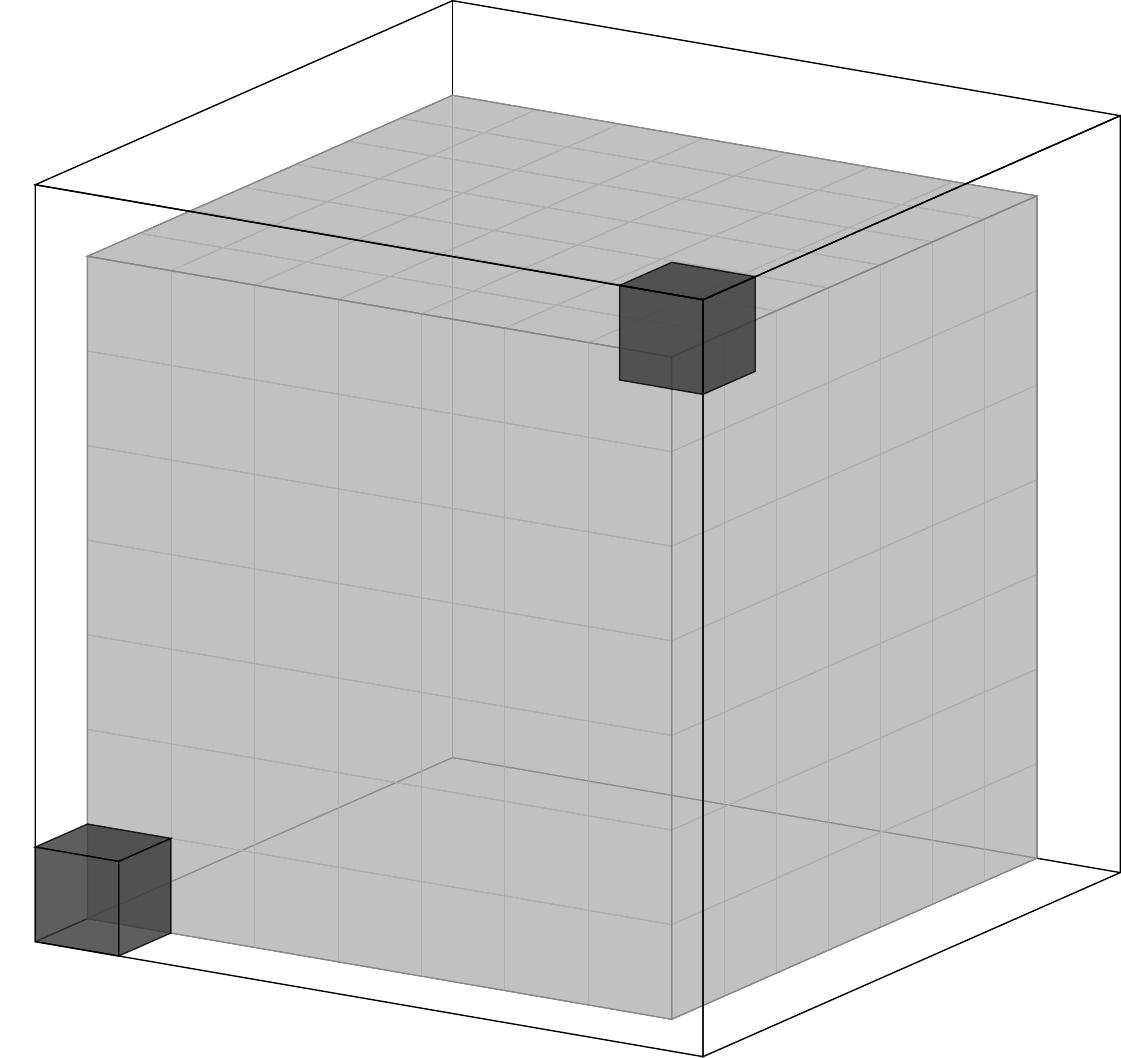}\hspace{1cm}
\includegraphics[scale=0.25]{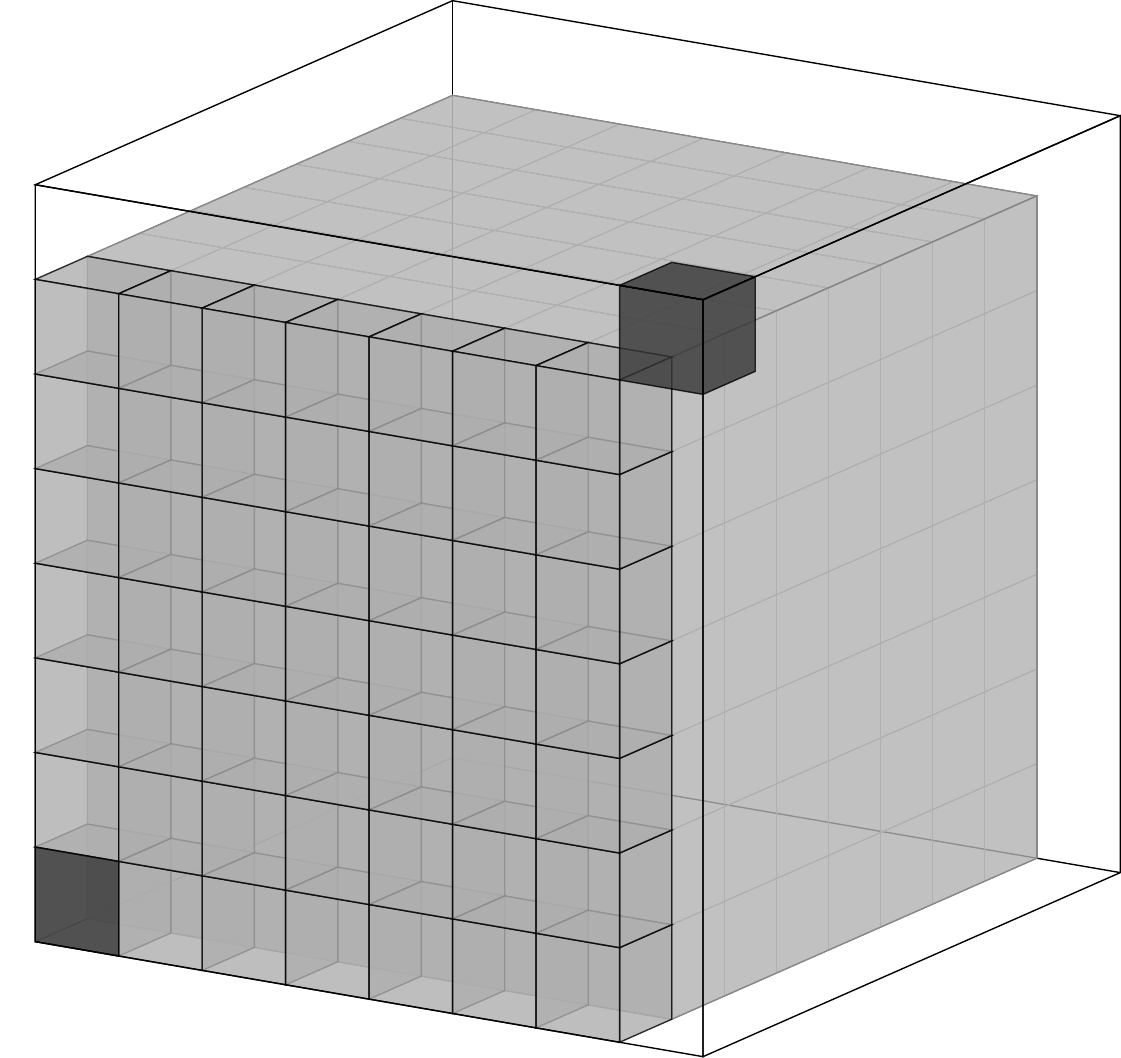}\hspace{1cm}
\includegraphics[scale=0.25]{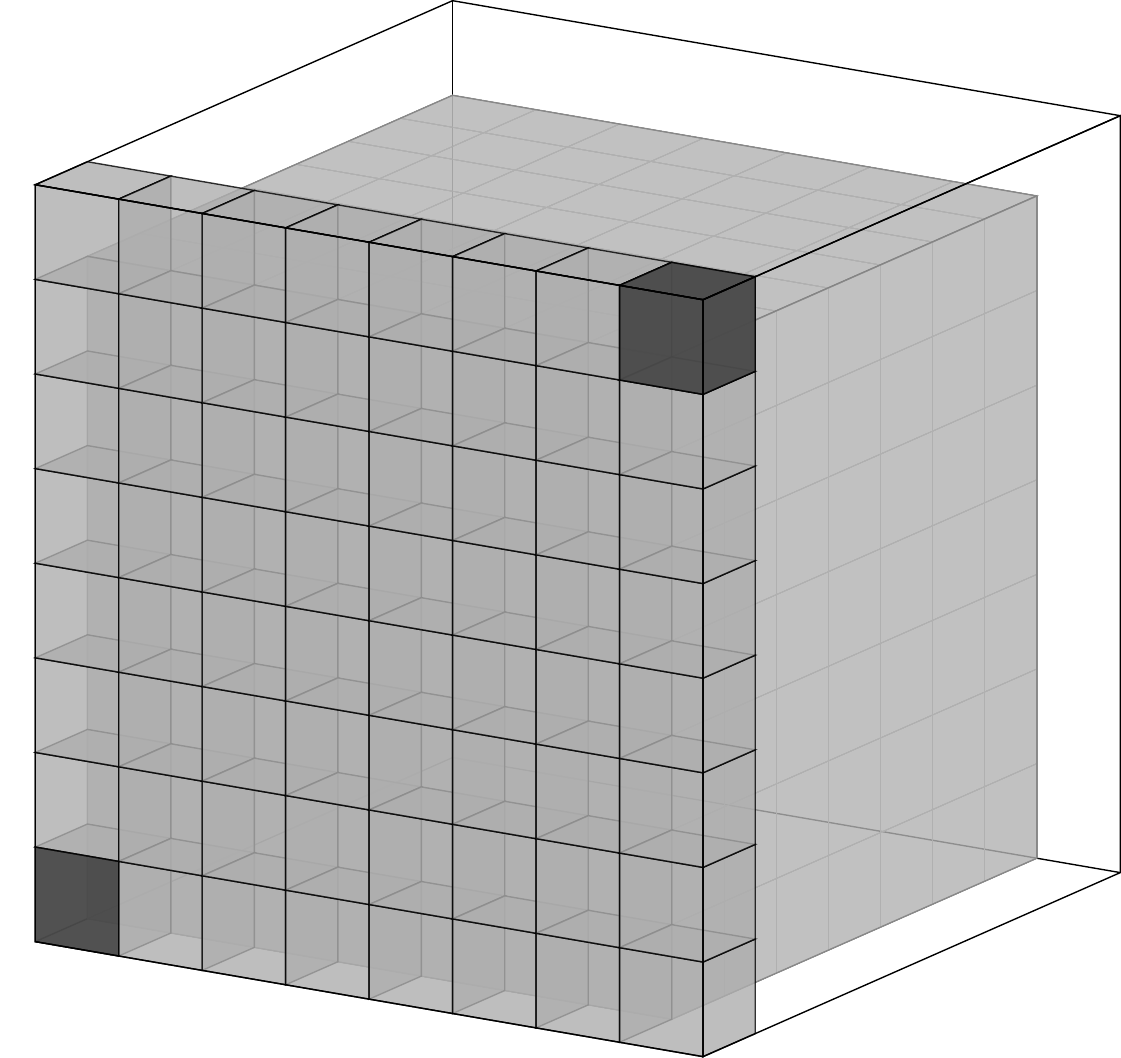}\hspace{1cm}
\includegraphics[scale=0.25]{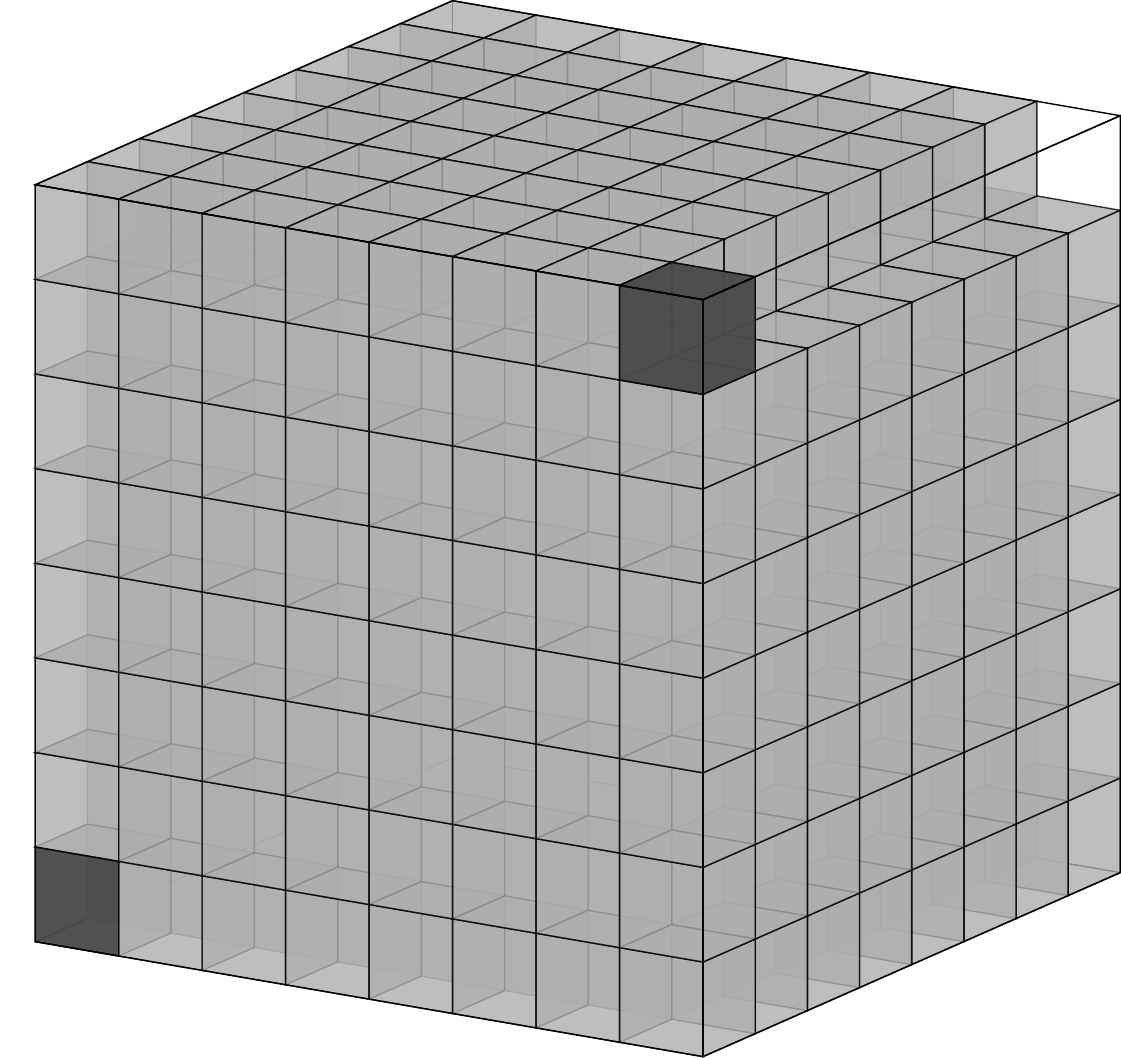}
\end{center}
    \caption{Four stages of infection in the torus $T(a_1,a_2,a_3)$. The elements of $[a_1-1]\times[a_2-1]\times[a_3-1]$ are depicted by light grey cubes outlined in grey, $(a_1,1,1)$ and $(a_1,a_2,a_3)$ are depicted by dark grey cubes and the rest of the vertices are depicted by light grey cubes outlined in black.}
    \label{fig:torus}
\end{figure}
\end{proof}

Thus, \eqref{eq:torusLower} and Corollary~\ref{cor:torusUpper} narrow down $m(T(a_1,a_2,a_3);3)$ to only two possible values when $a_1,a_2,a_3$ are at least 12. It would be interesting to determine which of these two values is correct for large $a_1,a_2,a_3$.

\begin{prob}
Determine $m(T(a_1,a_2,a_3);3)$ for all $a_1,a_2,a_3$ sufficiently large. 
\end{prob}

Of course, it would be even more interesting to determine $m(T(a_1,a_2,a_3);3)$ exactly for all choices of $a_1,a_2$ and $a_3$. Additional open problems can be found in the Master's Thesis of the third author~\cite{RomerMSc}. 

\bibliographystyle{plain}
\bibliography{bootstrap}

\begin{thebibliography}{10}

\bibitem{AdlerStaufferAharony89}
J.~Adler, D.~Stauffer, and A.~Aharony.
\newblock Comparison of bootstrap percolation models.
\newblock {\em J. Phys. A}, 22:L297–L301, 1989.

\bibitem{AizenmanLebowitz88}
M.~Aizenman and J.~L. Lebowitz.
\newblock Metastability effects in bootstrap percolation.
\newblock {\em J. Phys. A}, 21(19):3801--3813, 1988.

\bibitem{Balister+16}
P.~Balister, B.~Bollob\'{a}s, M.~Przykucki, and P.~Smith.
\newblock Subcritical {$\mathcal{U}$}-bootstrap percolation models have
  non-trivial phase transitions.
\newblock {\em Trans. Amer. Math. Soc.}, 368(10):7385--7411, 2016.

\bibitem{BaloghBollobas06}
J.~Balogh and B.~Bollob\'{a}s.
\newblock Bootstrap percolation on the hypercube.
\newblock {\em Probab. Theory Related Fields}, 134(4):624--648, 2006.

\bibitem{Balogh+12}
J.~Balogh, B.~Bollob{\'a}s, H.~Duminil-Copin, and R.~Morris.
\newblock The sharp threshold for bootstrap percolation in all dimensions.
\newblock {\em Trans. Amer. Math. Soc.}, 364(5):2667--2701, 2012.

\bibitem{BaloghBollobasMorris09}
J.~Balogh, B.~Bollob{\'a}s, and R.~Morris.
\newblock Bootstrap percolation in three dimensions.
\newblock {\em Ann. Probab.}, 37(4):1329--1380, 2009.

\bibitem{BaloghBollobasMorris09-maj}
J.~Balogh, B.~Bollob{\'a}s, and R.~Morris.
\newblock Majority bootstrap percolation on the hypercube.
\newblock {\em Combin. Probab. Comput.}, 18(1-2):17--51, 2009.

\bibitem{BaloghBollobasMorris10}
J.~Balogh, B.~Bollob{\'a}s, and R.~Morris.
\newblock Bootstrap percolation in high dimensions.
\newblock {\em Combin. Probab. Comput.}, 19(5-6):643--692, 2010.

\bibitem{BaloghPete98}
J.~Balogh and G.~Pete.
\newblock Random disease on the square grid.
\newblock In {\em Proceedings of the {E}ighth {I}nternational {C}onference
  ``{R}andom {S}tructures and {A}lgorithms'' ({P}oznan, 1997)}, volume~13,
  pages 409--422, 1998.

\bibitem{Benevides+21+}
F.~Benevides, J.-C. Bermond, H.~Lesfari, and N.~Nisse.
\newblock Minimum lethal sets in grids and tori under 3-neighbour bootstrap
  percolation.
\newblock Technical report, 2021.
\newblock {HAL Research Report 03161419v4}.

\bibitem{BenevidesPrzykucki15}
F.~Benevides and M.~Przykucki.
\newblock Maximum percolation time in two-dimensional bootstrap percolation.
\newblock {\em SIAM J. Discrete Math.}, 29(1):224--251, 2015.

\bibitem{Bollobas06}
B.~Bollob\'{a}s.
\newblock {\em The art of mathematics}.
\newblock Cambridge University Press, New York, 2006.
\newblock Coffee time in Memphis.

\bibitem{Bollobas+17}
B.~Bollob\'{a}s, H.~Duminil-Copin, R.~Morris, and P.~Smith.
\newblock The sharp threshold for the {D}uarte model.
\newblock {\em Ann. Probab.}, 45(6B):4222--4272, 2017.

\bibitem{BollobasSmithUzzell15}
B.~Bollob\'{a}s, P.~Smith, and A.~Uzzell.
\newblock Monotone cellular automata in a random environment.
\newblock {\em Combin. Probab. Comput.}, 24(4):687--722, 2015.

\bibitem{CerfCirillo99}
R.~Cerf and E.~N.~M. Cirillo.
\newblock Finite size scaling in three-dimensional bootstrap percolation.
\newblock {\em Ann. Probab.}, 27(4):1837--1850, 1999.

\bibitem{CerfManzo02}
R.~Cerf and F.~Manzo.
\newblock The threshold regime of finite volume bootstrap percolation.
\newblock {\em Stochastic Process. Appl.}, 101(1):69--82, 2002.

\bibitem{ChalupaLeathReich79}
J.~Chalupa, P.~L. Leath, and G.~R Reich.
\newblock Bootstrap percolation on a {B}ethe lattice.
\newblock {\em J. Phys. C}, 12:L31--L35, 1979.

\bibitem{vanEnter87}
A.~C.~D. {\noopsort{Enter}{van Enter}}.
\newblock Proof of {S}traley's argument for bootstrap percolation.
\newblock {\em J. Statist. Phys.}, 48(3-4):943--945, 1987.

\bibitem{GravnerHolroyd08}
J.~Gravner and A.~E. Holroyd.
\newblock Slow convergence in bootstrap percolation.
\newblock {\em Ann. Appl. Probab.}, 18(3):909--928, 2008.

\bibitem{GravnerHolroydMorris12}
J.~Gravner, A.~E. Holroyd, and R.~Morris.
\newblock A sharper threshold for bootstrap percolation in two dimensions.
\newblock {\em Probab. Theory Related Fields}, 153(1-2):1--23, 2012.

\bibitem{GravnerHolroydSivakoff21}
J.~Gravner, A.~E. Holroyd, and D.~Sivakoff.
\newblock Polluted bootstrap percolation in three dimensions.
\newblock {\em Ann. Appl. Probab.}, 31(1):218--246, 2021.

\bibitem{HambardzumyanHatamiQian20}
L.~Hambardzumyan, H.~Hatami, and Y.~Qian.
\newblock Lower bounds for graph bootstrap percolation via properties of
  polynomials.
\newblock {\em J. Combin. Theory Ser. A}, 174:105253, 12, 2020.

\bibitem{HartarskyPhD}
I.~Hartarsky.
\newblock {\em Bootstrap percolation and kinetically constrained models:
  two-dimensional universality and beyond}.
\newblock PhD thesis, {CEREMADE}, {U}niversit\'e {P}aris {D}auphine, {PSL},
  Paris, France, 2022.

\bibitem{HartarskyMorris19}
I.~Hartarsky and R.~Morris.
\newblock The second term for two-neighbour bootstrap percolation in two
  dimensions.
\newblock {\em Trans. Amer. Math. Soc.}, 372(9):6465--6505, 2019.

\bibitem{Holroyd03}
A.~E. Holroyd.
\newblock Sharp metastability threshold for two-dimensional bootstrap
  percolation.
\newblock {\em Probab. Theory Related Fields}, 125(2):195--224, 2003.

\bibitem{Holroyd06}
A.~E. Holroyd.
\newblock The metastability threshold for modified bootstrap percolation in
  {$d$} dimensions.
\newblock {\em Electron. J. Probab.}, 11:no. 17, 418--433, 2006.

\bibitem{SchroederOEIS}
OEIS~Foundation Inc.
\newblock {The On-Line Encyclopedia of Integer Sequences}, 2022.
\newblock \url{https://oeis.org/A006318}.

\bibitem{LatinOEIS}
OEIS~Foundation Inc.
\newblock {The On-Line Encyclopedia of Integer Sequences}, 2022.
\newblock \url{https://oeis.org/A002860}.

\bibitem{MorrisonNoel18}
N.~Morrison and J.~A. Noel.
\newblock Extremal bounds for bootstrap percolation in the hypercube.
\newblock {\em J. Combin. Theory Ser. A}, 156:61--84, 2018.

\bibitem{PeteMSc}
G.~Pete.
\newblock Disease processes and bootstrap percolation.
\newblock Master's thesis, Bolyai Institute, J\'ozsef Attila University,
  Szeged, Hungary, 1997.

\bibitem{Pete97}
G.~Pete.
\newblock How to make the cube weedy?
\newblock {\em Polygon}, VII(1):69--80, 1997.
\newblock in Hungarian.

\bibitem{PrzykuckiShelton20}
M.~Przykucki and T.~Shelton.
\newblock Smallest percolating sets in bootstrap percolation on grids.
\newblock {\em Electron. J. Combin.}, 27(4):Paper No. 4.34, 11, 2020.

\bibitem{RomerMSc}
A.~E. Romer.
\newblock Tight bounds on 3-neighbor bootstrap percolation.
\newblock Master's thesis, University of Victoria, Victoria, Canada, 2022.

\bibitem{ShapiroStephens91}
L.~Shapiro and A.~B. Stephens.
\newblock Bootstrap percolation, the {S}chr\"{o}der numbers, and the
  {$N$}-kings problem.
\newblock {\em SIAM J. Discrete Math.}, 4(2):275--280, 1991.

\bibitem{Uzzell19}
A.~J. Uzzell.
\newblock An improved upper bound for bootstrap percolation in all dimensions.
\newblock {\em Combin. Probab. Comput.}, 28(6):936--960, 2019.

\bibitem{Winkler07}
P.~Winkler.
\newblock {\em Mathematical mind-benders}.
\newblock A K Peters, Ltd., Wellesley, MA, 2007.

\end{thebibliography}

\appendix

\section{Constructions in Small Grids}
\label{app:sporadic}
The purpose of this appendix is to gather up several ad hoc optimal constructions for small choices of parameters that we use in our proof of Theorem~\ref{thm:largeEnough}. Some of these were found by hand, while others were found using a heuristic computer search.

\subsection{Perfect Constructions}
\label{app:perfect}

\begin{const}
\label{const:225}
$(2,2,5)$ is perfect.
\end{const}

\begin{proof}
The construction below shows that $(2,2,5)$ is perfect, since $\frac{2\cdot 2+2\cdot 5+2\cdot 5}{3}=8$.

\begin{center}

\end{center}
\end{proof}

\begin{const}
\label{const:255}
$(2,5,5)$ is perfect.
\end{const}

\begin{proof}
The construction below shows that $(2,5,5)$ is perfect, since $\frac{2\cdot 5+2\cdot 5+5\cdot 5}{3}=15$.

\begin{center}
%
\end{center}
\end{proof}

\begin{const}
\label{const:558}
$(5,5,8)$ is perfect.
\end{const}

\begin{proof}
The construction below shows that $(5,5,8)$ is optimal, since $\frac{5\cdot 5+5\cdot 8+5\cdot 8}{3}=35$.
\begin{center}
%
\end{center}
\end{proof}

\begin{const}
\label{const:569}
$(5,6,9)$ is perfect.
\end{const}

\begin{proof}
The construction below shows that $(5,6,9)$ is optimal, since $\frac{5\cdot 6+5\cdot 9+6\cdot 9}{3}=43$.
\begin{center}
%
\end{center}
\end{proof}

\subsection{Optimal, But Not Perfect, Constructions}
\label{app:optimal}

\begin{const}
\label{const:344}
$(3,4,4)$ is optimal.
\end{const}

\begin{proof}
The construction below shows that $(3,4,4)$ is optimal, since $\left\lceil \frac{3\cdot 4+3\cdot 4+4\cdot 4}{3}\right\rceil=14$.
\begin{center}
%
\end{center}
\end{proof}

\begin{const}
\label{const:556}
$(5,5,6)$ is optimal.
\end{const}

\begin{proof}
The construction below shows that $(5,5,6)$ is optimal, since $\left\lceil \frac{5\cdot 5+5\cdot 6+5\cdot 6}{3}\right\rceil=29$.
\begin{center}
%
\end{center}
\end{proof}

\begin{const}
\label{const:557}
$(5,5,7)$ is optimal.
\end{const}

\begin{proof}
The construction below shows that $(5,5,7)$ is optimal, since $\left\lceil \frac{5\cdot 5+5\cdot 7+5\cdot 7}{3}\right\rceil=32$.
\begin{center}
%
\end{center}
\end{proof}

\begin{const}
\label{const:567}
$(5,6,7)$ is optimal.
\end{const}

\begin{proof}
The construction below shows that $(5,6,7)$ is optimal, since $\left\lceil \frac{5\cdot 6+5\cdot 7+6\cdot 7}{3}\right\rceil=36$.
\begin{center}
%
\end{center}
\end{proof}

\end{document}